\documentclass[twoside]{scrartcl}
\usepackage{scrlayer-scrpage}

\usepackage[utf8]{inputenc}
\usepackage[T1]{fontenc}
\usepackage[english]{babel}
\usepackage[unicode=true,plainpages=false]{hyperref}

\usepackage[left=2.5cm, right=3cm, top=2cm, bottom=3.5cm]{geometry}

\usepackage[shortlabels]{enumitem}

\usepackage{amsmath}
\usepackage{amssymb}
\usepackage{amsthm}
\usepackage{mathrsfs}
\usepackage{yhmath}
\usepackage{wasysym}
\usepackage{mathtools}

\usepackage[capitalise]{cleveref}

\usepackage{caption}
\usepackage{subcaption}

\usepackage{graphicx}
\usepackage{pgf,tikz}
\usepackage{tikz-cd}
\usepackage{rotating}

\usepackage{lmodern} 
\usepackage{tabularx} 
\usepackage[draft=false,babel,kerning=true,spacing=true]{microtype} 

\allowdisplaybreaks

\let\cprod\times

\makeatletter
\newcommand*{\relrelbarsep}{.386ex}
\newcommand*{\relrelbar}{%
  \mathrel{%
    \mathpalette\@relrelbar\relrelbarsep
  }%
}
\newcommand*{\@relrelbar}[2]{%
  \raise#2\hbox to 0pt{$\m@th#1\relbar$\hss}%
  \lower#2\hbox{$\m@th#1\relbar$}%
}
\providecommand*{\rightrightarrowsfill@}{%
  \arrowfill@\relrelbar\relrelbar\rightrightarrows
}
\providecommand*{\leftleftarrowsfill@}{%
  \arrowfill@\leftleftarrows\relrelbar\relrelbar
}
\providecommand*{\xrightrightarrows}[2][]{%
  \ext@arrow 0359\rightrightarrowsfill@{#1}{#2}%
}
\providecommand*{\xleftleftarrows}[2][]{%
  \ext@arrow 3095\leftleftarrowsfill@{#1}{#2}%
}
\makeatother

\renewcommand{\times}{\cdot}
\DeclarePairedDelimiter{\abs}{\lvert}{\rvert}

\DeclareMathOperator{\coeq}{coeq}

\newcommand{\isom}{\cong}

\newcommand{\from}{\mathbin{\leftarrow}}
\newcommand{\xto}[1]{\mathbin{\xrightarrow{#1}}} 
\newcommand{\xfrom}[1]{\mathbin{\xleftarrow{#1}}} 
\newcommand{\isoto}{\xto\sim}
\newcommand{\isofrom}{\xfrom\sim}
\newcommand{\longisoto}{\xto{\ \sim\ }}

\DeclareMathOperator{\Hom}{Hom}

\newcommand{\setst}{\mathrel{|}}
\newcommand{\isect}{\mathbin{\cap}}

\newcommand{\bigunion}{\bigcup}
\newcommand{\dunion}{\mathbin{\sqcup}}
\newcommand{\bigdunion}{\bigsqcup}

\renewcommand{\emptyset}{\varnothing}

\renewcommand{\subset}{\subseteq}
\renewcommand{\supset}{\supseteq}

\newcommand{\comp}{\mathbin{\circ}}
\newcommand{\id}[1]{\mathrm{id}_{#1}}
\newcommand{\pr}[1]{\mathrm{pr}_{#1}}
\newcommand{\restrict}[2]{{#1}|_{#2}}
\newcommand{\longto}{\longrightarrow}
\newcommand{\injto}{\mathrel{\hookrightarrow}}
\newcommand{\surjto}{\mathrel{\twoheadrightarrow}}

\newcommand{\dsum}{\oplus}

\newcommand{\Z}{\mathbb{Z}}
\newcommand{\N}{\mathbb{N}}
\newcommand{\Fld}{\mathbb{F}}
\newcommand{\Q}{\mathbb{Q}}
\newcommand{\R}{\mathbb{R}}
\newcommand{\Cpx}{\mathbb{C}}

\newcommand{\divides}{\mathbin{|}}

\newcommand{\GL}{\mathrm{GL}}

\newcommand{\units}{^\cprod}
\newcommand{\tensor}{\otimes}

\newcommand{\tclos}[1]{\overline{#1}} 

\DeclareMathOperator{\cts}{\mathcal C}

\newcommand{\topos}{\sim}

\newcommand{\catmod}[1]{\text{${#1}$-$\mathrm{Mod}$}}

\newcommand{\catlocsys}[2]{\mathrm{Loc}_{#1}(#2)}
\newcommand{\catlocsysint}[2]{\mathrm{ILoc}_{#1}(#2)}

\newcommand{\catmodproptriv}[1]{\mathrm{Mod}^{pt}(#1)}

\newcommand{\catwfintype}[1]{\mathrm{wFinType}(#1)}

\newcommand{\catshs}[1]{\operatorname{Sh}(#1)}

\DeclareMathOperator{\IHom}{\mathscr{H}\kern -3pt \mathit{om}}
\DeclareMathOperator{\IExt}{\mathscr{E}\kern -2pt \mathit{xt}}

\newcommand{\an}{{\mathrm{an}}} 
\newcommand{\et}{{\mathrm{et}}}
\newcommand{\fet}{{\mathrm{fet}}}
\newcommand{\proet}{{\mathrm{proet}}} 
\newcommand{\qproet}{{\mathrm{qproet}}} 
\newcommand{\vsite}{{\mathrm{v}}} 
\newcommand{\qcqs}{{\mathrm{qcqs}}} 
\newcommand{\aff}{{\mathrm{aff}}} 

\newcommand{\ri}{\mathcal O} 
\newcommand{\mm}{\mathfrak m} 

\DeclareMathOperator{\Spa}{Spa}
\DeclareMathOperator{\Spd}{Spd}

\newcommand{\Modplus}{{\mathcal Mod}^{pt}(\check{O}_X^+)}
\newcommand{\Mod}{{\mathcal Mod}^{pt}(\check{O}_X)}
\newcommand{\ModY}{{\mathcal Mod}^{pt}(\check{O}_Y)}

\newcommand{\Perf}{\mathrm{Perf}}

\newcommand{\catbdl}[1]{\mathcal B(#1)}

\newcommand{\supn}{{\mathrm{sup}}}

\theoremstyle{plain}
\newtheorem{theorem}{Theorem}[section]
\newtheorem{theorem*}{Theorem}
\newtheorem{proposition}[theorem]{Proposition}
\newtheorem{proposition*}[theorem*]{Proposition}
\newtheorem{corollary}[theorem]{Corollary}
\newtheorem{lemma}[theorem]{Lemma}

\theoremstyle{definition}
\newtheorem{definition}[theorem]{Definition}
\newtheorem{definition*}[theorem*]{Definition}
\newtheorem{example}[theorem]{Example}

\newtheorem{remark}[theorem]{Remark}

\newtheorem{hypothesis*}[theorem*]{Hypothesis}

\numberwithin{equation}{theorem}
\numberwithin{figure}{subsection}
\numberwithin{table}{subsection}

\newlist{thmenum}{enumerate}{1}
\setlist[thmenum]{label=(\roman*), ref=\thetheorem.(\roman*)}
\crefalias{thmenumi}{theorem}

\newlist{propenum}{enumerate}{1}
\setlist[propenum]{label=(\roman*), ref=\theproposition.(\roman*)}
\crefalias{propenumi}{proposition}

\newlist{corenum}{enumerate}{1}
\setlist[corenum]{label=(\roman*), ref=\thecorollary.(\roman*)}
\crefalias{corenumi}{corollary}

\newlist{lemenum}{enumerate}{1}
\setlist[lemenum]{label=(\roman*), ref=\thelemma.(\roman*)}
\crefalias{lemenumi}{lemma}

\newlist{exampleenum}{enumerate}{1}
\setlist[exampleenum]{label=(\alph*), ref=\theexamples.(\alph*)}
\crefalias{exampleenumi}{example}

\newlist{remarksenum}{enumerate}{1}
\setlist[remarksenum]{label=(\roman*), ref=\theremarks.(\roman*)}
\crefalias{remarksenumi}{remark}

\newlist{defenum}{enumerate}{1}
\setlist[defenum]{label=(\alph*), ref=\thedefinition.(\alph*)}
\crefalias{defenumi}{definition}

\title{Local systems on diamonds and $p$-adic vector bundles}
\author{Lucas Mann and Annette Werner}

\begin{document}
\maketitle

\begin{abstract}
We use Scholze's framework of diamonds to gain new insights in correspondences between $p$-adic vector bundles and local systems. Such correspondences arise in the context of $p$-adic Simpson theory in the case of vanishing Higgs fields. In the present paper we provide a detailed analysis of  local systems on diamonds for the \'etale, pro-\'etale, and the $v$-topology, and study the structure sheaves for all three topologies in question.  Applied to proper adic spaces of finite type over $\mathbb{C}_p$ this enables us to prove a category equivalence between $\mathbb{C}_p$-local systems with integral models, and modules under the $v$-structure sheaf which modulo each $p^n$ can be trivialized on a proper cover. The flexibility of the $v$-topology together with a descent result on integral models of local systems allows us to prove that the trivializability condition in the module category may be checked on any normal proper cover. This result leads to an extension of the parallel transport theory by Deninger and the second author to vector bundles with numerically flat reduction on a proper normal cover. \textbf{2020 MSC:} 14G45, 14G22, 11G25.
\end{abstract}

\tableofcontents

\section*{Introduction} 
\label{sec:intro}
\addcontentsline{toc}{section}{Introduction}

We use Scholze's theory of diamonds to gain new insights in correspondences between vector bundles and local systems which arise in the context of $p$-adic Simpson theory. Our results are motivated by  \cite{deninger-werner-parallel-transport-varieties}, where Deninger and the second author define a category of vector bundles $\mathcal{B}^s_{V_{\mathbb{C}_p}}$ with numerically flat reduction on every proper, smooth variety $V$  over $\overline{\mathbb{Q}}_p$.  By $\Pi_1(V_{\mathbb{C}_p})$ we denote the \'etale fundamental groupoid. One of the main results of  \cite{deninger-werner-parallel-transport-varieties} states that for every vector bundle in $\mathcal{B}^s_{V_{\mathbb{C}_p}}$ there exists a continuous functor  $\rho_E: \Pi_1(V_{\mathbb{C}_p}) \rightarrow \mathrm{Vec_{\mathbb{C}_p}}$  of parallel transport along \'etale paths. The association $E \mapsto \rho_E$ is functorial with respect to morphisms of smooth, proper varieties and compatible with several natural operations on vector bundles. This generalizes previous work on curves (\cite{deninger-werner-parallel-transport-1}, \cite{deninger-werner-parallel-transport-2}) which can be seen as a $p$-adic analog of the classical Narasimhan-Seshadri correspondence on Riemann surfaces. In the complex case, this theory is the special case of Simpson's correspondence \cite{S} in the case of vanishing Higgs field. There exists also a $p$-adic version of Simpson's correspondence developed by Faltings (\cite{Fa2}, \cite{fa3}). A detailed and systematic treatment is provided by \cite{agt}. If $X$ is a curve, this correspondence relates Higgs bundles and generalized representations. It is shown in \cite{xu} that the above mentioned parallel transport on curves provides an inverse to Faltings' functor for a category of of suitable vector bundles with trivial Higgs field. 

Recently Würthen  \cite{matti-rigid-DW-functor} has generalized the results by Deninger and the second author to proper seminormal rigid analytic varieties using the framework of pro-étale topology. Working with locally free sheaves on the pro-\'etale site provides a common framework for vector bundles and limits of \'etale local systems. In particular it is useful to get rid of cumbersome arguments with models of schemes. In the present paper we push this approach even further with the framework of Scholze's diamonds \cite{etale-cohomology-of-diamonds}. Here the very general $v$-topology allows us to deal with arbitrary proper covers. To be more precise, we first look at the category of local systems of $\ri_{\mathbb{C}_p}$-modules for the $v$-topology on diamonds. It turns out that by replacing the $v$-topology by the quasi-pro-étale topology, we obtain an equivalence of categories. This works over a large class of coefficient rings $\Lambda$, including all complete Huber rings and colimits of those. For \emph{discrete} rings it is also equivalent to consider the étale  topology, but, as is well known, this fails for non-discrete rings such as $\ri_{\mathbb{C}_p}$. 

Let $X$ be a proper adic space of finite type over $\mathbb{C}_p$. 
We define structure sheaves $\check\ri^+_X$ and $\check\ri_X$ on the $v$-site of the associated diamond (which we denote by $X_\vsite$),  and consider $\check\ri_X^+$-modules. Denote by  $\Modplus$ the category of all $\check\ri^+_X$-modules $\mathcal E$ on $X_\vsite$  with properly trivializable reduction modulo all $p^n$ consisting of all $\mathcal E$ such that for all $n$ the sheaf $\mathcal E / p^n$ pulls back to a free $\check\ri^+_Y / p^n$-module of finite rank on a proper cover $Y$ of $X$. Then we construct a functor $\Delta^+$ from  $\Modplus$ to the category $\catlocsys{\ri_{\mathbb{C}_p}}X$ of local systems of $\ri_{\mathbb{C}_p}$-modules on $X$. One of our main results is that $\Delta^+$ is an equivalence of categories. It is compatible with several natural operations on vector bundles.  Moreover, for every $\mathcal E$  in $\Modplus$ there is a natural isomorphism
\begin{align*}
	\Delta^+(\mathcal E) \tensor_\ri \check\ri^+_X = \mathcal E.
\end{align*} 
Inverting $p$, we say that an $\check\ri_X$-module $E$ has properly trivializable reduction modulo all $p^n$  if there exists a subsheaf $\mathcal E \subset E$ contained in $\Modplus$ such that $E = \mathcal E[p^{-1}]$. The category of $\check\ri_X$-modules with properly trivializable reduction modulo all $p^n$ is denoted by $\Mod$.  Setting $\Delta(E) = \Delta^+(\mathcal E)[1/p]$  defines a functor from $ \Mod $ to $\catlocsysint{\Cpx_p}X$, where $\catlocsysint{\Cpx_p}X$ denotes the category of $\mathbb{C}_p$-local systems on $X$ which are derived from an $\ri_{\mathbb{C}_p}$-local system by inverting $p$. 

One of the main theorems in the present paper is the following (see \cref{rslt:DW-for-OX}):

\begin{theorem}\label{introduction-thm-Delta}
\begin{thmenum}
	\item The functor $ \Delta \colon \Mod \longisoto \catlocsysint{\Cpx_p}X $ is an equivalence of categories, which is compatible with direct sums, tensor products, internal homs and exterior powers.
	\item For every $E \in \Mod$ there is a natural isomorphism $\Delta(E) \tensor_{\Cpx_p} \check\ri_X = E$.
	\item Let $f\colon Y \to X$ be a morphism of finite-type proper adic spaces over $\Cpx_p$ and let $E \in \Mod$. Then $f^*E \in \ModY$ and $\Delta(f^*E) = f^*\Delta(E)$.
	\item Let $f\colon Y \to X$ be a surjective morphism of finite type proper adic spaces over $\Cpx_p$ and assume that $X$ is normal. Let $E$ be an $\check\ri_X$-module. If $f^*E \in \catmodproptriv{\check\ri_Y}$ then $E \in \catmodproptriv{\check\ri_X}$. In other words, the property of having properly trivializable reduction descends along $f$.
\end{thmenum}
\end{theorem}

Part (iv) of the previous theorem is a powerful descent result relying on the strength of the $v$-topology. As will be explained below, the algebraic theory in the framework of \cite{deninger-werner-parallel-transport-varieties} provides strong motivation for such a descent result along proper covers to hold.
By the category equivalence from (i) the descent datum of $E$ along $f$ induces a descent datum of the associated $\Cpx_p$-local system. In order to descend the integral model, we need the following theorem (\cref{rslt:descent-of-integral-model-along-morphism-of-proper-spaces-with-loc-irr-target}):

\begin{theorem} \label{intro-descent}
Let $X, Y $ be adic spaces weakly of finite type over $\mathbb{C}_p$, let $f\colon Y \to X$ be a surjective proper morphism of finite type, and assume that $X$ is normal. If $L$ is a $\mathbb{C}_p$-local system on $X$ such that $\restrict LY$ has an integral model (i.e. there exists a $ \ri_{\mathbb{C}_p}$-local system on $Y$ inducing $\restrict LY$ by inverting $p$), then $L$ also has an integral model on $X$.
\end{theorem}

The  proof of this theorem relies on a Stein factorization result, which separates the problem into the case of finite morphisms and the case of morphisms with geometrically connected fibers. Both cases can be solved by using w-contractible covers in order to reduce to a topological question on profinite sets; the former case additionally requires some elementary results on adic spaces which are  collected in \cite{simpson-appendix}.

On a proper, connected seminormal adic space $X$ of finite type over $\mathbb{C}_p$ we define furthermore the category $\catbdl{X}$ as the full subcategory of adic vector bundles (i.e. locally free $\ri_ {X}$-sheaves) $E$ on $X$ such that the pullback of $E$ to the $v$-site is an object of $\Mod$. For every $E$ in $\catbdl{X}$ we therefore get a local system in $\catlocsysint{\Cpx_p}X$ via the functor $\Delta$ on $\Mod$. By means of the category equivalence $\Delta$, we  show the following result (\cref{rslt:comparison-to-Matti}).

\begin{corollary}
Every vector bundle $E$ in $\catbdl{X}$ gives rise to a continuous functor $\rho_E$ from the fundamental groupoid $\Pi_1(X)$ to the category of finite free $\mathbb{C}_p$-vector spaces.
\end{corollary}

If $E$ is has a suitable integral model, our construction coincides with the one given by W\"urthen in \cite[Section 4.3]{matti-rigid-DW-functor}. We also prove that the category $\catbdl X$ is abelian (\cref{rslt:category-abelian}).

Let us now return to  the algebraic setting of a  proper, smooth variety $V$  over $\overline{\mathbb{Q}}_p$. Note that by Langer \cite[Proposition 5.5]{Langer-Sfund}, the category of numerically flat vector bundles on $V$ is an abelian tensor category. In view of the results \cite{deninger-werner-parallel-transport-varieties} one might hope that every numerically flat vector bundle on $V_{\mathbb{C}_p}$ gives rise to $p$-adic parallel transport, i.e. that it is contained in  $\mathcal{B}^s_{V_{\mathbb{C}_p}}$.  Since for any surjective morphism $g: Y \rightarrow X$  of connected proper schemes over a field, a bundle $E$ is numerically flat on $X$ if and only if $g^\ast E$ is numerically flat on $Y$, a first test for this conjecture is the extension of the results by Deninger and the second author to vector bundles which admit numerically flat reduction only after pullback by a surjective morphism of proper, smooth varieties. This is achieved by the following theorem  (\cref{rslt:algebraic-case}) of the present paper: 

\begin{theorem} \label{rslt:introduction1}
Let $E$ be a vector bundle on $V_{\mathbb{C}_p}$ such that there exists a normal proper variety $W$ over $\overline{\mathbb{Q}}_p$ and a  surjective $\overline{\mathbb{Q}}_p$-morphism $f: W \rightarrow V$ such that $f^\ast E$ lies in the category $\mathcal{B}^s_{W_{\mathbb{C}_p}}$, i.e. $f^\ast E$  has numerically flat reduction on $W_{\mathbb{C}_p}$.

Then there exists a continuous functor  $\rho_E: \Pi_1(V_{\mathbb{C}_p}) \rightarrow \mathrm{Vec_{\mathbb{C}_p}}$ such that the identification of fibers $(f^\ast E)_x \simeq E_{f(x)}$ induces a natural isomorphism $\rho_{f^\ast E} \simeq \rho_E \circ f_\ast$. The association $E \mapsto \rho_E$ is compatible with direct sums, tensor products, internal homs and exterior powers. Moreover it is functorial with respect to morphisms of varieties and coincides with the functor from \cite{deninger-werner-parallel-transport-varieties} if $E$ lies in $\mathcal{B}^s_{V_{\mathbb{C}_p}}$.
\end{theorem}

For finite coverings of curves a similar result was proven previously by comparing topological and étale fundamental groups and using analytic methods on Riemann surfaces \cite{deninger-werner-parallel-transport-2}. The techniques in the present paper lead to a new proof in the curve case which does not rely on complex analytic methods. The freedom of analyzing the reduction behaviour after a proper pullback is quite useful as was seen in the investigation \cite{hawe}  of  semistable syzygy bundles on curves. By  some quite detailed computations it is shown in this paper  that certain syzygy bundles have numerically flat reduction after pullback to a suitable finite cover. In $p$-adic Simpson correspondence it is a deep open problem to determine the category of Higgs bundles corresponding to actual $p$-adic representations in the $p$-adic Simpson correspondence. In the case of trivial Higgs field we hope that the results of the present paper will prove useful to shed light on this problem. 
 
The structure of the present paper is as follows. In the first section we collect some basic facts on perfectoid spaces and diamonds. In particular, we introduce $w$-contractible spaces and show that they form a basis of the quasi-pro-\'etale site on a diamond. They have the useful property that every pro-\'etale cover admits a section. The second section deals with the structure sheaves on the quasi-pro-\'etale site $X_\qproet$ and the $v$-site $X_\mathrm{v}$ of a diamond $X$ over $\mathrm{Spd }\mathbb{Z}_p$. By  definition of $\mathrm{Spd }\mathbb{Z}_p$, every perfectoid characteristic $p$ space $Y$ over $X$ gives rise to an element in $\mathrm{Spd }\mathbb{Z}_p(Y)$, hence an untilt $Y^\sharp$ of $Y$. The association $\hat \ri^+_X(Y) = \ri_{Y^\sharp}^+(Y^\sharp)$ gives rise to a sheaf $\hat \ri^+_X$ on $X_\qproet$. Inverting $p$, we define $\hat\ri_X$. Similarly, we define the sheaf $\check\ri_X$ on $X_\mathrm{v}$, which has the property that $\lambda_\ast \check\ri_X = \hat \ri_X$ for the natural morphism $\lambda: X_\mathrm{v} \rightarrow X_\qproet$. Let  furthermore $\nu: X_\qproet \rightarrow X_\et$ be the natural projection. For an adic space $Z$  over $\mathbb{Q}_p$ we can identify the \'etale site over $Z$ with the \'etale site of its associated diamond $X$. It turns out (\cref{rslt:completed-structure-sheaves-coincide}) that $\hat \ri_X^+$ can be identified with the pro-\'etale completion $ \varprojlim_n \nu^{-1} \ri_X^+/p^n $ on $X_\qproet$, where $\ri_X^+= \ri_{X_\et}^+$.

The third section introduces $\Lambda$-local systems for the \'etale, quasi-pro-\'etale and $v$-topology on a diamond for condensed (e.g. topological) rings $\Lambda$. In \cref{rslt:qproet-locsys-equiv-v-locsys-complete-huber-rings} we show that pullback with respect to $\lambda: X_\mathrm{v} \rightarrow X_\qproet$ induces an equivalence of the respective categories of $\Lambda$-local systems,  under mild assumptions on $\Lambda$. If $\Lambda$ is a discrete ring, then a similar result holds for $\nu: X_\qproet \rightarrow X_\et$. We also show that for a finite discrete ring $\Lambda$ any $\Lambda$-local system on a spatial diamond can be trivialized on a finite \'etale cover of $X$. This implies (\cref{rslt:profinet-cover}) that on a connected spatial diamond we can trivialize any $\ri_E$-local system for a completed algebraic extension $E$ of $\mathbb{Q}_p$ on a connected quasi-profinite-\'etale cover. We also define (\cref{def:rep-integral-model}) the category $\catlocsysint EX$ of local systems with integral models (which are local systems with respect to the ring of integers) and prove (\cref{rslt:ILoc-is-abelian}) that this category is abelian. In section four, we compute global sections of $\ri_{\mathbb{C}_p}$-local systems on proper adic spaces of finite type over $\mathbb{C}_p$. We start with generalizing Scholze's fundamental comparison theorem from \cite{relative-p-adic-hodge-2} to local systems of  $\ri$- or $\ri / \varpi$-modules for any pseudo-uniformizer $\varpi$ (\cref{rslt:O-E-cohomology-on-proper-space}). We can then pass from almost isomorphisms to isomorphisms in our setting (\cref{rslt:global-sections-on-proper-space}). Section 5 is devoted to the descent result \cref{intro-descent} for integral models of local systems along surjective proper morphisms.  It relies on some fundamental results on adic spaces provided by \cite{simpson-appendix}. In section 6 we prove our correspondence results between local systems and locally free sheaves. First we consider sheaves of $\check\ri_X^+/p^n$-modules on a proper adic space $X$ of finite type over $\mathbb{C}_p$, which are properly trivializable, i.e. which are free $\check\ri_X^+/p^n$-modules of finite rank after pullback to a proper cover of $X$. Using the comparison results of section 4, we define a functor $\Delta^+_n$ from the category of properly trivializable $\check\ri_X^+/p^n$-modules to the category of $\ri_X/p^n$-local systems on $X$, which induces an equivalence of categories (\cref{rslt:DW-n-is-equivalence}). Passing to the limit we obtain the equivalence $\Delta^+$, and after inverting $p$, we can show \cref{introduction-thm-Delta}. Section 7 applies these results to vector bundles on proper, seminormal adic spaces over $\mathbb{C}_p$, and section 8 deals with the algebraic result \cref{rslt:introduction1}.

\textbf{Notation and conventions.} Throughout, fix a prime $p$, let $\Cpx_p$ denote the completed algebraic closure of $\Q_p$ and let $\ri \subset \Cpx_p$ be the ring of integers. All sheaves and cohomology are considered on the v-topology, unless explicitly stated otherwise.

\textbf{Acknowledgements.} We are very grateful to Peter Scholze for many enlightening comments. This work was inspired by discussions during the Simons Symposion on $p$-adic Hodge theory in 2019. The second author is grateful to the Simons foundation for financial support of this meeting and to the organizers, Bhargav Bhatt and Martin Olsson. We also thank Bhargav Bhatt for suggesting to look into the $v$-topology. Moreover, we would like to thank Christopher Deninger for his comments, Torsten Wedhorn for useful hints on adic spaces and Matti W\"urthen for many helpful conversations.

\section{Some Background on Diamonds}

The theory of diamonds is introduced in \cite{etale-cohomology-of-diamonds} and can be seen as a rigid-analytic analogue of the theory of algebraic spaces in algebraic geometry. Let $\Perf$ be the category of ($\kappa$-small, see \cite[Convention 4.2]{etale-cohomology-of-diamonds}) char-$p$ perfectoid spaces. Then a \emph{diamond} $X$ is a sheaf on the big pro-étale site on $\Perf$ such that $X$ can be written as the quotient $X = Y/R$ of some perfectoid space $Y$ by a pro-étale equivalence relation $R \subset Y \cprod Y$ (cf. \cite[Definition 11.1]{etale-cohomology-of-diamonds}). Every diamond $X$ possesses a natural ``underlying'' topological space $\abs X = \abs Y/\abs R$. In case this topology is (locally) spectral we call $X$ a \emph{(locally) spatial} diamond. This definition is necessary because the topology of a diamond can in general be quite far from being spectral, for example every compact Hausdorff space can be realised as the topology of some diamond (see \cite[Example 11.16]{etale-cohomology-of-diamonds}). A diamond $X$ comes equipped with the following sites (cf. \cite[Definition 14.1]{etale-cohomology-of-diamonds}):
\begin{itemize}
	\item The \emph{analytic site} $X_\an$. This is just the site associated to the topological space $\abs X$.
	\item The \emph{étale site} $X_\et$, if $X$ is locally spatial.
	\item The \emph{quasi-pro-étale site} $X_\qproet$ defined in \cite[Definition 14.1]{etale-cohomology-of-diamonds}. This is an analogue of the pro-étale site of schemes (the term ``quasi'' refers to the fact that in general not every quasi-pro-étale map $Y \to X$ is locally an inverse limit of étale maps, which is similar to the case of weakly étale maps of schemes).
\end{itemize}
Here are two very useful properties of diamonds:
\begin{enumerate}[(a)]
	\item There is a functor $X \mapsto X^\diamond$ from analytic adic spaces over $\Z_p$ to diamonds. This functor preserves the analytic and the étale site, i.e $\abs X = \abs{X^\diamond}$ and $X_\et = X^\diamond_\et$ (see \cite[Lemma 15.6]{etale-cohomology-of-diamonds}). Moreover this functor induces, for any non-archimedean field $K$ over $\Q_p$, a fully faithful functor
	\begin{align}
		\{ \text{seminormal\footnotemark{} rigid-analytic spaces over $K$} \} \injto \{ \text{diamonds over $\Spd K$} \} \label{eq:fully-faithful-functor-from-seminormal-to-diamond}
	\end{align}
	\footnotetext{Seminormality of adic spaces is defined in \cite[\S3.7]{relative-p-adic-hodge-2} and is a rather mild hypothesis; in particular it is true for all normal spaces.}
	(see \cite[Proposition 10.2.3]{scholze-berkeley-lectures}).

	\item Diamonds are not tied to any base field. In fact, a diamond itself does not even know whether it lives in characteristic $0$ or $p$. This is best illustrated in the case of perfectoid spaces: If $X$ is any perfectoid space (in either characteristic) then the associated diamond $X^\diamond$ combines all untilts of the tilt of $X$ into one object. The reason this works is because the tilt and untilt functors preserve many geometric properties, including the topology and the étale site. Thus by using diamonds it is very easy to pass between different characteristics.

	On the other hand it is often important to keep the information of which untilt one started with: This can be achieved by remembering the ``structure morphism'' $X^\diamond \to \Spd\Z_p$ obtained from the unique morphism $X \to \Z_p$ (cf. \cref{rmk:morphism-to-Z-p-important-to-define-structure-sheaf} below). Note that in particular it is crucial in \cref{eq:fully-faithful-functor-from-seminormal-to-diamond} to have the structure morphism to $K$ as an additional datum on both sides.
\end{enumerate}
Another intriguing property of diamonds (and also the main reason why we want to use them in the present paper) is the existence of the v-site. The v-site on $\Perf$ is the site whose coverings are given by collections $\{f_i\colon Y_i \to X\}_{i\in I}$ such that for every quasi-compact open subset $U \subset X$ there is a finite subset $J \subset I$ and quasi-compact open subsets $V_j \subset Y_j$ for $j \in J$ with the property $\bigunion_{j\in J} f(V_j) = U$ (see \cite[Definition 8.1.(iii)]{etale-cohomology-of-diamonds}). This definition is very similar to the definition of the fpqc site for schemes, but without any flatness condition! The remarkable fact is that the v-site still possesses very nice properties, especially in comparison with the pro-étale site.

With the v-site at hand, one defines a \emph{small v-sheaf} $X$ to be a sheaf of sets on the v-site of $\Perf$ such that there exists a surjective map $Y \to X$ from a perfectoid space $Y$ (see \cite[Definition 12.1]{etale-cohomology-of-diamonds}). It is a non-trivial fact that every diamond is a small v-sheaf (the hard part is to show that a diamond is a sheaf for the v-site and not just the big pro-étale site, see \cite[Proposition 11.9]{etale-cohomology-of-diamonds}). While a general small v-sheaf $X$ does not possess a quasi-pro-étale site (there is a notion of quasi-pro-étale morphism of small v-sheaves, but the associated site behaves badly in general), we can instead introduce:
\begin{itemize}
	\item The \emph{v-site} $X_\vsite$: The objects of $X_\vsite$ are all maps $Y \to X$ from small v-sheaves $Y$, with coverings given by jointly surjective maps.
\end{itemize}
For every diamond $X$ there is an obvious projection of sites
\begin{align*}
	\lambda_X\colon X_\vsite \to X_\qproet,
\end{align*}
and if $X$ is locally spatial (so that $X_\et$ is defined) there is a projection of sites
\begin{align*}
	\nu_X\colon X_\qproet \to X_\et.
\end{align*}
We will usually drop the index and simply write $\lambda = \lambda_X$ and $\nu = \nu_X$. The projections $\lambda$ and $\nu$ satisfy $\lambda_* \lambda^{-1} = \id{}$ and $\nu_* \nu^{-1} = \id{}$ (see \cite[Proposition 14.7 and 14.8]{etale-cohomology-of-diamonds}). This implies that abelian sheaves on $X_\et$ embed fully faithfully into abelian sheaves on $X_\qproet$, abelian sheaves on $X_\qproet$ embed fully faithfully into abelian sheaves on $X_\vsite$. It is even true that the former embedding and the composition of both embeddings preserve cohomology.

A powerful technique for studying sheaves on any site is to reduce a given problem to a local problem and then verify it on a nice basis of the site. It is thus helpful to know useful bases of the sites we work with. Both the v-site of a small v-sheaf and the quasi-pro-étale site of a diamond possess bases consisting of objects of each of the following types (for precise definitions see \cite[\S7]{etale-cohomology-of-diamonds}):
\begin{itemize}
	\item Perfectoid spaces, or even affinoid perfectoid spaces.
	\item \emph{Totally disconnected spaces}: These are special types of affinoid perfectoid spaces which are weakly contractible for the analytic topology. In particular, every analytic sheaf on a totally disconnected space is acyclic. A special case are \emph{w-local spaces}, whose analogue in the theory for schemes is used frequently. Totally disconnected space are of a particular simple form: Each connected component is of the form $\Spa(K, K^+)$ for some perfectoid field $K$ with an open and bounded valuation subring $K^+ \subset K$.
	\item \emph{Strictly totally disconnected spaces}: These are special cases of totally disconnected spaces with the additional property of being weakly contractible for the étale site. Hence every étale sheaf on a strictly totally disconnected space is acyclic. Again there is the slightly stronger notion of \emph{w-strictly local spaces} (see \cite[Definition 7.17]{etale-cohomology-of-diamonds}), which has an important analogue in the theory for schemes. The connected components of a strictly totally disconnected space are of the form $\Spa(C, C^+)$, where now $C$ is an algebraically closed perfectoid field.
	\item \emph{w-Contractible Spaces}: These are special cases of strictly totally disconnected spaces and will be introduced below (see \cref{def:w-contractible-space}). They are weakly contractible for the quasi-pro-étale site, so that every quasi-pro-étale sheaf on a w-contractible space is acyclic.
\end{itemize}
Unlike in the theory of the pro-étale topology of schemes, for diamonds it is often enough to consider strictly totally disconnected spaces. However in the present paper we do have some applications for w-contractible spaces, so let us introduce them now. Recall the definition of extremally disconnected spaces in \cite[Definition 2.4.4]{proetale-topology}. 

\begin{definition} \label{def:w-contractible-space}
A \emph{w-contractible space} is a w-strictly local perfectoid space $X$ such that $\pi_0(X)$ is extremally disconnected.
\end{definition}

\begin{lemma} \label{rslt:w-contractible-basis}
Let $X$ be a diamond. Then the w-contractible spaces in $X_\qproet$ form a basis of $X_\qproet$. Moreover, if $U \in X_\qproet$ is w-contractible then it is quasi-pro-étale weakly contractible, i.e. every pro-étale cover of $U$ admits a section. In particular, for any sheaf $\mathcal F$ of abelian groups on $X_\qproet$ we have $H^i(U_\qproet, \mathcal F) = 0$ for $i > 0$, and  $X_\qproet$ is replete.

\end{lemma}
\begin{proof}
Let us first show that the w-contractible spaces form a basis of $X_\qproet$. Because strictly totally disconnected spaces form a basis of $X_\qproet$ (by definition of diamonds there is a cover of $X$ by perfectoid spaces, then use \cite[Lemma 7.18]{etale-cohomology-of-diamonds}) we can reduce to $X$ being strictly totally disconnected. By \cite[Proposition 7.12]{etale-cohomology-of-diamonds} there is an affinoid pro-étale cover of $X$ by a w-local space which by \cite[Lemma 7.19]{etale-cohomology-of-diamonds} is still strictly totally disconnected and hence w-strictly local. We thus reduce to the case that $X$ is w-strictly local.

By \cite[Example 2.4.6]{proetale-topology} there is a covering $T \to \pi_0(X)$ by an extremally disconnected space. Now we can argue similarly as in \cite[Lemma 2.2.8]{proetale-topology}. By \cite[Lemma 2.1.13]{proetale-topology}, we can write $T = \varprojlim_i T_i$ where each $T_i$ is an exterior disjoint union of a finite family of clopen subsets of $\pi_0(X)$. To each clopen subset of $\pi_0(X)$ we associate the clopen subset of $X$ obtained as the preimage under $X \to \pi_0(X)$. This way we can associate to each $T_i$ a w-strictly local space. Taking the limit, we obtain a w-strictly local space $\tilde X$, affinoid pro-étale over $X$, with $\pi_0(\tilde X) = T$. Thus $\tilde X$ is w-contractible and provides the desired cover of $X$.

It remains to show that every w-contractible space is pro-étale weakly contractible (then repleteness of $X_\qproet$ follows from \cite[Proposition 3.2.3]{proetale-topology}). Thus let $X$ be a w-contractible space and let $f\colon U \to X$ be a pro-étale cover. Then $U$ is a perfectoid space. By quasi-compactness of $X$ there exist finitely many affinoid open subsets $U_1, \dots, U_n \subset U$ such that $U' := \bigdunion_{i=1}^n U_i \to X$ is a (pro-étale) covering. It is then enough to construct a section of $U' \to X$, as we can compose this section with $U' \to U$ to get the desired section of $U \to X$. We can thus replace $U$ by $U'$ to assume that $U$ is affinoid perfectoid. Then $U$ is strictly totally disconnected by \cite[Lemma 7.19]{etale-cohomology-of-diamonds}. Let $X^c \subset X$ denote the subset of closed points and let $U_0 := f^{-1}(X^c)$. Since $X^c$ is closed in $X$ by w-locality of $X$, the preimage $U_0$ is closed in $U$. We will see below that every connected component of $U$ contains at most one point of $U_0$, which implies that $U_0 $ is in fact a compact Hausdorff space.  On the other hand, the natural map $X^c \to \pi_0(X)$ is a homeomorphism because it is a bijection of compact Hausdorff spaces. Therefore $X^c$ is extremaly disconnected and consequently the surjective map
$ f |_{U_0}\colon U_0 \to X^c$ admits a section $s_0\colon X^c \to U_0$.

By the structure of strictly totally disconnected spaces, every connected component of $U$ and of $X$ is of the form $\Spa(C, C^+)$ for some algebraically closed perfectoid field $C$ and some open and bounded valuation subring $C^+$. If $u \in U$ and $x \in X$ are closed points and $U_u = \Spa(C', C'^+) \subset U$ and $X_x = \Spa(C, C^+) \subset X$ are the connected components containing $u$ and $x$ respectively such that $f(U_u) \subset X_x$, then $f$ induces a pro-étale map $U_u \to X_x$ and consequently an isomorphism $C' \isom C$ under which $C'^+ \supset C^+$. If moreover $u \in U_0$ then we must have $C'^+ = C^+$; in this case the map $U_u \to X_x$ is an isomorphism and hence admits a unique section. Altogether this shows that $s_0\colon X^c \to U_0$ can be uniquely extended to a set-theoretic map $s\colon X \to U$ which respects connected components.

It remains to show that $s$ is a morphism of adic spaces. By \cite[Corollary 7.22]{etale-cohomology-of-diamonds} it is enough to show that $s$ is a spectral map. But note that $\pi_0(s(X)) \subset \pi_0(U)$ is closed, hence $s(X) \subset U$ is closed (because it is obtained as the preimage of $\pi_0(s(X))$ under the projection $U \surjto \pi_0(U)$). Therefore $s(X)$ is a spectral space and $\restrict f{s(X)}\colon s(X) \to X$ is a generalizing spectral map of spectral spaces. By \cite[Lemma 2.5]{etale-cohomology-of-diamonds} $\restrict f{s(X)}$ is a quotient map, hence a homeomorphism, so that $s$ is a homeomorphism onto its image and in particular continuous.
\end{proof}

\section{Structure Sheaves of Diamonds}

As is shown in \cite[Lemma 15.6]{etale-cohomology-of-diamonds}, there is a functor $X \to X^\diamond$ from analytic adic spaces over $\Z_p$ to diamonds which preserves the analytic site $\abs X = \abs{X^\diamond}$ and the étale site $X_\et = X^\diamond_\et$. This shows that diamonds can capture the topological properties of analytic adic spaces very well, but in order to study properties of classical rigid-analytic varieties intrinsically inside the diamond world, it is important to capture the structure sheaves $\ri_X$ and $\ri^+_X$ as well. Therefore, in the present section we introduce and study the structure sheaf $\check\ri_X$ and the integral structure sheaf $\check\ri^+_X$ on the v-site of a diamond (or even small v-sheaf) $X$ over $\Spd\Z_p$. We show that these sheaves coincide with the classical structure sheaves when $X$ comes from an adic space over $\Q_p$ via the diamond functor defined in \cite[Definition 15.5]{etale-cohomology-of-diamonds}, see \cref{rslt:comparison-of-adic-and-diamond-world} below.

To define the sheaves $\check\ri_X$ and $\check\ri^+_X$, note that by definition of small v-sheaves, the perfectoid spaces in $X_\vsite$ form a basis of the site. Hence to define a sheaf $\mathcal F$ on $X_\vsite$ it is sufficient to define its sections $\mathcal F(Y)$ on all perfectoid $Y \in X_\vsite$. We will work with the diamond $\Spd \Z_p$ defined in \cite[Lemma 15.1]{etale-cohomology-of-diamonds}.
\begin{definition}
Let $X$ be a small v-sheaf over $\Spd\Z_p$.
\begin{defenum}
	\item Let $Y \in X_\vsite$ be representable by a perfectoid space (of characteristic $p$). The morphism $Y \to \Spd\Z_p$, obtained as the composition $Y \to X \to \Spd\Z_p$, corresponds to an element of $(\Spd\Z_p)(Y)$. By definition of $\Spd\Z_p$, this element is an untilt of $Y$. We will denote it by $Y^\sharp$.

	\item The \emph{v-structure sheaf} $\check\ri_X$ of $X$ is the sheaf on $X_\vsite$ such that $\check\ri_X(Y) := \ri_{Y^\sharp}(Y^\sharp)$ for all perfectoid $Y \in X_\vsite$, where $\ri_{Y^\sharp}$ is the structure sheaf of the perfectoid adic space $Y^\sharp$. Similarly we define the \emph{integral v-structure sheaf} $\check\ri^+_X$ by $\check\ri^+_X(Y) = \ri^+_{Y^\sharp}(Y^\sharp)$ for perfectoid $Y \in X_\vsite$.

	\item \label{def:completed-structure-sheaf-on-diamond} Suppose $X$ is a diamond. Then the \emph{quasi-pro-étale structure sheaf} $\hat\ri_X$ of $X$ is the sheaf on $X_\qproet$ such that $\hat\ri_X(Y) := \ri_{Y^\sharp}(Y^\sharp)$ for all perfectoid $Y \in X_\qproet$. Similarly we define the \emph{integral quasi-pro-étale structure sheaf} $\hat\ri^+_X$ on $X_\qproet$.
\end{defenum}
\end{definition}

Our definition of the v-structure sheaf and integral v-structure sheaf on a small v-sheaf $X$ is based on the fact that the presheaves $Y \mapsto \ri_Y(Y)$ and $Y \mapsto \ri^+_Y(Y)$ on the v-site of all perfectoid spaces are sheaves, see \cite[Theorem 8.7]{etale-cohomology-of-diamonds}. If $X$ is a diamond over $\Spd\Z_p$ and $\lambda\colon X_\vsite \to X_\qproet$ the projection of sites, then
\begin{align}
	\hat\ri_X = \lambda_* \check\ri_X, \qquad \hat\ri^+_X = \lambda_* \check\ri^+_X. \label{eq:compat-of-structure-sheaves-v-qproet}
\end{align}
However, it is not true that $\check\ri^{(+)}_X$ can be obtained from $\hat\ri^{(+)}_X$ by the inverse image functor $\lambda^{-1}$, so $\check\ri^{(+)}_X$ carries more information than $\hat\ri^{(+)}_X$.

\begin{remark} \label{rmk:morphism-to-Z-p-important-to-define-structure-sheaf}
The definition of the v-structure sheaf $\check\ri_X$ above is highly dependent on the structure morphism $X \to \Spd\Z_p$, since when passing from adic spaces to diamonds, one loses the information which untilt one started with. More precisely, if $X$ is a perfectoid space then the diamond $X^\diamond$ embodies all untilts of $X^\flat$ at once. In particular we cannot retrieve $X$ from $X^\diamond$ -- this is only possible if we also track the structure morphism $X^\diamond \to \Spd\Z_p$ induced from the natural morphism $X \to \Z_p$. Note that if we are only interested in the topology and the étale site of $X$, then this subtelty does not matter, as all untilts share the same topology and étale site.
\end{remark}

In order to justify our definition of structure sheaves we will check its compatibility with the classical structure sheaf of adic spaces in case that the small v-sheaf $X$ comes from an adic space. To make that statement more precise, let us introduce the quasi-pro-étale site and the completed structure sheaf on an adic space:

\begin{definition}
Let $X$ be an adic space over $\Q_p$ and let $X^\diamond$ be the associated diamond according to \cite[Definition 15.5]{etale-cohomology-of-diamonds}. Note that  $X_\et = X^\diamond_\et$ by \cite[Lemma 15.6]{etale-cohomology-of-diamonds}.
\begin{defenum}
	\item The \emph{quasi-pro-étale site} of $X$ is the site $X_\qproet := X^\diamond_\qproet$. Similarly, the \emph{v-site} of $X$ is the site $X_\vsite := X^\diamond_\vsite$. We have natural projections of sites $\lambda\colon X_\vsite \to X_\qproet$ and $\nu\colon X_\qproet \to X_\et$.
	\item \label{def:completed-structure-sheaf-on-adic-space} Let $\nu^{-1}\ri_X$ and $\nu^{-1}\ri^+_X$ denote the inverse images of the étale structure sheaves $\ri_X$ and $\ri^+_X$ from $X_\et$ to $X_\qproet$. The \emph{completed (integral) structure sheaf} $\hat\ri_X$ (resp. $\hat\ri^+_X$) on $X_\qproet$ is defined as
	\begin{align*}
		\hat\ri^+_X := \varprojlim_n \nu^{-1}\ri^+_X/p^n, \qquad \hat\ri_X := \hat\ri^+_X[p^{-1}].
	\end{align*}
	\end{defenum}
\end{definition}

This definition is motivated by the analogous definition for the pro-étale site on an adic space in \cite[Definition 4.1]{rigid-p-adic-hodge}.
Given an adic space $X$ over $\Q_p$, we now have defined two completed structure sheaves on $X_\qproet$: On the one hand we have the sheaf $\hat\ri_X$ defined ``from bottom to top'' through the classical structure sheaf on $X_\et$. On the other hand we have the sheaf $\hat\ri_{X^\diamond}$ defined ``from top to bottom'' through the basis of perfectoid spaces inside $X_\qproet$. In the following we will show that both versions of the completed structure sheaf are the same. Before we come to that result, we need a better understanding of the site $X_\qproet$ in terms of the adic space $X$. We fix a cardinal $\kappa$ as in \cite[Convention 4.2]{etale-cohomology-of-diamonds} to evade set-theoretic issues.

\begin{definition} \label{def:pro-etale-morphism}
Let $f\colon Y \to X$ be a morphism of $\kappa$-small adic spaces over $\Q_p$, and assume that $X$ and $Y$ are either locally noetherian or perfectoid.
\begin{defenum}
	\item \label{def:pro-etale-morphism.affinoid} The map $f$ is \emph{affinoid pro-étale} if $Y = \Spa(S, S^+)$ and $X = \Spa(R, R^+)$ are affinoid and one can write $Y = \varprojlim_i Y_i \to X$ as a cofiltered limit of étale maps $Y_i \to X$ from affinoid adic spaces $Y_i = \Spa(S_i, S_i^+)$ along a $\kappa$-small index category $I$.

	\item The map $f$ is \emph{pro-étale} if for all $y \in Y$ there is an open neighbourhood $V \subset Y$ of $y$ and an open subset $U \subset X$ such that $f(V) \subset U$ and the restriction $\restrict fV\colon V \to U$ is affinoid pro-étale.
\end{defenum}
\end{definition}

In case that both $X$ and $Y$ are perfectoid spaces, our \cref{def:pro-etale-morphism} coincides with \cite[Definition 7.8]{etale-cohomology-of-diamonds}.

\begin{lemma} \label{rslt:affinoid-pro-etale-stable-under-composition}
\begin{lemenum}
\item The composition of two affinoid pro-étale morphisms (as defined in \cref{def:pro-etale-morphism.affinoid}) is affinoid pro-étale.
\item Let $Y \rightarrow X$ be an affinoid pro-étale (or pro-étale, respectively) morphism, and let $X' \rightarrow X$ be any morphism from an affinoid perfectoid space $X'$ (from a perfectoid space, respectively). Then the base change $X' \cprod_X Y \rightarrow X'$ is also affinoid pro-étale (pro-étale, respectively).
\end{lemenum}
\end{lemma}
\begin{proof}
(i) Let $Z \to Y \to X$ be a composition of affinoid pro-étale morphisms. If $X$, $Y$ and $Z$ are perfectoid, then the claim is \cite[Lemma 7.11.(i)]{etale-cohomology-of-diamonds}. In general, we can still apply the proof of the reference if we show that \cite[Proposition 6.4.(ii,iv)]{etale-cohomology-of-diamonds} still holds in our setting. That is, we have to show that for every $\kappa$-small projective limit $X = \varprojlim_i X_i$ of noetherian affinoid adic space over $\Q_p$ (with perfectoid or locally noetherian limit $X$) we have natural equivalences of categories
\begin{align*}
	\varinjlim_i (X_i)_{\et,\qcqs} \isoto X_{\et,\qcqs}, \qquad \varinjlim_i (X_i)_{\et,\aff} \isoto X_{\et,\aff}.
\end{align*}
The first equivalence follows from \cite[Proposition 11.23.(ii)]{etale-cohomology-of-diamonds} (using \cite[Lemma 15.6]{etale-cohomology-of-diamonds}). This already proves that that the second functor is fully faithful (as it is the restriction of the first functor). To prove that the second functor is essentially surjective we apply \cite[Proposition 1.7.1]{huber-etale-cohomology} as in the proof of \cite[Proposition 6.4.(iv)]{etale-cohomology-of-diamonds}.

(ii) follows from \cref{def:pro-etale-morphism}.
\end{proof}

\begin{lemma} \label{rslt:affinoid-pro-etale-is-basis}
Let $X$ be a $\kappa$-small adic space over $\Q_p$ and assume that $X$ is either locally noetherian or perfectoid. Then the set of all $Y^\diamond$ for perfectoid spaces $Y$ with pro-étale $Y \to X$ forms a basis of $X_\qproet$.
\end{lemma}
\begin{proof}
We can assume that $X$ is affinoid. Let $Z' \in X_\qproet = X^\diamond_\qproet$ be given. We will show that $Z'$ can be covered by elements $Y^\diamond$ for perfectoid spaces $Y$ which are affinoid pro-étale over $X$. We may assume that $Z'$ is perfectoid. By definition of $X^\diamond$, the morphism $Z' \to X^\diamond$ determines an untilt $Z$ of $Z'$ and a morphism $Z \to X$ such that $Z' = Z^\diamond$ and $Z \to X$ is mapped to $Z^\diamond \to X^\diamond$ under $()^\diamond$.

By \cite[Lemma 15.3]{etale-cohomology-of-diamonds} we can find an affinoid perfectoid space $X'$ with an affinoid pro-étale cover $X' \to X$. By \cite[Lemma 7.18]{etale-cohomology-of-diamonds} there is an affinoid pro-étale cover $\tilde X \to X'$ by a strictly totally disconnected perfectoid space $\tilde X$. From \cref{rslt:affinoid-pro-etale-stable-under-composition} we deduce that $\tilde X \to X$ is an affinoid pro-étale cover of $X$.

Since $Z' \rightarrow X^\diamond$ is quasi-pro-étale (\cite[Definition 10.1]{etale-cohomology-of-diamonds}), the base change $Z^\diamond \cprod_{X^\diamond} \tilde X^\diamond$ is perfectoid and $Z^\diamond \cprod_{X^\diamond} \tilde X^\diamond \rightarrow \tilde X^\diamond$ is pro-étale. The adic space $\tilde Z := Z \cprod_X \tilde X$ maps to $Z^\diamond \cprod_{X^\diamond} \tilde X^\diamond$ under the diamond functor, which is compatible with fiber products.  By definition of pro-étale morphisms, we can cover the perfectoid space $Z^\diamond \cprod_{X^\diamond} \tilde X^\diamond$ by perfectoid spaces $Y_i^\flat$ which are affinoid pro-étale over $\tilde X^\diamond$. By the tilting equivalence, we find affinoid perfectoid spaces $Y_i$, which are affinoid pro-étale over $\tilde X$, with tilt $Y_i^\flat$.
By \cref{rslt:affinoid-pro-etale-stable-under-composition} these $Y_i$ are also affinoid pro-étale over $X$. Their diamonds $Y_i^\diamond$ are represented by the tilts $Y_i^\flat$ and cover $Z'$.\end{proof}

\begin{lemma} \label{rslt:completed-structure-sheaves-coincide}
Let $X$ be a $\kappa$-small adic space over $\Q_p$ and assume that $X$ is either locally noetherian or perfectoid. Then on $X_\qproet$ we have
\begin{align*}
	\hat\ri_X = \hat\ri_{X^\diamond}, \qquad \hat\ri^+_X = \hat\ri^+_{X^\diamond}.
\end{align*}
In other words, \cref{def:completed-structure-sheaf-on-diamond} and \cref{def:completed-structure-sheaf-on-adic-space} coincide.
\end{lemma}
\begin{proof}
We can assume that $X$ is affinoid. The first of the claimed equalities follows from the second one, so we only have to show that $\hat\ri^+_X = \hat\ri^+_{X^\diamond}$. By \cref{rslt:affinoid-pro-etale-is-basis} it is enough to prove that both sides of the equation coincide on $Y^\diamond$ for affinoid perfectoid spaces $Y$ which are affinoid pro-étale over $X$. Write $Y = \Spa(S, S^+) = \varprojlim Y_i$ for étale morphisms $Y_i = \Spa(S_i, S^+_i) \to X$. By definition we have
\begin{align*}
	\hat\ri^+_{X^\diamond}(Y^\diamond) = S^+.
\end{align*}
On the other hand, by \cite[Proposition 14.9]{etale-cohomology-of-diamonds} we notice that
\begin{align*}
	\nu^{-1}\ri^+_X(Y^\diamond) = \varinjlim_i \ri^+_X(Y_i^\diamond) = \varinjlim_i \ri^+_X(Y_i) = \varinjlim_i S_i^+.
\end{align*}
Also, $S^+$ is the $p$-adic completion of $\varinjlim_i S_i^+$, hence $\hat\ri^+_{X^\diamond}(Y^\diamond)$ is the $p$-adic completion of $\nu^{-1}\ri^+_X(Y^\diamond)$. Now $\hat\ri^+_X$ is the $p$-adic completion of $\nu^{-1}\ri^+_X$, so we have to show that the completion of sheaves translates to the completion of sections on $Y^\diamond$. The problem is that a priori it is not clear that $(\nu^{-1}\ri^+_X/p^n)(Y^\diamond) = \nu^{-1}\ri^+_X(Y^\diamond)/p^n$. But it turns out that this identity is easily seen to hold in the almost sense, which is enough for our purpose. We present the details in the following.

By the above computations we see that $\nu^{-1}\ri^+_X(Y^\diamond)/p^n = S^+/p^n = \hat\ri^+_{X^\diamond}(Y^\diamond)/p^n$ for every $n \ge 1$. By \cite[Proposition 8.8]{etale-cohomology-of-diamonds}, $\hat\ri^+_{X^\diamond}(Y^\diamond)/p^n$ is almost isomorphic to $(\hat\ri^+_{X^\diamond}/p^n)(Y^\diamond)$. Altogether we deduce that $\nu^{-1}\ri^+_X(Y^\diamond)/p^n$ is almost isomorphic to $(\hat\ri^+_{X^\diamond}/p^n)(Y^\diamond)$. But $\nu^{-1}\ri^+_X/p^n$ is the sheafification of the presheaf $Y^\diamond \mapsto \nu^{-1}\ri^+_X(Y^\diamond)/p^n$. Hence if we restrict to almost coefficients then no sheafification is needed and $\nu^{-1}\ri^+_X/p^n$ is almost isomorphic to $\hat\ri^+_{X^\diamond}/p^n$. We conclude that for all $n \ge 1$, $\nu^{-1}\ri^+_X/p^n(Y^\diamond)$ is almost isomorphic to $S^+/p^n$. As in the proof of \cite[Lemma 4.10.(ii,iii)]{rigid-p-adic-hodge}, taking the inverse limit over all $n$ yields $\hat\ri^+_X(Y^\diamond) = S^+$, as desired.
\end{proof}

We now arrive at the following comparison result between the rigid-analytic world and the diamond world:

\begin{proposition} \label{rslt:comparison-of-adic-and-diamond-world}
Let $X$ be a locally noetherian adic space over $\Q_p$. Then $\abs X = \abs{X^\diamond}$, $X_\et = X^\diamond_\et$, $X_\fet = X^\diamond_\fet$, $X_\qproet = X^\diamond_\qproet$, $X_\vsite = X^\diamond_\vsite$, and $\hat\ri_X = \hat\ri_{X^\diamond}, \hat\ri^+_X = \hat\ri^+_{X^\diamond}$. In particular, $\ri^+_X/\varpi = \nu_*(\hat\ri^+_{X^\diamond}/\varpi)$ for any pseudo-uniformizer $\varpi$.

If $X$ is seminormal (e.g. normal) and of finite type over a non-archimedean field $K$ over $\Q_p$,  then $\ri_X = \nu_*\hat\ri_{X^\diamond}$ and $\ri^+_X = \nu_*\hat\ri^+_{X^\diamond}$.
\end{proposition}
\begin{proof}
The first two sentences of our claim follow  \cite[Lemma 15.6]{etale-cohomology-of-diamonds} and \cref{rslt:completed-structure-sheaves-coincide}. To deduce $\ri^+_X/\varpi = \nu_*(\hat\ri^+_{X^\diamond}/\varpi)$ we first note that $\hat\ri^+_{X^\diamond}/\varpi = \hat\ri^+_X/\varpi = (\nu^{-1}\ri^+_X)/\varpi$. Using the exactness of $\nu^{-1}$ we obtain $\hat\ri^+_{X^\diamond}/\varpi = \nu^{-1}(\ri^+_X/\varpi)$. Now the claim follows from \cite[Proposition 14.8]{etale-cohomology-of-diamonds}.

To prove the claim about the structure sheaves in the seminormal case, by \cref{rslt:completed-structure-sheaves-coincide} it is enough to show $\nu_*\hat\ri^+_X = \ri^+_X$. This can be found in \cite[Corollary 6.19]{rigid-p-adic-hodge} or \cite[Theorem 8.2.3]{relative-p-adic-hodge-2} (the references use a slightly different notion of the pro-étale site of $X$, but the proofs still apply).
\end{proof}

Altogether we see that much of the structure of a locally noetherian adic space $X$ over $\Q_p$ is captured by its associated diamond: The diamond possesses a topology, an étale site and a completed structure sheaf, all of which coincide with the corresponding structure on the adic space $X$. If $X$ is seminormal and of finite type over some non-archimedean field $K$ then the associated diamond even captures the structure sheaf on $X$. From now on, we will implicitly consider all adic spaces as diamonds and work in the category of diamonds (or even small v-sheaves) throughout.

In the following, we will mostly work with the v-site and hence with the sheaves $\check\ri_X$ and $\check\ri^+_X$. By definition we have $\lambda_* \check\ri^{(+)}_X = \hat\ri^{(+)}_X$ (where $\lambda\colon X_\vsite \to X_\qproet$ is the projection of sites). We want to strengthen these identities by proving them even in the (almost) derived sense. This implies that computing (almost) cohomology on the v-site and the quasi-pro-étale site produces the same result, allowing us to extend many results from the quasi-pro-étale site to the v-site.

\begin{lemma} \label{rslt:higher-direct-image-of-structure-sheaves-qproet-to-v}
Let $X$ be a diamond over $\Spd\Z_p$ and let $\lambda\colon X_\vsite \to X_\qproet$ be the natural projection of sites. Let $\varpi$ be a global pseudo-uniformizer in $\check\ri_X$. Then
\begin{align*}
	R\lambda_* \check\ri_X = \hat\ri_X, \qquad R\lambda_* \check\ri^{+a}_X = \hat\ri^{+a}_X, \qquad R\lambda_* (\check\ri^{+a}_X/\varpi) = \hat\ri^{+a}_X/\varpi.
\end{align*}
\end{lemma}
\begin{proof}
If $\mathcal F$ is any sheaf on $X_\vsite$ then $R^i\lambda_*\mathcal F$ is the sheafification of the presheaf $U \mapsto H^i(U_\vsite, \mathcal F)$. Thus to show $R^i\lambda_*\mathcal F = 0$ for $i > 0$ it is enough to show that there is a basis of $X_\qproet$ such that for $U$ an element of that basis we have $H^i(U_\vsite, \mathcal F) = 0$. We take the basis consisting of affinoid perfectoid spaces in $X_\qproet$. Then the claim follows from \cite[Proposition 8.8]{etale-cohomology-of-diamonds}.
\end{proof}

\begin{lemma} \label{rslt:higher-direct-image-of-structure-sheaves-et-to-qproet}
Let $X$ be a locally noetherian adic space over $\Q_p$ and let $\varpi$ be a global pseudo-uniformizer in $\ri_X$. Then
\begin{align*}
	R\nu_*(\hat\ri^+_X/\varpi) = \ri^+_X/\varpi.
\end{align*}
\end{lemma}
\begin{proof}
This follows from \cite[Proposition 14.8]{etale-cohomology-of-diamonds} using that $\hat\ri^+_X/\varpi = \nu^{-1}(\ri^+_X/\varpi)$ (as in the proof of \cref{rslt:comparison-of-adic-and-diamond-world}).
\end{proof}

By \cref{rslt:higher-direct-image-of-structure-sheaves-qproet-to-v} we have $\lambda_* (\check\ri^{+a}_X/\varpi) = \hat\ri^{+a}_X/\varpi$ on a diamond $X$ over $\Spd\Z_p$. This is in line with the comparisons $\lambda_* \check\ri^{(+)}_X = \hat\ri^{(+)}_X$ following from the definitions. However, the mod-$\varpi$ version is a priori only true in the almost sense, which is not enough for our applications. Proving this equality in the non-almost sense will thus be our next goal (see \cref{rslt:O+-mod-p-on-vsite} below). The following argument was suggested by Peter Scholze.

\begin{lemma} \label{rslt:v-cover-over-Spa-C-is-limit-with-sections}
Let $Y = \Spa(C, C^+)$ for some algebraically closed perfectoid field $C$ of characteristic $p$ with open and bounded valuation subring $C^+$. Let $Z \to Y$ be a v-covering, for some affinoid perfectoid space $Z$. Then there is a representation $Z = \varprojlim Z_i \to Y$ for some affinoid perfectoid spaces $Z_i$ over $Y$ such that each $Z_i \to Y$ admits a section.
\end{lemma}
\begin{proof}
Write $Z = \Spa(S, S^+)$ and let $I = S^+$. For every finite subset $J \subset I$ consider the perfectoid discs
\begin{align*}
	R^+_J = C^+\langle T_i^{1/p^\infty} \setst i \in J \rangle, \qquad R_J = C\langle T_i^{1/p^\infty} \setst i \in J\rangle.
\end{align*}
Consider also the ``infinite perfectoid disc''
\begin{align*}
	R^+ = C^+\langle T_i^{1/p^\infty} \setst i \in I \rangle, \qquad R = C\langle T_i^{1/p^\infty} \setst i \in I\rangle.
\end{align*}
Fix a pseudo-uniformizer $\varpi$ in $C$ and equip $R^+$ with the $\varpi$-adic topology and $R = R^+[\varpi^{-1}]$ with the induced topology. By writing $R^+$ as the $\varpi$-adically completed direct limit over the $R^+_J$'s one sees that $R^+$ is open and integrally closed in $R$. One checks that $R^\circ = C^\circ\langle T_i^{1/p^\infty} \setst i \in I \rangle$ from which one deduces easily that $R$ is perfectoid. Thus, $\tilde Y := \Spa(R, R^+)$ is an affinoid perfectoid space.

Consider the continuous map $f^*\colon R \to S$ mapping $T_i \mapsto i \in S^+$ (by assumption $S$ is perfectoid in characteristic $p$, hence perfect, so that it has unique $p^n$-th roots). This map satisfies $f^*(R^+) \subset S^+$ and thus induces a map $f\colon Z \to \tilde Y$ over $Y$. Furthermore, since both $S^+$ and $R^+$ are equipped with the $\varpi$-adic topology, $f^*$ is a quotient map. This implies that $f$ is a Zariski closed immersion (see \cite[Proposition III.4.1.2]{morel-adic-spaces}). Thus by \cite[Remark 7.9]{etale-cohomology-of-diamonds} we can write $f\colon Z = \varprojlim_\alpha U_\alpha \to \tilde Y$ for some rational open subsets $U_\alpha \subset \tilde Y$.

This reduces the claim to the following: Let $U \subset \tilde Y$ be a rational open subset. Then the projection $U \to Y$ admits a section. By \cite[Lemma 3.10]{huber-cts-valuations} we can write $U = U(\frac{f_1, \dots, f_n}g)$ such that $f_1, \dots, f_n, g \in R$ contain only finitely many of the $T_i$'s. If $J$ is the set of these $i$'s then we see that $U$ can be written as $U = U' \cprod_{\tilde Y_J} \tilde Y$, where $\tilde Y_J := \Spa(R_J, R^+_J)$ and $U'$ is an open rational subset of $\tilde Y_J$. Since there obviously exists a section of $\tilde Y \to \tilde Y_J$, in order to construct a section of $U \to Y$ it remains to construct a section of $U' \to Y$. Such a section exists by \cite[Lemma 9.5]{etale-cohomology-of-diamonds} (cf. the end of the proof of \cite[Lemma 9.4]{etale-cohomology-of-diamonds}).
\end{proof}

\begin{lemma} \label{rslt:presheaf-with-limit-property-is-sheaf-on-strictly-tot-disc}
Let $X$ be a small v-sheaf and let $\mathcal F$ be a presheaf of abelian groups on $X_\vsite$ satisfying the following properties:
\begin{enumerate}[(a)]
	\item For any limit $U = \varprojlim_i U_i$ in $X_\vsite$ with $U$ and all $U_i$ being affinoid perfectoid we have $\mathcal F(U) = \varinjlim_i \mathcal F_i(U)$.
	\item For any $U_1, U_2 \in X_\vsite$ we have $\mathcal F(U_1 \dunion U_2) = \mathcal F(U_1) \cprod \mathcal F(U_2)$.
\end{enumerate}
Then $\mathcal F$ satisfies the sheaf property for every covering $Z \to Y$ in $X_\vsite$ for which $Y$ is strictly totally disconnected.
\end{lemma}
\begin{proof}
Let the covering $f:Z \to Y$ be given. Note that the sequence $0 \to \mathcal F(Y) \to \mathcal F(Z) \to \mathcal F(Z \cprod_Y Z)$ can be written as
\begin{align*}
	0 \to \Gamma(Y, \mathcal F) \to \Gamma(Y, f_*(\restrict{\mathcal F}Z)) \to \Gamma(Y, (f \cprod f)_*(\restrict{\mathcal F}{Z \cprod_Y Z})).
\end{align*}
We have to show that this sequence is exact. Consider the presheaves $\mathcal F_\et$, $f_*(\restrict{\mathcal F}Z)_\et$ and $(f \cprod f)_*(\restrict{\mathcal F}{Z \cprod_Y Z})_\et$ on $Y_\et$, where $(-)_\et$ denotes the projection to the étale site along $\nu \comp \lambda\colon Y_\vsite \to Y_\et$. We claim that these three presheaves are in fact sheaves. Since $Y$ is strictly totally disconnected (in particular every étale cover of $Y$ splits), this amounts to saying that all three presheaves transform finite disjoint unions into products; this is guaranteed by condition (b). Now $\Gamma(Y, -)$ is left exact, so that it is enough to show that the following sequence of sheaves on $Y_\et$ is exact:
\begin{align*}
	0 \to \mathcal F_\et \to f_*(\restrict{\mathcal F}Z)_\et \to (f \cprod f)_*(\restrict{\mathcal F}{Z \cprod_Y Z})_\et.
\end{align*}
The exactness of this sequence can be checked on stalks by \cite[Proposition 14.3]{etale-cohomology-of-diamonds}. By \cite[Example 5.2]{etale-cohomology-of-diamonds} and the fact that $\mathcal F$ (and hence also $f_*(\restrict{\mathcal F}Z)$ and $(f \cprod f)_*(\restrict{\mathcal F}{Z \cprod_Y Z})$) behave nicely under limits by assumption, we can reduce to the case that $Y = \Spa(C, C^+)$ for an algebraically closed perfectoid field $C$ with an open and bounded valuation subring $C^+$.

We can assume that $Z$ is affinoid perfectoid. By \cref{rslt:v-cover-over-Spa-C-is-limit-with-sections} we can write $Z = \varprojlim_i Z_i \to Y$ for some $Z_i \in Y_\vsite$ which admit a section $Y \to Z_i$. As is true on any site, every presheaf satisfies the sheaf property along a covering that admits a section. Consequently, $\mathcal F$ satisfies the sheaf property for the coverings $Z_i \to Y$ for all $i$. By assumption on $\mathcal F$, the sequence $0 \to \mathcal F(Y) \to \mathcal F(Z) \to \mathcal F(Z \cprod_Y Z)$ is obtained as a direct limit of the exact sequences $0 \to \mathcal F(Y) \to \mathcal F(Z_i) \to \mathcal F(Z_i \cprod_Y Z_i)$ and hence is also exact.
\end{proof}

\begin{proposition} \label{rslt:O+-mod-p-on-vsite}
Let $X$ be a small v-sheaf over $\Spd\Z_p$ and let $\varpi$ be a global pseudo-uniformizer in $\check\ri^+_X$. Then for all strictly totally disconnected spaces $Y \in X_\vsite$ we have
\begin{align*}
	(\check\ri^+_X/\varpi)(Y) = \check\ri^+_X(Y)/\varpi.
\end{align*}
In particular, if $X$ is a diamond and $\lambda\colon X_\vsite \to X_\qproet$ denotes the projection of sites then
\begin{align*}
	\lambda_*(\check\ri^+_X/\varpi) = \hat\ri^+_X/\varpi.
\end{align*}
\end{proposition}
\begin{proof}
The second part of the claim follows from the first part by evaluating both sides on w-contractible spaces $Y$, which form a basis of $X_\qproet$ by \cref{rslt:w-contractible-basis}. Indeed, on such a space $Y$, any (quasi-)pro-étale sheaf is acyclic, so in particular this is true for $\hat\ri^+_X$ and hence we have $(\hat\ri^+_X/\varpi)(Y) = \hat\ri^+_X(Y)/\varpi = \check\ri^+_X(Y)/\varpi = (\check\ri^+_X/\varpi)(Y)$.

To prove the first part of the claim, let $\mathcal F$ be the presheaf on $X_\vsite$ defined by $Y \mapsto \check\ri^+_X(Y)/\varpi$. It is enough to show that $\mathcal F$ satisfies the sheaf property for all coverings $f\colon Z \to Y$ in $X_\vsite$ with $Y$ strictly totally disconnected. We apply \cref{rslt:presheaf-with-limit-property-is-sheaf-on-strictly-tot-disc}, i.e. we have to verify that $\mathcal F$ satisfies the conditions (a) and (b). It is clear that $\mathcal F$ satisfies condition (b). To show (a), we have to prove that $\mathcal F(U) = \varinjlim \mathcal F(U_i)$ for any cofiltered inverse limit $U = \varprojlim_i U_i$ of affinoid perfectoid spaces in $X_\vsite$. But this is clear: We first note that $U^\sharp = \varprojlim_i U_i^\sharp$ (e.g. by \cite[Corollary 3.20]{etale-cohomology-of-diamonds}). Thus, writing $U^\sharp = \Spa(R, R^+)$ and $U_i^\sharp = \Spa(R_i, R_i^+)$ we have that $R^+$ is the completion of $\varinjlim R_i^+$, in particular $\mathcal F(U) = R^+/\varpi = \varinjlim_i R_i^+/\varpi = \varinjlim_i \mathcal F(U_i)$.
\end{proof}

We will use \cref{rslt:O+-mod-p-on-vsite} in \cref{rslt:global-sections-on-proper-space} below in order to avoid having to work with almost coefficients later on. Another nice consequence of \cref{rslt:O+-mod-p-on-vsite} is the following.

\begin{corollary} \label{rslt:hat-O-flat-over-O-E}
Let $X$ be a small v-sheaf over $\Cpx_p$. Then $\check\ri^+_X/\varpi$ is flat over $\ri/\varpi$ (as sheaves on $X_\vsite$) for all pseudo-uniformizers $\varpi \in \ri$.
\end{corollary}
\begin{proof}
The sheaf $(\ri/\varpi)_X$ on $X_\vsite$ is the sheafification of the presheaf $Y \mapsto \ri/\varpi$, so it is enough to show that $(\check\ri^+_X/\varpi)(U)$ is flat over $\ri/\varpi$ for $U$ in a basis of $X_\vsite$. Such a basis is given by strictly totally disconnected spaces $U$, in which case we have $(\check\ri^+_X/\varpi)(U) = \check\ri^+_X(U)/\varpi$ by \cref{rslt:O+-mod-p-on-vsite}. On the other hand, $\check\ri^+_X(U) = \ri^+_{U^\sharp}(U^\sharp)$ has no $\varpi$-torsion and is thus flat over $\ri$; hence $\check\ri^+_X(U)/\varpi$ is flat over $\ri/\varpi$.
\end{proof}

Having extensively studied the structure sheaf of a diamond (or even small v-sheaf) $X$, we now turn our attention to locally free sheaves on $X$. Our focus lies on understanding how they compare when studied on the étale, quasi-pro-étale and v-site. We show that given a locally free sheaf $\mathcal M$ of $\ri_X$-modules on any of the three sites, we can equivalently study $\mathcal M$ on any bigger site. However there are more locally free sheaves of $\ri_X$-modules on $X_\vsite$ than there are on $X_\qproet$ (and the same for $X_\qproet$ and $X_\et$).

\begin{proposition} \label{rslt:locfree-sheaf-fully-faithful-qproet-to-v}
Let $X$ be a diamond over $\Spd\Z_p$, let $\varpi$ be a global pseudo-uniformizer in $\check\ri_X$ and let
\begin{align*}
	\lambda\colon (X_\vsite, \check\ri_X) &\to (X_\qproet, \hat\ri_X),\\
	\lambda^+\colon (X_\vsite, \check\ri^+_X) &\to (X_\qproet, \hat\ri^+_X),\\
	\lambda^+_\varpi\colon (X_\vsite, \check\ri^+_X/\varpi) &\to (X_\qproet, \hat\ri^+_X/\varpi)
\end{align*}
be the natural morphisms of ringed sites. Then the pullbacks $\lambda^*$, $\lambda^{+*}$ and $\lambda^{+*}_\varpi$ are fully faithful on the respective categories of locally free sheaves of modules.

Moreover, on these categories, $\lambda^*$ preserves cohomology and $\lambda^{+*}$ and $\lambda^{+*}_\varpi$ preserve almost cohomology.
\end{proposition}

Note that this statement differs from \cite[Proposition 14.7]{etale-cohomology-of-diamonds} because the $\lambda^*$ in the reference denotes inverse image (denoted $\lambda^{-1}$ in the present paper), whereas the functors $\lambda^*$, $\lambda^{+*}$ and $\lambda^{+*}_\varpi$ are pullbacks along morphisms of \emph{ringed} sites.

\begin{proof}
We only prove the claim for $\lambda$; the other two cases are analogous (using \cref{rslt:O+-mod-p-on-vsite} for $\lambda^+_\varpi$). Let $\mathcal M$ be a locally free sheaf of $\hat\ri_X$-modules on $X_\qproet$. Then the natural morphism $\mathcal M \isoto R\lambda_*\lambda^*\mathcal M$ is an isomorphism: This can be checked locally on $X_\qproet$, so we can assume that $\mathcal M$ is free, which reduces to the case $\mathcal M = \hat\ri_X$. Then the claimed isomorphism is stated in \cref{rslt:higher-direct-image-of-structure-sheaves-qproet-to-v}. This already shows that $\lambda^*$ preserves cohomology. We also note that for two locally free sheaves of $\hat\ri_X$-modules $\mathcal M$ and $\mathcal N$ on $X_\qproet$ we have by adjunction
\begin{align*}
	\Hom_{\check\ri_X}(\lambda^*\mathcal M, \lambda^*\mathcal N) &= \Hom_{\hat\ri_X}(\mathcal M, \lambda_*\lambda^*\mathcal N) = \Hom_{\hat\ri_X}(\mathcal M, \mathcal N).
\end{align*}
This proves fully faithfulness of $\lambda^*$.
\end{proof}

\begin{proposition} \label{rslt:locfree-sheaf-fully-faithful-et-to-qproet}
Let $X$ be a locally noetherian adic space over $\Q_p$ and let $\varpi$ be a global pseudo-uniformizer in $\ri_X$. Let
\begin{align*}
	\nu\colon (X_\qproet, \hat\ri_X) &\to (X_\et, \ri_X),\\
	\nu^+\colon (X_\qproet, \hat\ri^+_X) &\to (X_\et, \ri^+_X),\\
	\nu^+_\varpi\colon (X_\qproet, \hat\ri^+_X/\varpi) &\to (X_\et, \ri^+_X/\varpi)
\end{align*}
be the natural morphisms of ringed sites. Then the pullback $\nu^{+*}_\varpi$ is fully faithful on the category of locally free sheaves of modules, where it also preserves cohomology.

If $X$ is seminormal (e.g. normal) and of finite type over a non-archimedian field over $\Q_p$, then the pullbacks $\nu^*$ and $\nu^{+*}$ are fully faithful on the respective categories of locally free sheaves of modules.
\end{proposition}
\begin{proof}
Analogous to \cref{rslt:locfree-sheaf-fully-faithful-qproet-to-v}, using \cref{rslt:higher-direct-image-of-structure-sheaves-et-to-qproet} and \cref{rslt:comparison-of-adic-and-diamond-world} in place of \cref{rslt:higher-direct-image-of-structure-sheaves-qproet-to-v}.
\end{proof}

\section{Local Systems on Diamonds} \label{sec:locsys}

In this section we introduce and study local systems on diamonds (or more generally small v-sheaves). Our main results are the following: First we show that pro-étale local systems behave similarly on diamonds as on schemes by adapting some results from \cite{proetale-topology} to diamonds. Afterwards we prove that pro-étale local systems are equivalent to v-local systems on a diamond (under mild assumptions on the coefficient ring), which allows us to replace the pro-étale site by the much bigger v-site later on.

In order to define categories of étale, quasi-pro-étale and v-local systems on a locally spatial diamond $X$ with coefficients in a suitable ring $\Lambda$
we start with some generalities which are hopefully useful for future reference. Note however, that in the present paper we will focus on topological rings $\Lambda$, where \cref{rslt:explicit-computation-of-S-X} below provides a more explicit description of the relevant objects.

In the classical setting of the étale site, a good theory of local systems is only possible for discrete rings $\Lambda$. Given such a ring $\Lambda$ and a ``space'' $X$ (e.g. a scheme, adic space or spatial diamond) one defines the constant sheaf $\Lambda_X$ on $X_\et$ to be the sheafification of the constant presheaf $U \mapsto \Lambda$. There is a different, in many cases more illuminating, definition of the constant sheaf $\Lambda_X$: Let $*_\et$ be the site consisting of finite discrete sets and jointly surjective coverings. This site is equivalent to the étale site of a geometric point, hence the notation.  Now a discrete ring $\Lambda$ is the same as a sheaf of rings on $*_\et$. Moreover, there is a natural morphism of sites $\pr\et\colon X_\et \to *_\et$ given by $\pr\et^{-1}(S) = \bigdunion_{s\in S} X$ for $S \in *_\et$. It follows straight from the definitions that the constant sheaf $\Lambda_X$ can be described as the pullback $\Lambda_X = \pr\et^{-1}\Lambda$.

The point of view on constant sheaves presented in the previous paragraph has many advantages because we can apply the general theory of morphisms of sites (e.g. the adjunction of $\pr{\et*}$ and $\pr\et^{-1}$). It suggests that we extend the definition in order to include topological rings as follows. Replace $*_\et$ by the site $*_\proet$ of profinite spaces with covers given by finitely many continuous jointly surjective maps. Sheaves on this site are called condensed sets and are studied extensively in \cite{condensed-mathematics}. In particular it is shown there that the category of compactly generated topological rings (this includes all metrizable spaces) embeds fully faithfully into the category of sheaves of rings on $*_\proet$ (see \cite[Proposition 1.7]{condensed-mathematics}), where we implicitly assume a cardinality bound $\kappa$.

\begin{definition}
Let $\mathcal C$ be any category. A \emph{discrete object of $\mathcal C$} is a sheaf of objects in $\mathcal C$ on the site $*_\et$. A \emph{condensed object of $\mathcal C$} is a sheaf of objects in $\mathcal C$ on the site $*_\proet$.

In case $\mathcal C$ is the category of sets/groups/rings, we also speak of discrete sets/groups/rings and condensed sets/groups/rings.
\end{definition}

\begin{definition} \label{def:constant-sheaves}
\begin{defenum}
	\item \label{def:constant-etale-sheaves} Let $X$ be a locally spatial diamond and let $T$ be a discrete set. We define the \emph{constant étale sheaf} $T_X$ on $X_\et$ as $T_X = \pr\et^{-1} T$, where $\pr\et\colon X_\et \to *_\et$ is the natural projection of sites defined above.
	\item Let $X$ be a diamond and let $T$ be a condensed set. We define the \emph{constant quasi-pro-étale sheaf} $T_X$ on $X_\qproet$ as $T_X = \pr\qproet^{-1} T$, where $\pr\qproet\colon X_\qproet \to *_\proet$ is the natural projection of sites given by $\pr\qproet^{-1}S = \varprojlim \pr\et^{-1}S_i$ for a profinite set $S = \varprojlim_i S_i$ with $S_i$ finite discrete.
	\item Let $X$ be a small v-sheaf (e.g. a diamond) and let $T$ be a condensed set. We define the \emph{constant v-sheaf} $T_X$ on $X_\vsite$ as $T_X = \pr\vsite^{-1} T$, where $\pr\vsite\colon X_\vsite \to  *_\proet$ is the natural projection of sites defined similarly as $\pr\qproet$ (in the case that $X$ is a diamond there is a definition for $X_\qproet$ and then $\pr\vsite$ is just the composition $X_\vsite \to X_\qproet \to *_\proet$).
\end{defenum}
\end{definition}

Note that these notions behave well under pullbacks. For example, if $f\colon Y \to X$ is a morphism of diamonds and $T$ is a condensed set then $f_\qproet^{-1} T_X = T_Y$. This follows from the functoriality of pullbacks. It is also clear that being a constant sheaf does not depend on the chosen site, i.e. constant sheaves pull back to constant sheaves along $X_\qproet \to X_\et$ and $X_\vsite \to X_\qproet$.

\begin{definition}\label{def:localsys}
In each of the cases of \cref{def:constant-sheaves}, assume that $\Lambda := T$ is a sheaf of rings. Then $\Lambda_X$ is a sheaf of rings on $X_\tau$, where $\tau \in \{ \et, \qproet, \vsite \}$ denotes the site in question, and we often drop the subscript, i.e. denote $\Lambda_X$ simply be $\Lambda$.
\begin{defenum}
	\item A sheaf $\mathcal L$ of $\Lambda$-modules on $X_\tau$ is called \emph{constant} if $\mathcal L \isom \Lambda^r$ for some $r \ge 0$. $\mathcal L$ is called \emph{locally constant} or a $\Lambda$-local system if locally on $X_\tau$, $\mathcal L$ is constant.
	\item The category of $\Lambda$-local systems on $X_\tau$ (with all morphisms of $\Lambda$-modules) is denoted $\catlocsys\Lambda{X_\tau}$. By abuse of notation we usually write $\catlocsys\Lambda{X}$	 for $\catlocsys\Lambda{X_\vsite}$ (this is justified by \cref{rmk:locsys-equiv-on-all-sites}).
\end{defenum}	
We say that a local system $\mathcal L $ on $X_\tau$  has \emph{constant rank} if there exists a cover $(Y_i)_ {i \in I}$ of $X$ in the $\tau$-topology such that $\mathcal L|_{Y_i}\isom \Lambda_{Y_i} ^r$ for $r$ independent of $i$.
\end{definition}

We first study the behaviour of the categories of local systems under direct and inverse limits of the coefficient ring. 

\begin{lemma} \label{rslt:constant-sheaf-commutes-with-lim-and-colim}
Let $X$ be a small v-sheaf. Then the functor $T \mapsto T_X$ associating to every condensed set $T$ the constant sheaf $T_X$ on $X_\vsite$ commutes with all limits and colimits. If $X$ is a diamond, then the same is true for $X_\qproet$ in place of $X_\vsite$.
\end{lemma}
\begin{proof}
In both the $X_\qproet$ and $X_\vsite$ case it is clear that the functor $T \mapsto T_X$ commutes with all colimits, because it is given by a pullback. To show that the functor also commutes with limits, it is enough to consider the case of $X_\vsite$, as $T_{X_\qproet} = \lambda_* T_{X_\vsite}$ (by \cite[Proposition 14.7]{etale-cohomology-of-diamonds}) and $\lambda_*$ commutes with all limits.

Let $\pr\vsite\colon X_\vsite \to *_\proet$ be the natural morphism of sites. Then we have to show that $\pr\vsite^{-1}$ commutes with all limits. We imitate the proof of \cite[Lemma 14.4]{etale-cohomology-of-diamonds}: Since restriction of sheaves commutes with limits, the claim can be checked v-locally on $X$, so we can assume that $X$ is a qcqs  perfectoid space. Then the topos $X_\vsite^\topos$ is equivalent to the topos $X_{\vsite,\qcqs}^\topos$ coming from the subsite of qcqs perfectoid spaces $Y$ over $X$. For every such $Y$ let $\lambda_\circ(Y) := \pi_0(Y)$. This defines a functor $\lambda_\circ\colon X_{\vsite,\qcqs} \to *_\qproet$. It has the following properties:
\begin{itemize}
	\item The functor $\lambda_\circ$ is left adjoint to the functor $\pr\vsite^{-1}$. To check this let any $V \in X_{\vsite,\qcqs}$ and $S \in *_\qproet$ be given. The claim is that
	\begin{align*}
		\Hom_X(V, \pr\vsite^{-1}S) = \Hom(\lambda_\circ(V), S).
	\end{align*}
	If $S$ is a finite set then $\pr\vsite^{-1}S$ is a finite disjoint union of copies of $X$ and the identity of $\Hom$'s is easy to see. For general $S$ write $S = \varprojlim_{i\in I} S_i$ with finite sets $S_i$ and use that $\pr\vsite^{-1}S = \varprojlim_i \pr\vsite^{-1} S_i$ (by definition of $\pr\vsite^{-1}$) and that $\Hom$ commutes with limits in its second argument.

	\item If $V_1 \to V_2 \from V_3$ is a diagram of strictly totally disconnected spaces in $X_{\vsite,\qcqs}$ then $\lambda_\circ(V_1 \cprod_{V_2} V_3) = \lambda_\circ(V_1) \cprod_{\lambda_\circ(V_2)} \lambda_\circ(V_3)$. To see this, one easily reduced to the case that $V_1$, $V_2$ and $V_3$ are connected, in which case we can apply \cite[Lemma 14.6]{etale-cohomology-of-diamonds}.
\end{itemize}
Let now $T$ be a condensed set and $Y \in X_{\vsite,\qcqs}$ a strictly totally disconnected space. We claim that
\begin{align*}
	(\pr\vsite^{-1}T)(Y) = T(\lambda_\circ(Y)).
\end{align*}
To prove this, first note that by the adjunction of $\lambda_\circ$ and $\pr\vsite^{-1}$ it follows directly from the definitions that $\pr\vsite^{-1}T$ is the sheafification of the presheaf $V \mapsto T(\lambda_\circ(V))$. It is thus enough to show that this presheaf satisfies the sheaf property for all coverings $V \to Y$ with $V$ strictly totally disconnected (because these $V$'s form a basis of the v-site). This follows from the sheaf property of $T$ on $*_\qproet$ and the fact that $\lambda_\circ$ commutes with fiber products of strictly totally disconnected spaces (as shown above).

To finish the proof, we observe that $T(\lambda_\circ(Y))$ commutes with limits in $T$, hence so does $(\pr\vsite^{-1}T)(Y)$. As this is true for all strictly totally disconnected $Y \in X_\vsite$, we deduce that $\pr\vsite^{-1}T$ commutes with limits in $T$.
\end{proof}

\begin{proposition} \label{rslt:projlim-of-locsys}
Let $X$ be a diamond and let $(\Lambda_i)_{i\in I}$ be an inverse system of condensed rings with limit $\Lambda = \varprojlim_i \Lambda_i$. Then the natural base change functor
\begin{align*}
	\catlocsys\Lambda{X_\qproet} \isoto \varprojlim_i \catlocsys{\Lambda_i}{X_\qproet}
\end{align*}
is an equivalence. If $I = (\N, \le)$ then also the natural base change functor
\begin{align*}
	\catlocsys\Lambda{X_\vsite} \isoto \varprojlim_i \catlocsys{\Lambda_i}{X_\vsite}
\end{align*}
is an equivalence.
\end{proposition}
\begin{proof}
Note first that by \cref{rslt:constant-sheaf-commutes-with-lim-and-colim} it does not matter whether we form the limit of the $\Lambda_i$'s on $*_\proet$, on $X_\qproet$ or on $X_\vsite$. This is implicitly used in the following proof.

We first show that both of the displayed functors are fully faithful (even for general $I$ on the v-site), so let $\tau \in \{ \qproet, \vsite \}$. Let us observe the following: For any $\Lambda$-local system $\mathcal L$ on $X_\tau$ we have a natural isomorphism $\varprojlim_i (\mathcal L \tensor_\Lambda \Lambda_i) = \mathcal L$. Indeed, we have a natural map from right to left, so the claimed isomorphism can be checked locally. But then $\mathcal L \isom \Lambda^r$ for some $r$ and the claim is obvious. Now let $\mathcal L, \mathcal L' \in \catlocsys\Lambda{X_\tau}$ be given. Under the displayed functor they are mapped to $(\mathcal L \tensor_\Lambda \Lambda_i)_i$ and $(\mathcal L' \tensor_\Lambda \Lambda_i)_i$ respectively. We get
\begin{align*}
	&\Hom((\mathcal L \tensor_\Lambda \Lambda_i)_i, (\mathcal L' \tensor_\Lambda \Lambda_i)_i) = \varprojlim_i \Hom_{\Lambda_i}(\mathcal L \tensor_\Lambda \Lambda_i, \mathcal L' \tensor_\Lambda \Lambda_i)\\
	&\qquad= \Gamma(X, \varprojlim_i \IHom_{\Lambda_i}(\mathcal L \tensor_\Lambda \Lambda_i, \mathcal L' \tensor_\Lambda \Lambda_i)) = \Gamma(X, \varprojlim_i \IHom_\Lambda(\mathcal L, \mathcal L') \tensor_\Lambda \Lambda_i)\\
	&\qquad= \Gamma(X, \IHom_\Lambda(\mathcal L, \mathcal L')) = \Hom(\mathcal L, \mathcal L'),
\end{align*}
proving fully faithfulness.

It remains to prove essential surjectivity: Let first $\tau = \qproet$ and let $(\mathcal L_i)_{i\in I}$ be a given inverse system with $\mathcal L_i \in \catlocsys{\Lambda_i}{X_\qproet}$. Then let $\mathcal L := \varprojlim_i \mathcal L_i$. We have to prove that $\mathcal L$ is locally constant. We can thus assume that $X$ is w-contractible. But then every local system on $X_\qproet$ is constant: It is constant on some $U \in X_\qproet$, but then there is a section $X \to U \to X$, so it is also constant on $X$. Hence each $\mathcal L_i$ is constant and consequently the same is true for $\mathcal L$.

Now let $\tau = \vsite$ and $I = \N$. Let $(\mathcal L_n)_{n\in\N}$ be a given inverse system with $\mathcal L_n \in \catlocsys{\Lambda_n}{X_\vsite}$. Then let $\mathcal L := \varprojlim_n \mathcal L_n$. We have to prove that $\mathcal L$ is locally constant. Working locally, we may assume that $X = X_1$ is a spatial diamond such that $\restrict{\mathcal L_1}{X_1}$ is constant. We can then inductively find covers by spatial diamonds $X_n \to X_{n-1}$ such that $\restrict{\mathcal L_n}{X_n}$ is constant. Let $\tilde X := \varprojlim_n X_n$. By \cite[Lemma 11.22]{etale-cohomology-of-diamonds}, $\tilde X$ is a spatial diamond and $\abs{\tilde X} = \varprojlim_n \abs{X_n}$. In particular $\abs{\tilde X} \to \abs X$ is surjective and hence $\tilde X \to X$ is surjective by \cite[Lemma 12.11]{etale-cohomology-of-diamonds}. As each of the $\mathcal L_n$ is constant on $\tilde X$ the same is true for $\mathcal L$.
\end{proof}

\begin{proposition} \label{rslt:injlim-of-locsys}
Let $X$ be a qcqs small v-sheaf and let $(\Lambda_i)_{i\in I}$ be a filtered direct system of condensed rings with limit $\Lambda = \varinjlim_i \Lambda_i$. Then the natural base-change functor
\begin{align*}
	\varinjlim_i \catlocsys{\Lambda_i}{X_\vsite} \isoto \catlocsys\Lambda{X_\vsite}
\end{align*}
is an equivalence. The same is true with $X_\vsite$ replaced by $X_\qproet$ if $X$ is a diamond and with $X_\vsite$ replaced by $X_\et$ if $X$ is a locally spatial diamond and all $\Lambda_i$ are discrete.
\end{proposition}
\begin{proof}
We only provide the proof for $X_\vsite$. The cases $X_\qproet$ and $X_\et$ can be handled in the same way.

Since $X$ is qcqs, we can restrict the v-site $X_\vsite$ to the site $X_{\vsite,\qcqs}$ of qcqs (and even representable) objects inside $X_\vsite$ to get the same topos (as in the beginning of the proof of \cite[Lemma 14.4]{etale-cohomology-of-diamonds}). By \cite[Lemma 0739]{stacks-project}, filtered direct limits of sheaves on $X_{\vsite,\qcqs}$ can be computed directly on sections, i.e. if $(\mathcal F_j)_{j\in J}$ is a direct system of sheaves on $X_{\vsite,\qcqs}$ then $\Gamma(U, \varinjlim_j \mathcal F_j) = \varinjlim_j \Gamma(U, \mathcal F_j)$ for all $U \in X_{\vsite,\qcqs}$.

Let us first check that the displayed functor is fully faithful. For this, let $\mathcal L$ and $\mathcal L'$ be two elements if the left-hand side. They can be represented by $\mathcal L_i, \mathcal L'_i \in \catlocsys{\Lambda_i}X$ for some $i \in I$. Then
\begin{align*}
	&\Hom(\mathcal L, \mathcal L') = \varinjlim_{k\ge i} \Hom_{\Lambda_k}(\mathcal L_i \tensor_{\Lambda_i} \Lambda_k, \mathcal L'_i \tensor_{\Lambda_i} \Lambda_k) = \varinjlim_{k\ge i} \Gamma(X_\vsite, \IHom_{\Lambda_k}(\mathcal L_i \tensor_{\Lambda_i} \Lambda_k, \mathcal L'_i \tensor_{\Lambda_i} \Lambda_k))\\
	&\qquad= \Gamma(X_\vsite, \varinjlim_{k\ge i} \IHom_{\Lambda_k}(\mathcal L_i \tensor_{\Lambda_i} \Lambda_k, \mathcal L'_i \tensor_{\Lambda_i} \Lambda_k)) = \Gamma(X_\vsite, \varinjlim_{k\ge i} \IHom_{\Lambda_i}(\mathcal L_i, \mathcal L'_i) \tensor_{\Lambda_i} \Lambda_k)\\
	&\qquad= \Gamma(X_\vsite, \IHom_{\Lambda_i}(\mathcal L_i, \mathcal L'_i) \tensor_{\Lambda_i} \Lambda) = \Gamma(X_\vsite, \IHom_\Lambda(\mathcal L_i \tensor_{\Lambda_i} \Lambda, \mathcal L'_i \tensor_{\Lambda_i} \Lambda))\\
	&\qquad= \Hom(\mathcal L_i \tensor_{\Lambda_i} \Lambda, \mathcal L'_i \tensor_{\Lambda_i} \Lambda),
\end{align*}
as desired. To prove essential surjectivity, let $\mathcal L$ be a given $\Lambda$-local system on $X_\vsite$. Then there is a covering $(U_r \to X)_{r\ge0}$ in $X_{\vsite,\qcqs}$ (with almost all $U_r$ empty) such that for each $r \ge 0$ there is an isomorphism $\restrict{\mathcal L}{U_r} \isom \Lambda_{U_r}^r$. Clearly we must have $U_r \cprod_X U_{r'} = \emptyset$ for $r \ne r'$. Thus $\mathcal L$ is the unique sheaf described by a descent datum $(\varphi_r\colon \Lambda_{U_r \cprod_X U_r}^r \isoto \Lambda_{U_r \cprod_X U_r}^r)_r$ along $(U_r \to X)_r$. Let us abbreviate $Y_r := U_r \cprod_X U_r$. By the fully faithfulness we just proved (applied to $Y_r$ instead of $X$) we see that for some $k_r \in I$ we can find a representation $\varphi_r = \varphi_{r,k_r} \in \Hom_{\Lambda_{k_r}}(\Lambda_{k_r,Y_r}^r, \Lambda_{k_r,Y_r}^r)$. By enlarging $k_r$ if necessary we can assume that $\varphi_{r,k_r}$ is an isomorphism and defines a descent datum along $U_r \to X$. By further enlarging $k_r$ we can assume that all $k_r$'s are equal, say $k_r = k$ (by quasi-compactness there are only finitely many $k_r$'s). Letting $\mathcal L_k$ be the descended $\Lambda_k$-local system on $X_\vsite$ along $(\varphi_{k,r}\colon \Lambda_{k,Y_r}^r \isoto \Lambda_{k,Y_r}^r)_r$, we see that $\mathcal L_k \tensor_{\Lambda_k} \Lambda = \mathcal L$, hence $\mathcal L$ lies in the essential image of our functor.
\end{proof}

If $X$ is a diamond we have a morphism of sites
\begin{align*}
	\lambda = \lambda_X\colon X_\vsite \to X_\qproet,
\end{align*}
and if $X$ is a locally spatial diamond we also have a morphism of sites
\[
	\nu = \nu_X\colon X_\qproet \to X_\et.
\]
By \cite[Proposition 14.7]{etale-cohomology-of-diamonds} and \cite[Proposition 14.8]{etale-cohomology-of-diamonds} the pullbacks $\lambda^{-1}$ and $\nu^{-1}$ are fully faithful. In particular, if $\Lambda$ is a condensed ring then $\lambda^{-1}$ induces a fully faithful functor from the category of $\Lambda$-local systems on $X_\qproet$ to the category of $\Lambda$-local systems on $X_\vsite$, and similarly if $\Lambda$ is a discrete ring, then $\nu^{-1}$ is a fully faithful functor from the category of $\Lambda$-local systems on $X_\et$ to the category of $\Lambda$-local systems on $X_\qproet$. We will show in the following that both of theses functors are equivalences for a large class of rings $\Lambda$. This means that for studying local systems it does not matter which site to work on.

Let us start with the functor $\nu^{-1}$:

\begin{proposition} \label{rslt:et-locsys-equiv-qproet-locsys}
Let $X$ be a locally spatial diamond and let $\Lambda$ be a discrete ring. Then pullback along the morphism of sites $\nu\colon X_\qproet \to X_\et$ induces an equivalence of categories
\begin{align*}
	\nu^{-1}\colon \catlocsys\Lambda{X_\et} \isoto \catlocsys\Lambda{X_\qproet}.
\end{align*}
\end{proposition}
\begin{proof}
As explained above, we only need to show that $\nu^{-1}$ is essentially surjective. Thus let $\mathcal L$ be a $\Lambda$-local system on $X_\qproet$. By definition of local systems, there is a quasi-pro-étale cover $(f_r\colon Y_r \to X)_{r\ge0}$ by some locally spatial diamonds $Y_r$ such that $f_{r,\qproet}^{-1}\mathcal L = (f_{r,\qproet}^{-1}\Lambda)^r$. On the other hand we have $f_{r,\qproet}^{-1}\Lambda = \nu_{Y_r}^{-1} \Lambda_{Y_r}$, so that $f_{r,\qproet}^{-1}\mathcal L$ lies in the image of $\nu_{Y_r}^{-1}$. Letting $f\colon Y = \bigdunion_r Y_r \to X$ be the combined map we see that $f_\qproet^{-1}\mathcal L$ lies in the image of $\nu_Y^{-1}$. By \cite[Theorem 14.12.(i)]{etale-cohomology-of-diamonds} we deduce that $\mathcal L$ lies in the image of $\nu_X^{-1}$, i.e. we have $\mathcal L = \nu^{-1} \mathcal L_0$ for some sheaf of $\Lambda$-modules $\mathcal L_0$ on $X_\et$.

To finish the proof, it remains to show that $\mathcal L$ is locally constant on the étale site (not just on the quasi-pro-étale site), i.e. that $\mathcal L_0$ is locally constant. The problem is étale local on $X$ so we can assume that $X$ is spatial. Then almost all $Y_r$ are empty, so that there is some $n \ge 0$ such that $Y = \bigdunion_{r=0}^n Y_r$; we can furthermore assume all $Y_r$ to be qcqs. Clearly $Y_r \cprod_X Y_{r'} = \emptyset$ for $r \ne r'$, from which we deduce that $\abs X = \bigdunion_{r=0}^n \abs{f_r}(\abs{Y_r})$. Now each $\abs{f_r}: \abs{Y_r} \rightarrow \abs X$ is quasi-compact and generalizing (see \cite[Proposition 11.19]{etale-cohomology-of-diamonds}) and hence specializing, since the images of these maps constitute a disjoint covering of $\abs X$. Then \cite[Lemma 2.4]{etale-cohomology-of-diamonds} shows that each $\abs{f_r}(\abs{Y_r})$ is closed and therefore also open. Thus we may replace $X$ by the open subspaces $\abs{f_r}(\abs{Y_r})$ in order to reduce to the case that $Y = Y_r$ for some $r$, i.e. that $\mathcal L$ has constant rank $r$.

As in the proof of \cite[Proposition 14.8]{etale-cohomology-of-diamonds} we can assume that $Y$ is the cofiltered inverse limit $Y = \varprojlim_i Y_i$ of étale covers $Y_i \to X$. By \cite[Proposition 14.9]{etale-cohomology-of-diamonds} we have $\mathcal L(Y) = \varinjlim_i \mathcal L(Y_i)$. The isomorphism $\restrict{\mathcal L}Y \isom \Lambda_Y^r$ is given by $r$ sections $s_1, \dots, s_r \in \mathcal L(Y)$. Thus $s_1, \dots, s_r$ all lie in $\mathcal L(Y_i)$ for some $i$. This implies that the morphism $\Lambda_Y^r \to \restrict{\mathcal L}Y$ extends to a morphism $\Lambda_{Y_i}^r \to \restrict{\mathcal L}{Y_i}$, which is an isomorphism because this is true quasi-pro-étale locally. Thus $Y_i \to X$ is an étale cover such that $\restrict{\mathcal L}{Y_i}$ is constant, as desired.

\end{proof}

The functor $\lambda^{-1}$ is somewhat harder to handle. Fortunately (and maybe surprisingly), the analogous statement of the second half of the proof of \cref{rslt:et-locsys-equiv-qproet-locsys} also works for $\lambda^{-1}$:

\begin{lemma} \label{rslt:qproet-locsys-equiv-v-locsys-helper}
Let $X$ be a diamond, let $\Lambda$ be a condensed ring and let $\mathcal L_0$ be a sheaf of $\Lambda$-modules on $X_\qproet$ such that $\mathcal L := \lambda^{-1}\mathcal L_0$ is a $\Lambda$-local system on $X_\vsite$. Then $\mathcal L_0$ is a $\Lambda$-local system on $X_\qproet$.
\end{lemma}
\begin{proof}
The claim can be checked quasi-pro-étale locally, so by definition of diamonds we can assume that $X$ is a perfectoid space; restricting further we can assume that $X$ is affinoid perfectoid and by \cite[Lemma 7.18]{etale-cohomology-of-diamonds} we can even assume that $X$ is a strictly totally disconnected perfectoid space. By definition of local systems there is a v-cover $(f_r\colon Y_r \to X)_{r\ge0}$ by some small v-sheaves $Y_r$ such that $\restrict{\mathcal L}{Y_r} \isom \Lambda_{Y_r}^r$ for all $r$.  By quasi-compactness of $X$, we can further assume that  this cover consists of finitely many quasi-compact  $Y_r$. We may replace each $Y_r$ by any v-cover of itself and therefore assume that it is affinoid perfectoid (e.g. totally disconnected). As in the proof of \cref{rslt:et-locsys-equiv-qproet-locsys}, the images of $Y_r$ in $X$ form a disjoint clopen cover of $X$, so that we can replace $X$ by each element of this cover in order to reduce to the case that $Y = Y_r \to X$ is a covering for some $r$. Note that $Y \to X$ is separated.
 
Now consider the (quasi-)pro-étale cover $\lambda_{X\circ}(Y) \to X$ defined in \cite[Lemma 14.5]{etale-cohomology-of-diamonds} (note that $\lambda_{X\circ}(Y) \to X$ is surjective by part (i) of the lemma). The last part of the proof of \cite[Lemma 14.4]{etale-cohomology-of-diamonds} (to be found below Lemma 14.6 in loc. cit.) together with \cite[Proposition 14.7]{etale-cohomology-of-diamonds} shows that
\begin{align*}
	\mathcal L(Y) = (\lambda^{-1}\mathcal L_0)(Y) = \mathcal L_0(\lambda_{X\circ}(Y)) = \mathcal L(\lambda_{X\circ}(Y)).
\end{align*}
On the other hand, the isomorphism $\restrict{\mathcal L}Y \isom \Lambda_Y^r$ is induced by $r$ sections $s_1, \dots, s_r$ in $\mathcal L(Y)$. We get $s_1, \dots, s_r \in \mathcal L(\lambda_{X\circ}(Y))$, hence the isomorphism $\Lambda_Y^r \isoto \restrict{\mathcal L}Y$ extends to a map $\Lambda_{\lambda_{X\circ}(Y)}^r \to \restrict{\mathcal L}{\lambda_{X\circ}(Y)}$, which must then be an isomorphism (because it is so v-locally). Thus $\lambda_{X\circ}(Y)$ is a (quasi-)pro-étale cover of $X$ on which $\mathcal L$ becomes constant, as desired.
\end{proof}

The problem for showing that $\lambda^{-1}\colon \catlocsys\Lambda{X_\qproet} \to \catlocsys\Lambda{X_\vsite}$ is an equivalence is that the analogous statement of \cite[Theorem 14.12.(i)]{etale-cohomology-of-diamonds} is unknown for general quasi-pro-étale sheaves, so we cannot argue in the same way as in the beginning of the proof of \cref{rslt:et-locsys-equiv-qproet-locsys}. Instead we will have to start with the case that our local system comes from the étale site (i.e. that $\Lambda$ is discrete) and go from there as far as possible. We start with the following simple observation:

\begin{proposition} \label{rslt:qproet-locsys-equiv-v-locsys-pro-discrete-case}
Let $X$ be a diamond and let $\Lambda$ be a condensed ring. Assume that either $\Lambda$ is discrete or that $\Lambda = \varprojlim_n \Lambda_n$ for an inverse system $(\Lambda_n)_{n\in\N}$ of discrete rings. Then pullback along the morphism of sites $\lambda\colon X_\vsite \to X_\qproet$ induces an equivalence of categories
\begin{align*}
	\lambda^{-1}\colon \catlocsys\Lambda{X_\qproet} \isoto \catlocsys\Lambda{X_\vsite}.
\end{align*}
\end{proposition}
\begin{proof}
If $\Lambda$ is discrete then we can argue as in \cref{rslt:et-locsys-equiv-qproet-locsys} (noting that \cite[Theorem 14.12.(i)]{etale-cohomology-of-diamonds} also applies to sheaves on the v-site), with the second part of the proof replaced by \cref{rslt:qproet-locsys-equiv-v-locsys-helper}. Then the case $\Lambda = \varprojlim_n \Lambda_n$ follows from \cref{rslt:projlim-of-locsys}.
\end{proof}

A special case of \cref{rslt:qproet-locsys-equiv-v-locsys-pro-discrete-case} is that $\Lambda$ is a complete adic ring, i.e. $\Lambda = \varprojlim_n \Lambda/I^n$ for an open ideal $I \subset \Lambda$. This includes the important cases $\Lambda = \Z_p$ and $\Lambda = \ri_{\Cpx_p}$. However, for our applications we need to extend \cref{rslt:qproet-locsys-equiv-v-locsys-pro-discrete-case} to include more general Huber rings, i.e. topological rings $\Lambda$ which contain an adic open subring $\Lambda'$ with finitely generated ideal of definition $I$. The main new ingredient is the following result:

\begin{lemma} \label{rslt:complete-huber-ring-has-etale-locally-integral-model}
Let $X$ be a locally spatial diamond and let $\Lambda$ be a complete Huber ring with ring of definition $\Lambda' \subset \Lambda$. Let $\mathcal L \in \catlocsys\Lambda{X_\vsite}$. Then there exists an étale cover $U \surjto X$ and a $\Lambda'$-local system $\mathcal L_U' \in \catlocsys{\Lambda'}{U_\vsite}$ such that $\restrict{\mathcal L}U = \mathcal L_U' \tensor_{\Lambda'} \Lambda$.
\end{lemma}
\begin{proof}
Let $\mathcal F$ be the sheaf of sets on $X_\vsite$ which associates to every $U \in X_\vsite$ the set of all $\Lambda'$-local systems $\mathcal L_U' \subset \restrict{\mathcal L}U$ on $U_\vsite$ such that $\mathcal L_U' \tensor_{\Lambda'} \Lambda = \mathcal L$. We claim that locally on $X_\vsite$, $\mathcal F$ is isomorphic to the constant sheaf associated to the discrete set $S_r = \GL_r(\Lambda)/\GL_r(\Lambda')$ for some $r \ge 0$. This claim is local on $X_\vsite$ so that we can assume that $X$ is w-contractible (cf. \cref{rslt:w-contractible-basis}) and $\mathcal L = \Lambda^r$ for some $r \ge 0$ (explicitly this means that $\mathcal L(X) = \cts(\abs X, \Lambda^r)$ by \cref{rslt:explicit-computation-of-S-X}). We will show that then $\mathcal F = (S_r)_X$.

There is an obvious injective map $(S_r)_X \to \mathcal F$ given by $g \mapsto g \Lambda'^r$, so it only remains to show surjectivity of that map, which we can verify on global sections (because w-contractible spaces form a basis of $X_\vsite$). Let $\mathcal L' \in \mathcal F(X)$ be given. By \cref{rslt:qproet-locsys-equiv-v-locsys-pro-discrete-case} $\mathcal L'$ is quasi-pro-étale locally constant, so the w-contractibility of $X$ implies that $\mathcal L'$ is constant, necessarily of rank $r$. Fix any $x \in X$ and let $\mathcal L'_x$ and $\mathcal L_x$ denotes the fibers at $x$. Then $\mathcal L_x = \Lambda^r$ and $\mathcal L'_x \isom \Lambda'^r$, so that after replacing $\mathcal L'$ by $g\mathcal L'$ for a suitable $g \in S_r$ we can assume $\mathcal L'_x = \Lambda'^r \subset \Lambda^r$. Now choose $r$ sections $s_1, \dots, s_r \in \mathcal L'(X)$ which form a basis of $\mathcal L'$. Then $s_i(x) \in \Lambda'^r$ for all $i$ and since $\Lambda'$ is open inside $\Lambda$, there exists an open neighbourhood $U$ of $x$ such that $s_i(y) \in \Lambda'^r$ for all $y \in U$. Moreover, one checks easily that the units $(\Lambda')\units \subset \Lambda'$ are open. Thus by looking at the determinant of $s_1, \dots, s_r$ on the points of $U$ one sees that after possibly shrinking $U$, $s_1(y), \dots, s_r(y)$ form a basis of $\Lambda'^r$ for all $y \in U$. But then $\restrict{\mathcal L'}U = \Lambda'^r_U \subset \Lambda_U^r$ is trivial. Repeating the same argument for other points $x$ we deduce that there is an open cover over $X$ and sections of $(S_r)_X$ on that cover which trivialize $\mathcal L'$; it follows easily that $\mathcal L'$ lies in the image of $S_r(X) \to \mathcal F(X)$, as desired.

To finish the proof, let $X$ be any locally spatial diamond. From the previous paragraph we deduce that there is a v-covering $Y \to X$ such that $\restrict{\mathcal F}Y$ is the pullback of an étale sheaf. Then \cite[Theorem 14.12.(i)]{etale-cohomology-of-diamonds} implies that $\mathcal F$ is the pullback $\mathcal F = \lambda^{-1} \nu^{-1} \mathcal F_0$ of an étale sheaf $\mathcal F_0$. We have seen that $\mathcal {F} (V) $ is non-empty for every $w$-contractible $V \to X$. This implies that the stalks of $\mathcal F_0$ are non-empty. Hence  there exists  an étale cover $U \to X$ such that $\mathcal F_0(U) \ne \emptyset$. Then any $\mathcal L_U' \in \mathcal F_0(U) = \mathcal F(U)$ satisfies the claim.
\end{proof}

We finally arrive at our main result concerning the comparison of quasi-pro-étale and v-local systems:

\begin{theorem} \label{rslt:qproet-locsys-equiv-v-locsys-complete-huber-rings}
Let $X$ be a diamond and let $\Lambda$ be a complete Huber ring. Then pullback along the morphism of sites $\lambda\colon X_\vsite \to X_\qproet$ induces an equivalence of categories
\begin{align*}
	\lambda^{-1}\colon \catlocsys\Lambda{X_\qproet} \isoto \catlocsys\Lambda{X_\vsite}.
\end{align*}
\end{theorem}
\begin{proof} Recall that by \cite[Proposition 14.7]{etale-cohomology-of-diamonds} we only need to check essential surjectivity. Let $\mathcal L \in \catlocsys\Lambda{X_\vsite}$ be given. Choose any quasi-pro-étale cover $X' \to X$ by a locally spatial diamond $X'$ and then apply \cref{rslt:complete-huber-ring-has-etale-locally-integral-model} to $X'$ in order to find an étale cover $U \to X'$ such that $\restrict{\mathcal L}U$ comes from a $\Lambda'$-local system on $U_\vsite$, where $\Lambda'$ is a ring of definition for $\Lambda$. Then apply \cref{rslt:qproet-locsys-equiv-v-locsys-pro-discrete-case} to deduce that $\restrict{\mathcal L}U$ lies in the image of $\lambda_U^{-1}\colon \catlocsys\Lambda{U_\qproet} \to \catlocsys\Lambda{U_\vsite}$. Then $\mathcal L$ is obtained by glueing the quasi-pro-étale sheaf $\restrict{\mathcal L}U$ along the quasi-pro-étale cover $U \to X$, so in particular $\mathcal L = \lambda^{-1}\mathcal L_0$ for some quasi-pro-étale sheaf $\mathcal L_0$. Then \cref{rslt:qproet-locsys-equiv-v-locsys-helper} implies that $\mathcal L_0 \in \catlocsys\Lambda{X_\qproet}$.
\end{proof}

\cref{rslt:qproet-locsys-equiv-v-locsys-complete-huber-rings} is a very useful descent result on local systems, based on the new insights enabled by the theory of diamonds and the fascinating properties of the v-site. In fact, the equivalence of quasi-pro-étale and v-local systems lies at the heart of our applications: It is the reason why we can remove étaleness hypotheses everywhere.

\begin{remark}
By \cref{rslt:injlim-of-locsys} one can generalize \cref{rslt:qproet-locsys-equiv-v-locsys-complete-huber-rings} to the case where $\Lambda$ is filtered colimit of complete Huber rings (at least if $X$ is qcqs). This includes many non-complete Huber rings, for example all (non-completed) algebraic extensions of $\Q_p$ and their rings of integers.
\end{remark}

\begin{remark} \label{rmk:locsys-equiv-on-all-sites}
Recall from \cref{def:localsys} that we write $\catlocsys\Lambda{X}$ for local systems on the v-site of a small v-sheaf $X$. In view of \cref{rslt:qproet-locsys-equiv-v-locsys-complete-huber-rings} and \cref{rslt:et-locsys-equiv-qproet-locsys}, as long as we impose mild restrictions on $\Lambda$, we can equivalently work on the quasi-pro-étale site (if $X$ is a diamond) or on the étale site (if $X$ is a locally spatial diamond and $\Lambda$ is discrete). In the rest of the paper, we will often implicitely make use of this fact.
\end{remark}

As promised above we provide a more explicit definition for the constant sheaf associated to a topological ring (or more generally a topological space):

\begin{lemma} \label{rslt:F-T-is-sheaf-on-vsite}
Let $X$ be a small v-sheaf and let $T$ be a topological space. Then the presheaf
\begin{align*}
	\mathcal F_T\colon U \mapsto \cts(\abs U, T) \qquad \text{for $U \in X_\vsite$},
\end{align*}
mapping $U$ to the continuous functions from its underlying topological space to $T$, is a sheaf on $X_\vsite$.
\end{lemma}
\begin{proof}
Obviously $\mathcal F_T$ satisfies the sheaf property for any analytic covering (i.e. open covering with respect to the topology $\abs U$ for $U \in X_\vsite$), as continuity is a local property. It is therefore enough to check that $\mathcal F_T$ satisfies the sheaf property for every covering of the form $f\colon V \to U$ in $X_\vsite$. Given such a covering, let $t \in \mathcal F_T(V)$, i.e. $t$ is a continuous map $\abs V \to T$, and assume that $t \comp \abs{\pr1} = t \comp \abs{\pr2}$ for $\pr1, \pr2\colon V \cprod_U V \to V$ the two projections. We have to show that there exists a unique continuous map $s\colon \abs U \to T$ such that $t = s \comp \abs f$. By \cite[Proposition 12.10]{etale-cohomology-of-diamonds} the map $\abs{V \cprod_U V} \to \abs V \cprod_{\abs U} \abs V$ is surjective, which implies that $s$ exists and is unique as a map. It remains to show that $s$ is continuous, which follows from the fact that $\abs f\colon \abs V \to \abs U$ is a quotient map by \cite[Proposition 12.9]{etale-cohomology-of-diamonds}.
\end{proof}

\begin{corollary} \label{rslt:explicit-computation-of-S-X}
Let $X$ be a small v-sheaf and $T$ a T1 topological space. Then the associated constant sheaf $T_X$ on $X_\vsite$ is $T_X = \mathcal F_T$, i.e. we have
\begin{align*}
	T_X(U) = \cts(\abs U, T)
\end{align*}
for all $U \in X_\vsite$.
\end{corollary}
\begin{proof}
By \cref{rslt:F-T-is-sheaf-on-vsite} we know that $\mathcal F_T$ is a sheaf, so it is enough to check that $\mathcal F_T$ and $T_X$ coincide on a basis of the site $X_\vsite$. Let therefore $U \in X_\vsite$ be a strictly totally disconnected perfectoid space. Then as in the proof of \cref{rslt:constant-sheaf-commutes-with-lim-and-colim} we get
\begin{align*}
	T_X(U) = T_U(U) = T(\pi_0(\abs U)) = \cts(\pi_0(\abs U), T).
\end{align*}
Every connected component of $U$ is of the form $\Spa(C, C^+)$ for some algebraically closed perfectoid field $C$ and some open and bounded valuation subring $C^+ \subset C^\circ$. In particular, every connected component of $U$ is the closure of a generic point. But since $T$ is T1, any continuous map $f\colon \abs U \to T$ is constant on the closure of any point of $\abs U$. In particular, $f$ is constant on every connected component of $U$ and hence factors over $\pi_0(\abs U)$. We deduce $\cts(\pi_0(\abs U), T) = \cts(\abs U, T)$.
\end{proof}

We are also interested in the commutation of base change with global sections, which we will establish for discrete rings in the following.

\begin{lemma} \label{rslt:cts-from-profinite-to-discrete-is-injlim}
Let $S$ be a profinite space, $S = \varprojlim_{i\in I} S_i$ for finite discrete spaces $S_i$. Let $T$ be an arbitrary discrete set. Then
\begin{align*}
	\cts(S, T) = \varinjlim_i \cts(S_i, T).
\end{align*}
\end{lemma}
\begin{proof}
There is a natural map from right to left which is clearly injective, so we only need to check surjectivity. Let $f\colon S \to T$ be a given continuous map. Since $T$ is discrete and $S$ is compact, $f$ only takes finitely many distinct values, i.e. there is a partition $S = \bigdunion_{j=1}^n U_j$ into open subsets $U_j \subset S$ such that $\restrict f{U_j}$ is constant for all $j$. We can even assume that $U_j = \pi_{i_j}^{-1}(\{ x_j \})$ for some $i_j \in I$, the projection $\pi_{i_j}\colon S \to S_{i_j}$ and some element $x_j \in S_{i_j}$, because sets of this type form a basis of $S$. Since $I$ is directed we conclude that there is some $i \in I$ such that $U_j = \pi_i^{-1}(V_j)$ for some partition $S_i = \bigdunion_{j=1}^n V_j$. Then $f$ factors through $S \to S_i$, proving the claimed surjectivity.
\end{proof}

\begin{lemma} \label{rslt:discrete-base-change-on-qcqs-helper}
Let $X$ be a spatial small v-sheaf, let $\Lambda$ be a discrete ring and let $M$ be a discrete $\Lambda$-module. Then
\begin{align*}
	 \Gamma(X_\vsite, M) = \Gamma(X_\vsite, \Lambda) \tensor_\Lambda M.
\end{align*}
\end{lemma}
\begin{proof}
By assumption $\abs X$ is a spectral space. Thus we have to show that for every spectral space $T$ there is a natural isomorphism $\cts(T, M) = \cts(T, \Lambda) \tensor_\Lambda M$. Note that since $M$ is discrete (and in particular totally disconnected), any continuous map $T \to M$ factors uniquely over $\pi_0(T)$, i.e. we have $\cts(T, M) = \cts(\pi_0(T), M)$. By \cite[Lemma 0906]{stacks-project} $\pi_0(T)$ is profinite, so we may assume that $T$ is profinite. Write $T = \varprojlim_{i\in I} T_i$ for finite discrete sets $T_i$. For each $i \in I$ we have $\cts(T_i, M) = M^{\abs {T_i}}$ and thus evidently $\cts(T_i, M) = \cts(T_i, \Lambda) \tensor_\Lambda M$. Using \cref{rslt:cts-from-profinite-to-discrete-is-injlim} and the fact that tensor products commute with colimits we obtain
\begin{align*}
	\cts(T, M) &= \varinjlim_i \cts(T_i, M) = \varinjlim_i \cts(T_i, M) \tensor_\Lambda M = \cts(T, \Lambda) \tensor_\Lambda M. \qedhere
\end{align*}
\end{proof}

\begin{proposition} \label{rslt:discrete-base-change-on-qcqs}
Let $X$ be a qcqs small v-sheaf, let $\Lambda$ be a discrete ring and let $M$ be a flat discrete $\Lambda$-module. Then for all $\mathcal N \in D^+(X_\vsite, \Lambda)$ and all $i \ge 0$ we have
\begin{align*}
	H^i(X_\vsite, \mathcal N \tensor_\Lambda M) = H^i(X_\vsite, \mathcal N) \tensor_\Lambda M.
\end{align*}
\end{proposition}
\begin{proof}
Let $\mathcal I$ be any sheaf of $\Lambda$-modules on $X_\vsite$ and consider the presheaf $\mathcal F_{\mathcal I} = [U \mapsto \mathcal I(U) \tensor_\Lambda M]$ on $X_\vsite$. Then $\mathcal F_{\mathcal I}$ satisfies the following properties:
\begin{itemize}
	\item $\mathcal F_{\mathcal I}$ is a sheaf. To prove this we have to show that tensoring over $\Lambda$ with $M$ preserves the exactness of the first three terms of the Čech complex associated to any covering $(V_i \to U)_{i \in I}$ in $X_\vsite$. As $X$ is qcqs we can restrict to finite $I$. But then the claimed exactness follows directly from the flatness of $M$ over $\Lambda$.

	\item If $\mathcal I$ is $\Gamma$-acyclic then so is $\mathcal F_{\mathcal I}$. To prove this, by SGA 4 V Proposition 4.3 all we have to show is that tensoring over $\Lambda$ with $M$ preserves the exactness of the whole Čech complex associated to any covering $(V_i \to U)_{i\in I}$ in $X_\vsite$. This follows in the same way as above.

	\item We have $\mathcal F_{\mathcal I} = \mathcal I \tensor_\Lambda M$. To prove this, note that we have a canonical morphism $\mathcal F_{\mathcal I} \to \mathcal I \tensor_\Lambda M$. The claim is that this morphism is an isomorphism. Now $\mathcal I \tensor_\Lambda M$ is the sheafification of the presheaf $U \mapsto \mathcal I(U) \tensor_{\Lambda(U)} M(U)$. By \cref{rslt:discrete-base-change-on-qcqs-helper}, for every spatial $U \in X_\vsite$ we have
	\begin{align*}
		\mathcal I(U) \tensor_{\Lambda(U)} M(U) = \mathcal I(U) \tensor_{\Lambda(U)} (\Lambda(U) \tensor_\Lambda M) = \mathcal I(U) \tensor_\Lambda M = \mathcal F_{\mathcal I}(U).
	\end{align*}
	Since the spatial small v-sheaves form a basis of $X_\vsite$ we deduce the claimed isomorphism.
\end{itemize}
Now to prove the claim, choose a quasi-isomorphism $\mathcal N \isom \mathcal I_\bullet$ with $\mathcal I_i$ being $\Gamma$-acyclic for all $i$. Then $\mathcal N \tensor_\Lambda M$ is quasi-isomorphic to $\mathcal I_\bullet \tensor_\Lambda M$ with all $\mathcal I_i \tensor_\Lambda M$ still $\Gamma$-acyclic and with $\Gamma(X_\vsite, \mathcal I_i \tensor_\Lambda M) = \Gamma(X_\vsite, \mathcal F_{\mathcal I_i}) = \Gamma(X_\vsite, \mathcal I_i) \tensor_\Lambda M$. Therefore
\begin{align*}
	H^i(X_\vsite, \mathcal N \tensor_\Lambda M) &= H^i(\Gamma(X_\vsite, \mathcal I_\bullet \tensor_\Lambda M)) = H^i(\Gamma(X_\vsite, \mathcal I_\bullet) \tensor_\Lambda M) \\& = H^i(\Gamma(X_\vsite, \mathcal I_\bullet)) \tensor_\Lambda M = H^i(X_\vsite, \mathcal N) \tensor_\Lambda M,
\end{align*}
as desired.
\end{proof}

\begin{example}
Our main application of \cref{rslt:discrete-base-change-on-qcqs} is the case where $\Lambda = \Fld_q$, $M = \ri/p$ and $\mathcal N = \mathcal L$ is an $\Fld_q$-local system on $X$. Then the statement is
\begin{align*}
	H^i(X_\vsite, \mathcal L \tensor_{\Fld_q} \ri/p) = H^i(X_\vsite, \mathcal L) \tensor_{\Fld_q} \ri/p,
\end{align*}
i.e. it does not matter whether base extension from $\Fld_q$ to $\ri/p$ is done before or after taking cohomology.
\end{example}

Let us also state a classical result about étale local systems with finite coefficient rings:

\begin{lemma} \label{rslt:locsys-trivial-on-fet-cover}
Let $X$ be a spatial diamond, let $\Lambda$ be a finite discrete ring and let $\mathcal L$ be a $\Lambda$-local system on $X$ of constant rank $r$. Then there exists a finite étale cover $Z \to X$ such that $\restrict{\mathcal L}Z$ is constant.
\end{lemma}
\begin{proof}
This follows in the same way as in the theory of schemes, because diamonds satisfy descent of finite étale morphisms by \cite[Proposition 10.11.(iii)]{etale-cohomology-of-diamonds}. For the convenience of the reader we expand the argument.

Let $\mathcal F$ be the sheaf on $X_\et$ which associates to any $U \in X_\et$ the set of $\Lambda$-linear isomorphisms of $\restrict{\mathcal L}U$ with $\Lambda_U^r$. Then it is clear that $\mathcal F$ is a finite locally constant sheaf;\footnote{By a \emph{finite locally constant sheaf} on $X_\et$ we mean the obvious definition, i.e. it is a sheaf of sets on $X_\et$ which locally on $X_\et$ is the constant sheaf associated to a finite set as defined in \cref{def:constant-etale-sheaves}.} we show that it is represented by a finite étale cover of $X$. By assumption there exists an étale cover $U \to X$ such that $\restrict{\mathcal F}U$ is constant. In particular, $\restrict{\mathcal F}U$ is represented by a finite étale diamond $Z_U$ over $U$ (namely $\abs{\GL_r(\Lambda)}$ copies of $U$) and the glueing datum of $\mathcal F$ along $U \to X$ induces a descent datum for $Z_U$ along $U \to X$: We are given an isomorphism $\varphi\colon Z_U \cprod_X U \isoto U \cprod_X Z_U$ over $U \cprod_X U$ which satisfies the cocycle condition. Recall that all diamonds are sheaves for the v-topology on char-$p$ perfectoid spaces; we can thus define a v-sheaf $Z$ by
\begin{align*}
	Z := \coeq\left(Z_U \cprod_X U \xrightrightarrows[\pr1 \comp \varphi]{\pr0} Z_U\right).
\end{align*}
Then $Z$ comes equipped with a morphism to $X$ and we claim that it is the descent of $Z_U$ along $U \to X$. Indeed, the cocycle condition implies that there is a natural isomorphism $Z \cprod_X U = Z_U$. Furthermore, \cite[Proposition 10.11.(iii)]{etale-cohomology-of-diamonds} implies that $Z \to X$ is a finite étale morphism and in particular that $Z$ is a (spatial) diamond (see \cite[Lemma 11.21]{etale-cohomology-of-diamonds}). By construction $Z$ represents the sheaf $\mathcal F$ on $X_\vsite$. But then $\mathcal F(Z)$ is non-empty and hence there is an isomorphism $\restrict{\mathcal L}Z \isom \Lambda_Z^r$.
\end{proof}

\begin{corollary}\label{rslt:profinet-cover}
Let $E$ be a completed algebraic extension of $\Q_p$ with ring of integers $\ri_E$. Let $X$ be a connected spatial diamond and let $\mathcal L$ be an $\ri_E$-local system on $X$. Then there exists a connected quasi-pro-étale cover $Y = \varprojlim Y_n \to X$ for finite étale covers $Y_n$ of $X$ such that $\restrict{\mathcal L}Y$ is constant.
\end{corollary}
\begin{proof}
By the same argument as in the proof of  \cref{rslt:et-locsys-equiv-qproet-locsys}) there is a disjoint clopen cover of $X$ on which $\mathcal L$ has constant rank. Since $X$ is connected this means that $\mathcal L$ must have constant rank.  By \cref{rslt:injlim-of-locsys}, there exists a finite extension $F$ of $\mathbb{Q}_p$ such that $\mathcal L/p$ is induced by an $\ri_F/p$-local system. Thus by \cref{rslt:locsys-trivial-on-fet-cover} there is a finite étale cover $Y_1$ of $X$ such that $\mathcal L/p$ becomes constant on $Y_1$. By \cref{rslt:fet-cover-of-connected-space-has-finitely-many-components} below, we can replace $Y_1$ by one of its connected components and hence assume that $Y_1$ is connected. In the same way we can construct a tower $\dots \to Y_n \to \dots \to Y_1 \to X$ with all $Y_n \to X$ finite étale, all $Y_n$ connected and $\restrict{\mathcal L/p^n}{Y_n}$ being constant. Letting $Y = \varprojlim Y_n$ and using \cref{rslt:projlim-of-locsys} we see that $\restrict{\mathcal L}Y$ is constant. Moreover, $\abs Y = \varprojlim \abs{Y_n}$ by \cite[Lemma 11.22]{etale-cohomology-of-diamonds} and a cofiltered limit of qcqs connected topological spaces is connected (cf. \cite[Remark 3.2]{matti-rigid-DW-functor}).
\end{proof}

\begin{lemma} \label{rslt:fet-cover-of-connected-space-has-finitely-many-components}
Let $f\colon Y \to X$ be a finite étale morphism of small v-sheaves. If $X$ is connected, then $Y$ has only finitely many connected components and $f$ restricted to any connected component of $Y$ is surjective.
\end{lemma}
\begin{proof}
Since $f$ is finite étale, $\abs f\colon \abs Y \to \abs X$ is open and closed (by the definition of the topology on small v-sheaves it is enough to show that the pullback of $f$ to a perfectoid space is open and closed; then openness follows from \cite[Lemma 1.7.9]{huber-etale-cohomology} and closedness follows e.g. from the valuative criterion of properness, \cite[Proposition 18.6]{etale-cohomology-of-diamonds}). Assume that $Y$ has infinitely many connected components. Then for every $n \ge 1$ one can decompose $\abs Y$ into $n$ disjoint clopen subsets. But the image under $f$ of each of these subsets is clopen in $X$, hence all of $X$. This implies that the fiber of $\abs f$ over any point $x \in \abs X$ contains at least $n$ points, for all $n$, i.e. all fibers are infinite. On the other hand, the finiteness of $f$ implies that its fibers must be finite, contradiction.
\end{proof}

We get an immediate application to the category of local systems with integral model, which plays an important role in our applications.

\begin{definition}\label{def:rep-integral-model}
Let $E$ be a completed algebraic extension of $\Q_p$ with ring of integers $\ri_E$. Let $X$ be a small v-sheaf and let $L$ be a $E$-local system on $X$. An \emph{integral model} of $L$ is an $\ri_E$-local system $\mathcal L \subset L$ such that $L = \mathcal L[p^{-1}]$. The full subcategory of $\catlocsys EX$ of those local systems which have an integral model is denoted $\catlocsysint EX$.
\end{definition}

\begin{theorem} \label{rslt:ILoc-is-abelian}
Let $E$ be a completed algebraic extension of $\Q_p$ with ring of integers $\ri_E$. Let $X$ be a connected spatial diamond. Then $\catlocsysint EX$ is an abelian subcategory of the category of all sheaves of $E$-modules on $X_\vsite$.
\end{theorem}
\begin{proof}
We only have to show that kernels and cokernels of morphisms in $\catlocsysint EX$  (which exist as sheaves of $E$-modules on $X_\vsite$) are objects in $\catlocsysint EX$. Let $L$ and $L''$ be $E$-local systems with integral models on $X$, and let $\varphi: L \to L''$ be a morphism of $E$-local systems. The local systems $L$ and $L''$  have  integral model $\mathcal L$ and $\mathcal L''$, respectively.  We denote by $\psi: \mathcal L \to L''$ the restriction of $\varphi$ to $\mathcal L$.

Let $\mathcal L'$ be the kernel of $\psi$ as a sheaf on $X_\vsite$. We want to show that $\mathcal L'$ is a local system of $\ri$-modules on $X$. Arguing as in \cref{rslt:profinet-cover} simultaneously for the covers trivializing $\mathcal L / p^n$ and $\mathcal L'' / p^n$ we find a connected quasi-pro-étale cover $Y = \varprojlim_n Y_n$ of $X$ where both $\mathcal L$ and $\mathcal L''$ are constant.  Hence on this cover the morphism $\varphi$ is given by an $E$-linear map $\varphi_0: E^r \to E^s$ (see \cref{rslt:explicit-computation-of-S-X}), and we may assume that $\mathcal L$ is the constant sheaf given by $\ri_E^r \subset E^r$. Then the morphism $\restrict\psi Y$ is induced by the module homomorphism $\psi_0$ we get by restricting $\varphi_0$ to $\ri_E^r$.

By definition, $\restrict{\mathcal L'}Y$ is the constant sheaf induced by the $\ri_E$-module $N = \ker(\psi_0)$. Note that the image of $\psi_0$ is a finitely generated and torsion-free $\ri_E$-module, hence free of finite rank by \cite[Proposition 2.8 (i)]{rigid-p-adic-hodge}, and in particular of finite presentation. Therefore $N$ is also finitely generated, which again implies free of finite rank since it is torsion-free. Hence $\mathcal L'$ is indeed a local system of $\ri_E$-modules. After inverting $p$ we find that $\mathcal L'[p^{-1}]$ coincides with the kernel of $\varphi$, which consequently lies in $\catlocsysint EX$.

The corresponding statement about the cokernels follows easily by duality: Let $L \to L'' \to M \to 0$ be an exact sequence of sheaves of $E$-modules such that $L, L'' \in \catlocsysint EX$. On a connected quasi-pro-étale cover on which $L$ and $L''$ are constant (constructed as in the first part of the proof), $M$ is constant as well. In particular, if $(-)^*$ denotes $\IHom(-, E)$ then $M^{**} = M$. One sees easily that $(-)^*$ restricts to a functor on $\catlocsysint EX$. But $M^*$ is the kernel of $L''^* \to L^*$, hence $M^* \in \catlocsysint EX$. Then $M = (M^*)^* \in \catlocsysint EX$ as well.
\end{proof}

\section{Computing Global Sections} \label{sec:globsec}

We will now compute the global sections of $\ri$-local systems on a proper adic space $X$ of finite type over $\Cpx_p$. Our results heavily relie on Scholze's result for the cohomology of $\Fld_p$-local systems in \cite[Theorem 5.1]{rigid-p-adic-hodge}; in fact, all we do is extend the coefficients from $\Fld_p$ to $\ri$ and remove the smoothness hypothesis on $X$ using a resolution of singularities argument (a version of the latter is already found in \cite[Theorem 3.13]{perfectoid-spaces-survey}).

\begin{remark}
In the following, we will always assume that $X$ is proper and \emph{of finite type} over $\Cpx_p$. We define a morphism to be proper, if it is quasicompact, separated and universally closed. This definition is analogous to the notion of properness for diamonds. Note however that it differs from Huber's definition  \cite[Definition 1.3.2]{huber-etale-cohomology}, where the condition of ${}^+$weakly finite type is added for properness. This in turn is more general than our condition of finite type. In fact, Huber's definition allows canonical compactifications: Given an affinoid adic space $X = \Spa(A, A^+)$ of (${}^+$weakly) finite type over $\Cpx_p$, the canonical compactification of $X$ is $\tclos X = \Spa(A, A'^+)$, where $A'^+$ is the completed integral closure of $\Cpx_p + A^{\circ\circ}$ in $A$. This example immediately shows that the following results on global sections do not hold for such $X$, so that we really need to restrict to finite type morphisms.
\end{remark}

We start with the following mild generalization of Scholze's fundamental comparison result \cite[Theorem 5.1]{rigid-p-adic-hodge}. Since it might be useful in other contexts, we provide the details of the proof.

\begin{lemma} \label{rslt:F-q-cohomology-on-proper-space}
Let $X$ be a proper adic space of finite type over $\Cpx_p$, let $q$ be a power of $p$ and let $\mathcal L_1$ be an $\Fld_q$-local system on $X$. Then $H^i(X_\vsite, \mathcal L_1)$ is finite over $\Fld_q$ and the natural map $\mathcal L_1 \to \mathcal L_1 \tensor_{\Fld_q} \check\ri^+_X/p$ induces almost isomorphisms
\begin{align*}
	H^i(X_\vsite, \mathcal L_1) \tensor_{\Fld_q} \ri^a/p \isoto H^i(X_\vsite, \mathcal L_1 \tensor_{\Fld_q} \check\ri_X^{+a}/p)
\end{align*}
for all $i \ge 0$.
\end{lemma}
\begin{proof}
Let us first assume that $X$ is smooth. By \cref{rslt:locfree-sheaf-fully-faithful-qproet-to-v} and \cite[Proposition 14.7]{etale-cohomology-of-diamonds} (together with \cref{rslt:qproet-locsys-equiv-v-locsys-complete-huber-rings}) we can compute all cohomologies on the quasi-pro-étale (or even the étale) site of $X$. In the case $q = p$ the claim is \cite[Theorem 5.1]{rigid-p-adic-hodge}. The proof of that theorem generalises directly to the case $q = p^n$ (and adapts to the minor difference between the quasi-pro-étale site used in the present paper and the pro-étale site used in the reference), which we outline in the following.

First note that Lemmas 5.2 through 5.5 and 5.10 of \cite{rigid-p-adic-hodge} do not depend on $q$ or on the difference between the quasi-pro-étale and the pro-étale site. Lemma 5.6 still holds if we replace $\Fld_p$ by $\Fld_q$ everywhere: The only place where it matters is in the application of Lemma 4.12; the latter is true for $\Fld_p$ replaced by $\Fld_q$ as it only uses $\mathcal L_1$ to get a finite étale cover of $X$ on which it becomes trivial. Note also that the arguments of Lemma 4.12 work the same on the quasi-pro-étale site, notably because $\hat\ri^+_X$ is still almost acyclic on affinoid perfectoids by \cite[Proposition 8.5.(iii)]{etale-cohomology-of-diamonds}. Now it follows that Lemma 5.8 of \cite{rigid-p-adic-hodge} is true with $\Fld_p$ replaced by $\Fld_q$, as this difference only matters in the application of Lemma 5.6.

Next we observe that Lemma 2.12 of \cite{rigid-p-adic-hodge} still holds if we replace $p$ by $q$ and the Frobenius $\varphi$ by the Frobenius $\varphi_q$ over $\Fld_q$ in condition (iii): The only part of the proof that needs a slight adjustment is the following: Let $K$ be a separably closed field over $\Fld_q$ and let $V$ be a finite-dimensional $K$-vector space together with an isomorphism $V \tensor_{K,\varphi_q} K$. Then $V$ is trivial as a $\varphi$-module, i.e. admits a $K$-basis fixed by $\varphi$. This statement can be found for example in \cite[Proposition 3.2.4]{schneider_2017}.

From all of the above we deduce that the first part of the proof of \cite[Theorem 5.1]{rigid-p-adic-hodge} works the same way with $\Fld_p$ replaced by $\Fld_q$. We arrive at $H^i(X_\qproet, \mathcal L_1 \tensor_{\Fld_q} \hat\ri_{X^\flat}) \isom (\Cpx_p^\flat)^r$, compatibly with the action of the Frobenius $\varphi_q$. As in the reference we use the Artin-Schreier sequence $0 \to \mathcal L_1 \to \mathcal L_1 \tensor_{\Fld_q} \hat\ri_{X^\flat} \to \mathcal L_1 \tensor_{\Fld_q} \hat\ri_{X^\flat} \to 0$, where the second map is $v \tensor f \mapsto v \tensor (f^q - f)$. This sequence is exact, by the (slightly generalised) argument of the reference. As in the reference we deduce
\begin{align*}
	H^i(X_\qproet, \mathcal L_1) = H^i(X_\qproet, \mathcal L_1 \tensor_{\Fld_q} \hat\ri_{X^\flat})^{\varphi_q=1} = ((\Cpx_p^\flat)^r)^{\varphi_q=1} = \Fld_q^r.
\end{align*}
This finishes the proof in the smooth case.

If $X$ is not smooth but $\mathcal L_1 = \Fld_q$ then we can deduce the claim from the smooth case by resolution of singularities, in exactly the same way as in \cite[Theorem 3.17]{perfectoid-spaces-survey}. For general $\mathcal L_1$, we argue similar to \cite[Theorem 3.13]{perfectoid-spaces-survey}: Let $f\colon Y \to X$ be a finite étale cover such that $f^{-1}\mathcal L_1 = \restrict{\mathcal L_1}Y$ is constant (this exists either by the classical theory of adic spaces or by \cref{rslt:locsys-trivial-on-fet-cover}). For $k \ge 1$, let $Y_k$ denote the $k$-fold fiber product of $Y$ over $X$ and let $f_k\colon Y_k \to X$ be the projection. Then we have an exact sequence of sheaves on $X_\vsite$:
\begin{align*}
	0 \to \mathcal L_1 \to f_{1*}f_1^{-1}\mathcal L_1 \to f_{2*}f_2^{-1}\mathcal L_1 \to f_{3*}f_3^{-1}\mathcal L_1 \to \dots
\end{align*}
Tensoring with $\check\ri^+_X/p$ and using the flatness of $\check\ri^+_X/p$ (see \cref{rslt:hat-O-flat-over-O-E}) we obtain the exact sequence of sheaves
\begin{align*}
	0 \to \mathcal L_1 \tensor_{\Fld_q} \check\ri^+_X/p \to f_{1*}f_1^{-1}\mathcal L_1 \tensor_{\Fld_q} \check\ri^+_X/p \to f_{2*}f_2^{-1}\mathcal L_1 \tensor_{\Fld_q} \check\ri^+_X/p \to \dots
\end{align*}
By a classical spectral sequence argument, in order to prove the desired almost isomorphism $H^i(X_\vsite, \mathcal L_1) \tensor_{\Fld_q} \ri^a/p \isoto H^i(X_\vsite, \mathcal L_1 \tensor_{\Fld_q} \check\ri_X^{+a}/p)$ it is enough to prove the almost isomorphisms
\begin{align*}
	H^i(X_\vsite, f_{k*}f_k^{-1}\mathcal L_1) \tensor_{\Fld_q} \ri^a/p \isoto H^i(X_\vsite, f_{k*}f_k^{-1}\mathcal L_1 \tensor_{\Fld_q} \check\ri_X^{+a}/p)
\end{align*}
for all $i \ge 0$ and $k \ge 1$. We have $f_{k*}f_k^{-1}\mathcal L_1 \tensor_{\Fld_q} \check\ri_X^{+a}/p = f_{k*}(f_k^{-1}\mathcal L_1 \tensor_{\Fld_q} \check\ri_{Y_k}^{+a}/p)$ by \cite[Proposition 5.5.1.(iv)]{huber-etale-cohomology}. Furthermore, since $f_k$ is finite, $f_{k*}$ is exact (by \cite[Proposition 2.6.4]{huber-etale-cohomology}), so that for every sheaf $\mathcal F$ on $Y_{k\vsite}$ we have $H^i(X_\vsite, f_{k*}\mathcal F) = H^i(Y_{k\vsite}, \mathcal F)$. This reduces the claim to proving the almost isomorphism
\begin{align*}
	H^i(Y_{k\vsite}, f_k^{-1}\mathcal L_1) \tensor_{\Fld_q} \ri^a/p \isoto H^i(Y_{k\vsite}, f_k^{-1}\mathcal L_1 \tensor_{\Fld_q} \check\ri_{Y_k}^{+a}/p)
\end{align*}
for all $i \ge 0$ and $k \ge 1$. But this follows from the case $\mathcal L_1 = \Fld_q$ discussed above, because $f_k^{-1}\mathcal L_1$ is constant.
\end{proof}

\begin{corollary} \label{rslt:O-F-mod-pi-n-cohomology-on-proper-space}
Let $X$ be a proper adic space of finite type over $\Cpx_p$, let $F$ be a finite extension of $\Q_p$ with ring of integers $\ri_F$ and uniformizer $\pi \in \ri_F$. For any $n > 0$, let $\mathcal L_n$ be an $\ri_F/\pi^n$-local system on $X_\vsite$. Then $H^i(X_\vsite, \mathcal L_n)$ is finite over $\ri_F/\pi^n$ and the natural map $\mathcal L_n \to \mathcal L_n \tensor_{\ri_F} \check\ri^+_X$ induces almost isomorphisms
\begin{align*}
	H^i(X_\vsite, \mathcal L_n) \tensor_{\ri_F} \ri^a \isoto H^i(X_\vsite, \mathcal L_n \tensor_{\ri_F} \check\ri_X^{+a})
\end{align*}
for all $i \ge 0$.
\end{corollary}
\begin{proof}
We use induction on $n$. The case $n = 1$ is \cref{rslt:F-q-cohomology-on-proper-space}. Now assume $n > 1$ and let $\mathcal L_1 = \mathcal L_n/\pi$ and $\mathcal L_{n-1} = \mathcal L_n/\pi^{n-1}$. Then $\mathcal L_j$ is a $\ri_F/\pi^j$-local system for $j \in \{ 1, n-1, n \}$ and there is a short exact sequence $0 \to \mathcal L_{n-1} \xto{\pi} \mathcal L_n \to \mathcal L_1 \to 0$. By \cref{rslt:hat-O-flat-over-O-E}, tensoring with $\check\ri^{+a}_X/\pi^n$ preserves almost exactness. By the long exact sequence of cohomology we obtain the following commutative diagram with almost exact rows (abbreviating $- \tensor_{\ri_F} \check\ri^{+a}_X$ by $(-)^{\ri^+_X}$ and $- \tensor_{\ri_F} \ri^a$ by $()^\ri$):
\begin{center}\begin{tikzcd}[column sep=small]
	H^{i-1}(X, \mathcal L_1)^\ri \arrow{r} \arrow{d} & H^i(X, \mathcal L_{n-1})^\ri \arrow{r}\arrow{d} & H^i(X, \mathcal L_n)^\ri \arrow{r}\arrow{d} & H^i(X, \mathcal L_1)^\ri \arrow{r}\arrow{d} & H^{i+1}(X, \mathcal L_{n-1})^\ri \arrow{d}\\
	H^{i-1}(X, \mathcal L^{\ri^+_X}_1) \arrow{r} & H^i(X, \mathcal L^{\ri^+_X}_{n-1}) \arrow{r} & H^i(X, \mathcal L^{\ri^+_X}_n) \arrow{r} & H^i(X, \mathcal L^{\ri^+_X}_1) \arrow{r} & H^{i+1}(X, \mathcal L^{\ri^+_X}_{n-1})
\end{tikzcd}\end{center}
By induction, all vertical arrows except the middle one are almost isomorphisms. Hence by the Five Lemma we deduce that the middle arrow is an almost isomorphism as well.
\end{proof}

\begin{proposition} \label{rslt:O-E-cohomology-on-proper-space}
Let $X$ be a proper adic space of finite type over $\Cpx_p$ and let $\varpi \in \ri$ be either a pseudo-uniformizer or $\varpi = 0$. Let $\mathcal L$ be an $\ri/\varpi$-local system on $X_\vsite$. Then the natural map $\mathcal L \to \mathcal L \tensor_\ri \check\ri^+_X$ induces almost isomorphisms
\begin{align*}
	H^i(X_\vsite, \mathcal L^a) \isoto H^i(X_\vsite, \mathcal L^a \tensor_{\ri^a} \check\ri_X^{+a})
\end{align*}
for all $i \ge 0$. Moreover, if $\varpi \ne 0$ then $H^i(X_\vsite, \mathcal L)$ is a finitely presented $\ri$-module.
\end{proposition}
\begin{proof}
Let us first assume $\varpi \ne 0$. Then (after possible multiplying $\varpi$ by a unit), \cref{rslt:injlim-of-locsys} implies that $\mathcal L$ is of the form $\mathcal L = \mathcal L_F \tensor_{\ri_F} \ri$ for some finite extension $F$ of $\Q_p$ and some $\ri_F/\varpi$-local system $\mathcal L_F$. Now $\varpi$ is some power of a uniformizer in $\ri_F$ and hence by \cref{rslt:O-F-mod-pi-n-cohomology-on-proper-space} we have almost isomorphisms $H^i(X, \mathcal L_F) \tensor_{\ri_F} \ri^a \isoto H^i(X, \mathcal L_F \tensor_{\ri_F} \check\ri_X^{+a})$ for all $i \ge 0$. By \cref{rslt:discrete-base-change-on-qcqs} we can pull in the tensor product on the left-hand side, which produces the desired statement (including the claim about finite presentation).

Now let $\varpi = 0$, i.e. $\mathcal L$ is an $\ri$-local system. As in the proof of \cref{rslt:F-q-cohomology-on-proper-space} we can as well work on the quasi-pro-étale site instead of the v-site. Then $\mathcal L$ is $p$-adically complete, i.e. $\mathcal L = \varprojlim_n \mathcal L/p^n$. In fact, this is even true in the derived sense, i.e. $\mathcal L = R\varprojlim_n \mathcal L/p^n$: The site $X_\vsite$ is replete by \cref{rslt:w-contractible-basis} so that the assertion follows from \cite[Proposition 3.1.10]{proetale-topology}. Now cohomology commutes with derived limits (see \cite[Lemma 0BKP]{stacks-project}) and hence, using the case $\varpi \ne 0$ discussed above we get an almost isomorphism
\begin{align*}
	& R\Gamma(X, \mathcal L) = R\Gamma(X, R\varprojlim_n \mathcal L/p^n) = R\varprojlim_n R\Gamma(X, \mathcal L/p^n) \to R\varprojlim_n R\Gamma(X, \mathcal L/p^n \tensor_\ri \hat\ri_X^+) \\&\qquad = R\Gamma(X, R\varprojlim_n (\mathcal L/p^n \tensor_\ri \hat\ri_X^+)).
\end{align*}
The $R\varprojlim$ on the right is just a $\varprojlim$ (by the same argument as above) and hence we are reduced to showing a natural isomorphism
\begin{align*}
	\varprojlim_n (\mathcal L/p^n \tensor_\ri \hat\ri_X^+) = (\varprojlim_n \mathcal L/p^n) \tensor_\ri \hat\ri_X^+.
\end{align*}
There is a natural map from right to left and thus it suffices to prove the claimed isomorphism locally. But $\mathcal L$ is locally constant, so we reduce to the case $\mathcal L = \ri$. Then the claim follows from the fact that both $\ri$ and $\hat\ri^+_X$ are $p$-adically complete.
\end{proof}

We now have established a powerful relation between the almost cohomology of a local system and its tensor product with $\check\ri^+_X$. For the construction of the main functor in this paper we only need this statement for the $0$-th cohomology, i.e. the global sections. However, we require the almost isomorphism to be an actual isomorphism, which we will prove now (in spirit this is similar to \cref{rslt:O+-mod-p-on-vsite}, on which it relies).

\begin{theorem} \label{rslt:global-sections-on-proper-space}
Let $X$ be a proper adic space of finite type over $\Cpx_p$ and let $\varpi \in \ri$ be a pseudo-uniformizer. Let $\mathcal L_\varpi$ be an $\ri/\varpi$-local system, $\mathcal L$ an $\ri$-local system and $L$ a $\Cpx_p$-local system with integral model. Then we have natural isomorphisms
\begin{align*}
	\Gamma(X_\vsite, \mathcal L_\varpi) &= \Gamma(X_\vsite, \mathcal L_\varpi \tensor_\ri \check\ri^+_X),\\
	\Gamma(X_\vsite, \mathcal L) &= \Gamma(X_\vsite, \mathcal L \tensor_\ri \check\ri^+_X),\\
	\Gamma(X_\vsite, L) &= \Gamma(X_\vsite, L \tensor_{\Cpx_p} \check\ri_X),
\end{align*}
In particular we have
\begin{align*}
	\Gamma(X_\vsite, \ri/\varpi) = \Gamma(X_\vsite, \check\ri^+_X/\varpi), \quad \Gamma(X_\vsite, \ri) = \Gamma(X_\vsite, \check\ri^+_X), \quad \Gamma(X_\vsite, \Cpx_p) = \Gamma(X_\vsite, \check\ri_X).
\end{align*}
\end{theorem}
\begin{proof}
The second isomorphism follows from the first one by taking the inverse limit over $\varpi = p^n$ (as global sections commute with inverse limits) and the third isomorphism follows from the second one by using the integral model and inverting $p$ (as $X$ is qcqs, cf. \cite[Lemma 0738]{stacks-project}). Thus we only have to prove the first isomorphism. Let $M = \Gamma(X_\vsite, \mathcal L_\varpi)$ and $N = \Gamma(X_\vsite, \mathcal L_\varpi \tensor_\ri \check\ri^+_X)$. By \cref{rslt:O-E-cohomology-on-proper-space} $M$ is finitely presented and there is a natural almost isomorphism $M \to N$.

We claim that $N$ contains no non-trivial almost zero elements. To see this we can assume that $X$ is connected. By \cref{rslt:injlim-of-locsys} and \cref{rslt:locsys-trivial-on-fet-cover} there is a finite étale cover $X' \to X$ such that $\restrict{\mathcal L_\varpi}{X'}$ is constant. Then $N$ injects into $\Gamma(X'_\vsite, \restrict{\mathcal L_\varpi}{X'} \tensor_\ri \check\ri^+_{X'})$ so that we can replace $X$ by $X'$ and hence assume that $\mathcal L_\varpi$ is constant; we then immediately reduce to $\mathcal L_\varpi = \ri/\varpi$, i.e. $N = \Gamma(X_\vsite, \check\ri^+_X/\varpi)$. By \cref{rslt:O+-mod-p-on-vsite} and \cref{rslt:comparison-of-adic-and-diamond-world} we have
\begin{align*}
	N = \Gamma(X_\vsite, \check\ri^+_X/\varpi) = \Gamma(X_\qproet, \hat\ri^+_X/\varpi) = \Gamma(X_\et, \ri^+_X/\varpi).
\end{align*}
Consider the presheaf $\mathcal F\colon U \mapsto \ri^+_X(U)/\varpi$ on $X_\et$. From the $\varpi$-torsion freeness of $\ri^+_X$ one deduces easily that $\mathcal F$ is separated. Since $\ri^+_X/\varpi$ is the sheafification of $\mathcal F$, \cite[Theorem 00WB]{stacks-project} shows that the natural map $\mathcal F \to \ri^+_X/\varpi$ is injective. Now let $s \in \Gamma(X_\et, \ri^+_X/\varpi)$ be almost zero. By \cite[Lemma 00WK]{stacks-project} there is an étale covering $(U_i \to X)_{i\in I}$ such that for every $i$, $\restrict s{U_i}$ lies in $\mathcal F(U_i) \subset (\ri^+_X/\varpi)(U_i)$. The fact that $s$ is almost zero implies that $\restrict s{U_i}$ is almost zero for all $i$. Hence it is enough to show that $\mathcal F(U_i)$ does not contain any non-trivial almost zero elements. By refining the covering if necessary, we can assume that all $U_i$ are affinoid.

It remains to show the following: Let $(A, A^+)$ be an adic $\Cpx_p$-algebra of topologically finite type; then $A^+/\varpi$ has no non-trivial almost zero elements. By definition of topologically finite type we have $A^+ = A^\circ$, the ring of power-bounded elements in $A$. Let $\overline a \in A^\circ/\varpi$ be almost zero and choose any lift $a \in A^\circ$.

For every multiplicative seminorm $\alpha\colon A \to \R_{\ge0}$ extending the absolute value on $\mathbb{C}_p$ we have $\alpha(a) \le \abs{\varpi}$ (because the fact that $\overline a$ is almost zero implies that $\alpha(m a) \le \abs{\varpi}$ for all $m \in \mm$). Hence for the supremum seminorm $\abs{\cdot}_\supn$ (as defined in \cite[Lemma 6.2.1/1]{nla.cat-vn772374}) we deduce $\abs{a}_\supn \le \abs{\varpi}$, i.e. $\abs{a/\varpi}_\supn \le 1$. But then \cite[Proposition 6.2.3/1]{nla.cat-vn772374} implies that $a/\varpi \in A^\circ$, i.e. $\overline a = 0$.

We have shown that $N$ has no non-trivial almost zero elements. We can thus deduce the claimed isomorphism $M = N$ from \cref{rslt:almost-iso-implies-iso-on-fin-pres-no-alm-zero} below.
\end{proof}

\begin{lemma} \label{rslt:almost-iso-implies-iso-on-fin-pres-no-alm-zero}
Let $f\colon M \to N$ be an almost isomorphism of $\ri$-modules. If $M$ is finitely presented and $N$ has no non-zero almost zero elements, then $f$ is an isomorphism.
\end{lemma}
\begin{proof}
By \cite[Proposition 2.8.(i)]{rigid-p-adic-hodge} and \cite[Proposition 2.10.(i)]{rigid-p-adic-hodge} we have $M = \ri/\pi_1 \dsum \dots \ri/\pi_n$ for some $\pi_1, \dots, \pi_n \in \mm$. In particular $M$ has no non-zero almost zero elements, which implies that $f$ is injective.

To prove that $f$ is surjective, fix an element $y \in N$. As $f$ is almost surjective, for every $m \in \mm$ there is some $x_m \in M$ such that $f(x_m) = my$. Fix such an $m$. Then for every other $m' \in \mm$ with $m' \divides m$ we have $f(\frac m{m'} x_m') = \frac m{m'} (m'y) = my = f(x_m)$, so by the injectivity of $f$ we deduce $\frac m{m'} x_m' = x_m$, i.e. $x_m$ is divisible by $\frac m{m'}$. As this is true for all $m'$, we deduce (by the above structure of $M$) that $x_m$ is divisible by $m$, i.e. there is some $x \in M$ with $mx = x_m$. Then $m (f(x) - y) = 0$. As $m$ was arbitrary, we deduce that $f(x) - y$ is almost zero, so by assumption on $N$ we get $f(x) = y$.
\end{proof}

The functors from the following result will provide a quasi-inverse to the main functor of this paper. It is therefore helpful to understand them properly, so we prove fully faithfulness and analyze the essential image. The fully faithfulness also follows directly from the construction of the main functor.

\begin{proposition} \label{rslt:fully-faithful-functor-loc-sys}
Let $X$ be a proper adic space of finite type over $\Cpx_p$ and let $\varpi \in \ri$ be a pseudo-uniformizer. Then the functors
\begin{alignat*}{2}
	- \tensor_{\ri/\varpi} \check\ri^+_X/\varpi \colon&& \catlocsys{\ri/\varpi}X &\longto \catmod{\check\ri^+_X/\varpi},\\
	- \tensor_\ri \check\ri^+_X \colon&& \catlocsys\ri X &\longto \catmod{\check\ri^+_X},\\
	- \tensor_{\Cpx_p} \check\ri_X \colon&& \catlocsysint{\Cpx_p}X &\longto \catmod{\check\ri_X}
\end{alignat*}
are fully faithful. Moreover all objects in the essential image of the first functor $\tensor_{\ri/\varpi} \check\ri^+_X/\varpi$ are in fact \emph{étale} locally free $\ri{_X}/\varpi$-modules.
\end{proposition}
\begin{proof}
We first show fully faithfulness. All three cases are analogous, so we only explain the proof in the third case. Let $L_1 = \mathcal L_1[p^{-1}]$ and $L_2 = \mathcal L_2[p^{-1}]$ be two $\Cpx_p$-local systems on $X$ with integral models in the sense of  \cref{def:rep-integral-model}. Then $\IHom_{\Cpx_p}(L_1, L_2)$ is also a $\Cpx_p$-local system on $X$. Furthermore, there is a natural isomorphism of $\check\ri_X$-modules on $X_\vsite$
\begin{align*}
	\IHom_{\check\ri_X}(L_1 \tensor_{\Cpx_p} \check\ri_X, L_2 \tensor_{\Cpx_p} \check\ri_X) = \IHom_{\Cpx_p}(L_1, L_2) \tensor_{\Cpx_p} \check\ri_X.
\end{align*}
Indeed, there is a natural map from right to left, so we can check locally, where everything is free and hence obvious. The global sections of the left-hand side are $\Hom_{\check\ri_X}(L_1 \tensor_{\Cpx_p} \check\ri_X, L_2 \tensor_{\Cpx_p} \check\ri_X)$. On the other hand, $\IHom_{\Cpx_p}(L_1, L_2)$ has an integral model, namely $\IHom_\ri(\mathcal L_1, \mathcal L_2)$ (this can again be checked locally). Hence by \cref{rslt:global-sections-on-proper-space} the global sections of the right-hand side of the above equation are $\Hom_{\Cpx_p}(L_1, L_2)$.

To prove the claim about the essential image of the first functor, let $\mathcal L_\varpi$ be an $\ri/\varpi$-local system. By \cref{rslt:injlim-of-locsys} we have $\mathcal L_\varpi = \mathcal L_F \tensor_{\ri_F} \ri$ for some finite extension $F$ of $\Q_p$ and some $\ri_F/\varpi$-local system $\mathcal L_F$ on $X$. Then by \cref{rslt:locsys-trivial-on-fet-cover} there is some finite étale cover $Y \to X$ such that $\restrict{\mathcal L_F}Y$ is constant. Then the same is true for $\mathcal L_\varpi$ and hence $\mathcal L_\varpi \tensor_{\ri/\varpi} \check\ri^+_X/\varpi$ is free on $Y$.
\end{proof}

\section{Descent of Integral Models} \label{sec:desc-int-model}

The goal of this section is to show that the existence of an integral model of a $\Cpx_p$-local system can be descended along a large class of proper surjective morphisms. This will ultimately show that the construction of the functor $\Delta$ from \cref{sec:construction} can be descended along these morphisms.

The strategy of the proof is to reduce the statement to a local statement which can then be checked on w-contractible spaces. In fact this allows us to reduce the problem to a purely topological question in terms of profinite sets. We first analyse the theory in that case, which is the content of the following results.

\begin{lemma} \label{rslt:etalification-of-finite-open-map-of-profinite-sets}
Let $X$ be an extremally disconnected profinite set, $Y$ a profinite set and $f\colon Y \to X$ an open continuous map with finite fibers of bounded size. Then there is a profinite set $Z$ with a surjective continuous map $g\colon Z \surjto Y$ such that $f\comp g\colon Z \to X$ is a finite disjoint union of clopen immersions.
\end{lemma}
\begin{proof}
We use induction on the maximum size $n$ of a fiber of $Y$ over $X$. Let $Z_n = f(Y)$, which is a clopen subset of $X$. In particular $Z_n$ is extremally disconnected and $f\colon Y \to Z_n$ is surjective. Hence by \cite[Proposition 08YN]{stacks-project} there is a continuous section $s_n\colon Z_n \to Y$ of $f$ on $Z_n$. Let $Y_{(n-1)} := \tclos{Y \setminus s_n(Z_n)}$. Since $f$ is closed (as a continuous map of compact Hausdorff spaces) we have $f(Y_{(n-1)}) = \tclos{f(Y \setminus s_n(Z_n))}$. On the other hand, $s_n(Z_n)$ is closed in $Y$, so its complement is open in $Y$, and by the openness of $f$ the image $f(Y \setminus s_n(Z_n))$ is open in $X$. As $X$ is extremally disconnected, the closure of this open set is still open, hence $f(Y_{(n-1)})$ is a clopen subset of $X$. In particular, the restriction $\restrict f{Y_{(n-1)}}\colon Y_{(n-1)} \to X$ is still a continuous open map.

We claim that every fiber of $\restrict f{Y_{(n-1)}}$ has at most size $n - 1$. Let $x \in X$ such that the fiber $Y_x$ of $Y$ over $X$ has size $n$. Since $Y$ is compact Hausdorff and in particular T4, we can find disjoint open neighbourhoods $U_1, \dots, U_n \subset Y$ of the points in $Y_x$. The images $f(U_i)$ of these open neighbourhoods in $X$ are open, and by letting $W \subset X$ be a clopen neighbourhood of $x$ inside the intersection of all $f(U_i)$'s we see that $f^{-1}(W) = \bigdunion_{i=1}^n W_i$ with clopen disjoint subsets $W_i$ of $Y$. Now $s_n(x)$ is contained in one of them, say $W_1$. Now $W_1 \cap s_n(W) = W_1 \cap( s_n \circ f)^{-1} (W_1)$ is an open neighborhood of $s_n(x)$, which shows that  $s_n(x)$ does not lie in $Y_{(n-1)}$, so that the fiber of $Y_{(n-1)}$ over $x$ has at most size $n - 1$.

We have now shown that we can apply the induction hypothesis to $\restrict f{Y_{(n-1)}}$ in order to get a profinite set $Z_{(n-1)}$ and a surjective map $g_{(n-1)}\colon Z_{(n-1)} \surjto Y_{(n-1)}$ such that the composition $f\comp g\colon Z_{(n-1)} \to X$ is a finite disjoint union of clopen immersions. But then we can just take $Z := Z_n \dunion Z_{(n-1)}$ with the obvious map to $Y$ to finish the proof.
\end{proof}

In the following, for any profinite set $X$ we let $X_\qproet$ denote the \emph{pro-étale site} of $X$: It consists of all continuous maps $Y \to X$ from profinite sets $Y$, with coverings given by finite families of jointly surjective maps (note that this is just the slice of $*_\proet$ over $X$).

\begin{lemma} \label{rslt:intersection-of-O-n-with-free-C-p-module-is-free-on-profinite-set}
Let $X$ be a profinite set and let $L \subset \Cpx_p^n$ be a free sub-$\Cpx_p$-module on $X_\qproet$. Let $\mathcal L := L \isect \ri^n$. Then $\mathcal L$ is a free $\ri$-module on $X_\qproet$ and $\mathcal L[p^{-1}] = L$.
\end{lemma}
\begin{proof}
Let $r$ be the rank of $L$. Take any $x \in X$ and consider the fiber $L_x$ of $L$ over $x$ (to be precise we can define $L_x = L(\{ x \})$, where we note that $\{ x \} \injto X$ is a pro-étale map). Then $L_x$ is an $r$-dimensional $\Cpx_p$-vector space inside $\Cpx_p^n$ and it follows from elementary considerations that $\mathcal L_x = \ri^n \isect L_x \subset L_x$ is an $\ri$-lattice. Let $e_{1,x}, \dots, e_{r,x} \in \Cpx_p^n$ be a $\ri$-basis of $\mathcal L_x$ and extend each $e_{i,x}$ to a section $e_i \in L(X)$. It is enough to show that there is an open neighbourhood $U \subset X$ of $x$ such that for all $y \in U$, $\{ e_1(y), \dots, e_r(y) \}$ is a basis of $\mathcal L_y$ over $\ri$. Indeed, by covering $X$ by a disjoint union of such $U$'s we may assume $U = X$. Then consider the map $\alpha\colon \ri^r \to \mathcal L$ of sheaves on $X_\qproet$ which is induced by the $e_i$'s. One first checks that $\alpha(X)\colon \ri^r(X) \to \mathcal L(X)$ is an isomorphism: Every section $s \in \mathcal L(X)$ can be uniquely written as a linear combination $s = \sum_{i=1}^r s_i e_i$ for some $s_i \in \Cpx_p(X) = \cts(X, \Cpx_p)$ (because $\mathcal L \subset L$ and $L$ is free over $\Cpx_p$) and it follows immediately that all $s_i$ have image in $\ri$. But then the same argument shows that $\alpha(V)\colon \ri^r(V) \to \mathcal L(V)$ is an isomorphism for all $V \in X_\qproet$.

It remains to show the existence of the open neighbourhood $U$ of $x$ on which $\{ e_1, \dots, e_r \}$ is a basis of $\mathcal L$. We observe that for a point $y \in X$, the vectors $e_1(y), \dots, e_r(y) \in L \subset \Cpx_p^n$ form a basis of $\mathcal L_y \subset L_y$ if and only if they form a basis of $L_y$, lie in $\ri^n$ and are linearly independent modulo $\mm$. In particular these three conditions are satisfied at $y = x$. But by looking at ranks of matrices and using the fact that $\mm \subset \ri \subset \Cpx_p$ are open inclusions, one deduces immediately that the three conditions are automatically also satisfied on an open neighbourhood of $x$.
\end{proof}

\begin{corollary} \label{rslt:descent-of-integral-model-on-profinite-set}
Let $X$ be an extremally disconnected profinite set, $Y$ a profinite set and $f\colon Y \surjto X$ a continuous surjective open map with finite fiber of bounded size. Let $L$ be a free $\Cpx_p$-module on $X_\qproet$ and let $\mathcal L_Y \subset f^{-1}L$ be a free $\ri$-module on $Y_\qproet$ such that $f^{-1}L = \mathcal L_Y[p^{-1}]$. Let $\mathcal L := L \isect f_*\mathcal L_Y$ (the intersection being taken inside $f_*f^{-1}L$). Then $\mathcal L$ is free over $\ri$ and $L = \mathcal L[p^{-1}]$.
\end{corollary}
\begin{proof}
Let $g\colon Z \surjto Y$ be as in \cref{rslt:etalification-of-finite-open-map-of-profinite-sets}. From the surjectivity of $g$ and the freeness of $\mathcal L_Y$ we deduce $\mathcal L_Y = g_*g^{-1}\mathcal L_Y \isect f^{-1}L$ (the intersection being taken inside $g_* g^{-1} (f^{-1} L)$). This equation remains true after applying $f_*$, which implies $\mathcal L = L \isect f_*\mathcal L_Y = L \isect (f \comp g)_* (g^{-1} \mathcal L_Y) \isect f_* f^{-1} L = L \isect (f\comp g)_* (g^{-1} \mathcal L_Y)$. Thus we can w.l.o.g. replace $Y$ by $Z$, $\mathcal L_Y $ by $g^{-1} \mathcal L_Y$ and $f$ by $f\comp g$, i.e. we can assume that $f\colon Y \to X$ is a finite disjoint union of clopen immersions. By passing to a disjoint open cover of $X$ we can then reduce to the case $Y = \bigdunion_{i=1}^n X$ for some $n \ge 1$. In this case we note that $f_* \mathcal L_Y$ is free of rank $rn$ over $\ri$, where $r$ is the rank of $\mathcal L_Y$. By fixing an isomorphism $f_*\mathcal L_Y \isom \ri^{rn}$ (and consequently $f_*f^{-1}L \isom \Cpx_p^{rn}$) we are precisely in the situation of \cref{rslt:intersection-of-O-n-with-free-C-p-module-is-free-on-profinite-set}, which implies that $\mathcal L$ has the desired properties.
\end{proof}

We have finished our analysis in the profinite case. The next results provide the bridge to the world of diamonds:

\begin{lemma} \label{rslt:construction-of-pi-X-minus-1}
Let $X$ be a spatial diamond. Then $\pi_0(X)$ is a profinite set and the functor $\pi_0\colon X_{\qproet,\qcqs} \to \pi_0(X)_\qproet$ admits a fully faithful right adjoint.
\end{lemma}
\begin{proof}
Since $\abs X$ is a spectral space (because $X$ is spatial), $\pi_0(X) = \pi_0(\abs X)$ is a profinite set. The same is true for any $Y \in X_{\qproet,\qcqs}$ (see \cite[Corollary 11.28]{etale-cohomology-of-diamonds}), so that $\pi_0(Y) \in \pi_0(X)_\qproet$. We now construct a right adjoint $\pi_X^{-1}\colon \pi_0(X)_\qproet \to X_{\qproet,\qcqs}$.

Suppose first that $U \subset \pi_0(X)$ is a quasi-compact open subset. Then we let $\pi_X^{-1}(U) \subset X$ be the preimage of $U$ under the projection $X \surjto \pi_0(X)$ (this is a clopen subset). We claim that for every $V \in X_{\qproet,\qcqs}$ we have
\begin{align*}
	\Hom_X(V, \pi_X^{-1}(U)) = \Hom_{\pi_0(X)}(\pi_0(V), U).
\end{align*}
To see this, note that both sides are either empty or contain a single morphism, depending on whether $V \to X$ (resp. $\pi_0(V) \to \pi_0(X)$) factors over $\pi_X^{-1}(U)$ (resp. $U$). It is then clear that both sides coincide. Next let $U = \bigdunion_{k=1}^n U_k$ be a disjoint union of quasi-compact open subsets $U_k \subset \pi_0(X)$, which is a topological space covering the union of the $U_k$. Then we let $\pi_X^{-1}(U) := \bigdunion_{k=1}^n \pi_X^{-1}(U_k)$. It is easy to see that the above identity of $\Hom$'s is still satisfied.

Now let $U \in \pi_0(X)_\qproet$ be general. Using \cite[Lemma 2.1.13]{proetale-topology}, we can write $U = \varprojlim_{i\in I} U_i$ with all $U_i$ being finite disjoint unions of quasi-compact open subsets of $\pi_0(X)$. Letting $\pi_X^{-1}(U) = \varprojlim_i \pi_X^{-1}(U_i)$ we get for all $V \in X_{\qproet,\qcqs}$:
\begin{align*}
	&\Hom_X(V, \pi_X^{-1}(U)) = \Hom_X(V, \varprojlim_i \pi_X^{-1}(U_i)) = \varprojlim_i \Hom_{\pi_0(X)}(\pi_0(V), U_i) =\\\qquad&=\Hom_{\pi_0(X)}(\pi_0(V), U).
\end{align*}
This shows that $\pi_X^{-1}(U)$ is independent of the limit representation $U = \varprojlim_i U$ and that the thus defined functor $\pi_X^{-1}$ is a right adjoint of $\pi_0$.

To show that $\pi_X^{-1}$ is fully faithful we have to show that $\pi_0(\pi_X^{-1}(U)) = U$ for all $U \in \pi_0(X)_\qproet$. If $U$ is a disjoint union of quasi-compact open subsets of $\pi_0(X)$ then this is clear. For general $U$ we again use the representation $U = \varprojlim_i U_i$ as above. Then
\begin{align*}
	\pi_0(\pi_X^{-1}(U)) = \pi_0(\varprojlim_i \pi_X^{-1}(U_i)) = \varprojlim_i \pi_0(\pi_X^{-1}(U_i)) = \varprojlim U_i = U,
\end{align*}
as desired.
\end{proof}

\begin{definition}
Let $X$ be a spatial diamond and let $\pi_X^{-1}\colon \pi_0(X)_\qproet \to X_\qproet$ be the right adjoint in \cref{rslt:construction-of-pi-X-minus-1}. This defines a morphism of sites
\begin{align*}
	\pi_X\colon X_\qproet \to \pi_0(X)_\qproet.
\end{align*}
\end{definition}

\begin{remark} \label{rmk:pr-qproet-is-composition-of-pi-X-with-slice}
Let $X$ be a spatial diamond. In \cref{def:constant-sheaves} we introduced a morphism of sites $\pr\qproet\colon X_\qproet \to *_\qproet$. One verifies easily that this morphism is in fact the composition of the morphism $\pi_X\colon X_\qproet \to \pi_0(X)_\qproet$ with the slice $\pi_0(X)_\qproet \to *_\qproet$.\end{remark}

\begin{lemma} \label{rslt:properties-of-pi-X}
Let $X$ be a spatial diamond.
\begin{lemenum}
	\item For every sheaf $\mathcal F$ on $\pi_0(X)_\qproet$ and every qcqs $V \in X_\qproet$ we have $(\pi_X^{-1}\mathcal F)(V) = \mathcal F(\pi_0(V))$. In particular, the inverse image functor $\pi_X^{-1}$ commutes with limits of sheaves.
	\item We have $\pi_{X*}\pi_X^{-1} = \id{}$ on $\catshs{\pi_0(X)_\qproet}$. In particular, the inverse image functor $\pi_X^{-1}\colon \catshs{\pi_0(X)_\qproet} \injto \catshs{X_\qproet}$ is fully faithful.
\end{lemenum}
\end{lemma}
\begin{proof}
This is analogous to \cite[Lemma 4.2.13]{proetale-topology}. To prove (i), let a sheaf $\mathcal F$ on $\pi_0(X)_\qproet$ be given. From the adjunction of $\pi_X^{-1}$ and $\pi_0$ (see \cref{rslt:construction-of-pi-X-minus-1}) one easily deduces that $\pi_X^{-1}\mathcal F$ is the sheafification of the presheaf $U \mapsto \mathcal F(\pi_0(U))$, so we only have to show that this presheaf satisfies the sheaf condition for covers $U \surjto V$ in $X_{\qproet,\qcqs}$. Given such a cover, we want to show that the sequence
\begin{align*}
	0 \to \mathcal F(\pi_0(V)) \to \mathcal F(\pi_0(U)) \rightrightarrows \mathcal F(\pi_0(U \cprod_V U))
\end{align*}
is exact. It is enough to show that $\pi_0(U \cprod_V U) \rightrightarrows \pi_0(U) \to \pi_0(V)$ is a coequalizer diagram in the category of profinite sets (because then the natural map $\pi_0(U \cprod_V U) \to \pi_0(U) \cprod_{\pi_0(V)} \pi_0(U)$ is surjective and we can conclude by the sheafiness of $\mathcal F$). This means that for every profinite set $S$ we need to show that
\begin{align*}
	\Hom(\pi_0(V), S) \to \Hom(\pi_0(U), S) \rightrightarrows \Hom(\pi_0(U \cprod_V U), S)
\end{align*}
is an equalizer diagram of sets. Since $S$ is a totally disconnected topological space we have $\Hom(\pi_0(V), S) = \cts(\pi_0(\abs V), S) = \cts(\abs V, S)$ and similarly for the other two entries in the above diagram. We are therefore reduced to showing that the presheaf $W \mapsto \cts(\abs W, S)$ is a sheaf on $X_\qproet$; but this is \cref{rslt:F-T-is-sheaf-on-vsite}.

Part (ii) follows easily from (i) using that $\pi_0(\pi_X^{-1}(U)) = U$ for all $U \in \pi_0(X)_\qproet$ (see \cref{rslt:construction-of-pi-X-minus-1}). 
\end{proof}

\begin{lemma} \label{rslt:pi-0-commutes-with-pro-etale-fibre-product}
Let $X$ be a strictly totally disconnected space, let $V \to X$ be an affinoid pro-étale morphism and let $Y \to X$ be any morphism from a spatial v-sheaf $Y$. Then there is a natural homeomorphism
\begin{align*}
	\pi_0(Y \cprod_X V) = \pi_0(Y) \cprod_{\pi_0(X)} \pi_0(V).
\end{align*}
\end{lemma}
\begin{proof}
There is a natural continuous map from left to right and both sides are compact Hausdorff spaces; hence it is enough to show that the map is bijective. For every connected component $X_0$ of $X$, the connected components of $Y \cprod_X V$ that map to $X_0$ are precisely the connected components of $(Y \cprod_X V) \cprod_X X_0 \subset Y \cprod_X V$. Thus it is enough to show the claim for $X_0$ in place of $X$, so we can w.l.o.g. assume that $X$ is connected.

We can similarly assume that $V$ is connected. Then $X = \Spa(C, C^+)$ for some algebraically closed perfectoid field $C$ and open and bounded valuation subring $C^+$ and $V = \Spa(C, C'^+)$ for another open and bounded valuation subring $C'^+ \subset C$ with $C^+ \subset C'^+$. It is then enough to consider the case $C'^+ = C^\circ$.

We can furthermore assume that $Y$ is connected, and have to show that then also $f^{-1}(V)$ is connected, where $f$ denotes the map $Y \to X$. Assume that $f^{-1}(V) = Y_1 \dunion Y_2$ for clopen subsets $Y_1, Y_2 \subset f^{-1}(V)$. Let $\tclos{Y_1}, \tclos{Y_2} \subset Y$ be the closures; we claim that $Y = \tclos{Y_1} \dunion \tclos{Y_2}$ which would finish the proof. To show the claim we note that both $Y_1$ and $Y_2$ are quasi-compact open subsets of $Y$ and in particular (pro-)constructible, hence their closures are obtained by adding all specializations (see \cite[Lemma 2.4]{etale-cohomology-of-diamonds}). It is therefore enough to show that every $y \in Y$ has a unique maximal generalization, and that $f^{-1}(V)$ contains all maximal generalizations of $Y$. The latter part is clear because $f^{-1}(V) \subset Y$ is open. For the former part, we note that the statement is true for analytic adic (and in particular perfectoid) spaces and that it descends along surjective generalizing maps, so the claim follows from \cite[Proposition 12.13.(iii)]{etale-cohomology-of-diamonds}.
\end{proof}

\begin{corollary} \label{rslt:base-change-for-pi-0}
Let $X$ be a strictly totally disconnected space and let $f\colon Y \to X$ be any morphism from a spatial diamond $Y$. Consider the natural morphisms of sites $\pi_X\colon X_\qproet \to \pi_0(X)_\qproet$ and $\pi_Y\colon Y_\qproet \to \pi_0(Y)_\qproet$. They fit in a commutative diagram of sites
\begin{center}\begin{tikzcd}
	Y_\qproet \arrow[r,"\pi_Y"] \arrow[d,"f",swap] & \pi_0(Y)_\qproet \arrow[d,"\pi_0(f)"]\\
	X_\qproet \arrow[r,"\pi_X"] & \pi_0(X)_\qproet
\end{tikzcd}\end{center}
Then for every sheaf $\mathcal F$ on $\pi_0(Y)_\qproet$, the base change morphism
\begin{align*}
	\pi_X^{-1} \pi_0(f)_* \mathcal F \isoto f_* \pi_Y^{-1} \mathcal F
\end{align*}
is an isomorphism.
\end{corollary}
\begin{proof}
It is enough to show that both sides have the same value on affinoid pro-étale $V \to X$. Using \cref{rslt:properties-of-pi-X} we get
\begin{align*}
	& (\pi_X^{-1} \pi_0(f)_* \mathcal F)(V) = (\pi_0(f)_* \mathcal F)(\pi_0(V)) = \mathcal F(\pi_0(Y) \cprod_{\pi_0(X)} \pi_0(V)),\\
	& (f_* \pi_Y^{-1} \mathcal F)(V) = (\pi_Y^{-1} \mathcal F)(Y \cprod_X V) = \mathcal F(\pi_0(Y \cprod_X V)).
\end{align*}
Thus the claim follows immediately from \cref{rslt:pi-0-commutes-with-pro-etale-fibre-product}.
\end{proof}

We are now prepared to prove the main technical result about descent of integral models. To keep our statements more concise we first introduce the following terminology.

\begin{definition}
Let $f\colon Y \to X$ be a map of diamonds. We say that \emph{integral models descend along $f$} if for every $\Cpx_p$-local system $L$ on $X$ such that $\restrict LY$ has an integral model, also $L$ has an integral model.
\end{definition}

\begin{proposition} \label{rslt:descent-of-integral-models-abstract-result}
Let $f\colon Y \to X$ be a surjective morphism of spatial diamonds and suppose there exists a quasi-pro-étale covering $\tilde X \to X$ by a w-contractible space $\tilde X$ such that the pullback morphism $\tilde f\colon \tilde Y := Y \cprod_X \tilde X \to \tilde X$ induces an open map $\pi_0(\tilde f)\colon \pi_0(\tilde Y) \to \pi_0(\tilde X)$ with finite fibers of bounded size. Then integral models descend along $f$.
\end{proposition}
\begin{proof}
Let $\mathcal L_Y$ be an integral model of $\restrict LY$, i.e. $\mathcal L_Y \subset \restrict LY$ is an $\ri$-local system such that $\restrict LY = \mathcal L_Y[p^{-1}]$. There are natural inclusions $L \injto f_*(\restrict LY)$ and $f_*\mathcal L_Y \injto f_*(\restrict LY)$. We define the sheaf $\mathcal L$ on $X_\vsite$ as the intersection $\mathcal L := L \isect f_*\mathcal L_Y$. The claim is that $\mathcal L$ is an integral model of $L$. Now this claim is local on $X_\vsite$, so we can replace $X$ by $\tilde X$ and w.l.o.g. assume that $X$ is a w-contractible space and $\pi_0(f)\colon \pi_0(Y) \to \pi_0(X)$ is an open map of profinite sets with finite fibers of bounded size. Since $X$ is w-contractible and $L$ is locally constant, there is an open disjoint cover of $X$ on which $L$ is constant (possibly of different ranks); by passing to that open cover we can assume that $L$ is constant, i.e. isomorphic to $\Cpx_p^r$ for some $r \ge 0$. Then also $\restrict LY$ is constant, so by \cref{rslt:integral-model-of-constant-sheaf-is-constant} below, $\mathcal L_Y \subset \restrict LY$ is free of rank $r$ over $\ri$. We claim that $\mathcal L$ is free of rank $r$ over $\ri$, which is enough to finish the proof.

To prove the claim, consider the morphism of sites $\pi_Y\colon Y_\qproet \to \pi_0(Y)_\qproet$ introduced in \cref{rslt:base-change-for-pi-0}. Since both $\mathcal L_Y$ and $\restrict LY$ are constant sheaves on $Y_\qproet$ asssociated to a T1 topological space, they arise as pullbacks from $\pi_0(Y)_\qproet$, more precisely $\mathcal L_Y = \pi_Y^{-1}\pi_{Y*} \mathcal L_Y$ and $\restrict LY = \pi_Y^{-1}\pi_{Y*} (\restrict LY)$. Similarly $L = \pi_X^{-1}\pi_{X*} L$. By \cref{rslt:base-change-for-pi-0} we deduce
\begin{align*}
	\mathcal L &= f_* \mathcal L_Y \isect L = f_* \pi_Y^{-1} \pi_{Y*} \mathcal L_Y \isect \pi_X^{-1} \pi_{X*} L = \pi_X^{-1} \pi_0(f)_* \pi_{Y*} \mathcal L_Y \isect \pi_X^{-1} \pi_{X*} L\\
	&= \pi_X^{-1} (\pi_0(f)_* \pi_{Y*} \mathcal L_Y \isect \pi_{X*} L).
\end{align*}
Thus it is enough to show that the sheaf $\pi_0(f)_* \pi_{Y*} \mathcal L_Y \isect \pi_{X*} L$ on $\pi_0(X)_\qproet$ is free of rank $r$ over $\ri$. But this is \cref{rslt:descent-of-integral-model-on-profinite-set}.
\end{proof}

\begin{lemma} \label{rslt:integral-model-of-constant-sheaf-is-constant}
Let $Y$ be a spatial diamond and let $\mathcal L$ be an integral model of the constant sheaf $L = \Cpx_p^r$ on $Y$. Then $\mathcal L$ is free of rank $r$ over $\ri$.
\end{lemma}
\begin{proof}
As in the proof of \cref{rslt:complete-huber-ring-has-etale-locally-integral-model} (see also \cite[Proposition 6.8.4.(6)]{proetale-topology}) one shows that the integral models of $L$ on $Y$ are parametrized by the constant sheaf associated to the discrete set $S = \GL_r(\Cpx_p)/\GL_r(\ri)$. Thus there is an isomorphism $s \in S(Y)$ from $\ri^r$ to $\mathcal L$.
\end{proof}

It remains to identify a large class of morphisms $f\colon Y \to X$ which satisfy the hypothesis of \cref{rslt:descent-of-integral-models-abstract-result}. For this we have to make extensive use of the geometry of adic spaces over $\Cpx_p$. Unfortunately most of the results we want to apply are only stated for rigid-analytic and Berkovich spaces instead of adic spaces in the literature. We have therefore gathered adic versions of the relevant results in \cite{simpson-appendix} in order to reference them here. We work in a slightly more general setting than proper adic spaces over $\Cpx_p$, which is the one needed later in the present paper, because we believe that the following results are interesting on their own.  Let us therefore fix a non-archimedean field $K$ and let $\catwfintype K$ denote the category of adic spaces locally of weakly finite type over $K$ (cf. \cite[Definition 1.2.1.(i)]{huber-etale-cohomology}).

\begin{lemma} \label{rslt:descent-of-integral-models-along-finite-map}
Let $X, Y \in \catwfintype K$, let $f\colon Y \to X$ be a finite surjective morphism and assume that both $X$ and $Y$ are of the same pure dimension $d$ and that $X$ is normal. Then integral models descend along $f$.
\end{lemma}
\begin{proof}
Let $\tilde X \to X$ be any quasi-pro-étale map from a w-contractible space $\tilde X$ which is the cofiltered limit $\tilde X = \varprojlim_i X_i$ of étale maps $X_i \to X$. Let $Y_i := Y \cprod_X X_i$ for all $i$, let $\tilde Y := Y \cprod_X \tilde X = \varprojlim_i Y_i$ and let $f_i\colon Y_i \to X_i$ and $\tilde f\colon \tilde Y \to \tilde X$ be the base changes of $f$. Note that by \cite[Lemma 1.18]{simpson-appendix}, all $X_i$ and $Y_i$ are of pure dimension $d$, all $X_i$ are normal and all $f_i$ are finite. Hence by \cite[Theorem 1.20]{simpson-appendix}, all $f_i$ are open.

We claim that also $\tilde f$ is open. Indeed, let $U \subset \tilde Y$ be open. Since $\abs{\tilde Y} = \varprojlim_i \abs{Y_i}$ by \cite[Lemma 11.22]{etale-cohomology-of-diamonds}, $U$ is covered by preimages of open sets under the projections $\tilde Y \to Y_i$, so we can assume that $U$ itself is of the form $U = U_i \cprod_{Y_i} \tilde Y$ for some $i$ and some qcqs open subset $U_i \subset Y_i$. But then we have $\tilde f(U) = f_i(U_i) \cprod_{X_i} \tilde X$ (indeed, $\restrict{f_i}{U_i}\colon U_i \to f_i(U_i)$ is surjective, hence $\restrict{\tilde f}{U}\colon U_i \cprod_{Y_i} \tilde Y = U_i \cprod_{X_i} \tilde X \to f_i(U_i) \cprod_{X_i} \tilde X$ is surjective), which is open in $\tilde X$.

Now $\tilde f\colon \tilde Y \to \tilde X$ is a closed and open map, which easily implies that $\pi_0(\tilde f)\colon \pi_0(\tilde Y) \to \pi_0(\tilde X)$ is open. Furthermore, $\tilde f$ is a finite morphism of affinoid adic spaces and all points of $\tilde X$ are geometric, which easily implies that the fiber site of $\tilde f$ is bounded: If $\ri_{\tilde Y}(\tilde Y)$ is generated over $\ri_{\tilde X}(\tilde X)$ by $n$ sections then each fiber over some $\Spa(C, C^+)$ has at most size $n$. By \cref{rslt:descent-of-integral-models-abstract-result} we deduce that integral models descend along $f$.
\end{proof}

\begin{theorem} \label{rslt:descent-of-integral-model-along-morphism-of-proper-spaces-with-loc-irr-target}
Let $X, Y \in \catwfintype K$, let $f\colon Y \to X$ be a surjective proper morphism of finite type and assume that $X$ is normal. Then integral models descend along $f$.
\end{theorem}
\begin{proof}
By the Stein Factorization (see \cite[Theorem 2.9]{simpson-appendix}) we can decompose $f$ into proper maps $g\colon Y \to Z$ and $h\colon Z \to X$ such that the fibers of $g$ are geometrically connected and $h$ is finite. It is enough to show that integral models descend along $g$ and $h$ separately.

Let us start with $g$. Take any w-contractible $\tilde Z \in Z_\qproet$ and let $\tilde g\colon \tilde Y := Y \cprod_Z \tilde Z \to \tilde Z$ be the base change. Then the fibers of $\tilde g$ are geometric fibers of $g$ which are connected by assumption. It follows that the map $\pi_0(\tilde g)\colon \pi_0(\tilde Y) \to \pi_0(\tilde Z)$ is a (necessarily continuous) bijection of compact Hausdorff spaces, hence a homeomorphism and in particular open. By \cref{rslt:descent-of-integral-models-abstract-result}, integral models descend along $g$.

It remains to prove that integral models descend along $h$. We want to reduce this to the case that $X$ and $Z$ are pure-dimensional, in which case we can apply \cref{rslt:descent-of-integral-models-along-finite-map}. By passing to connected components of $X$ we can assume that $X$ is irreducible (cf. \cite[Lemma 1.12]{simpson-appendix}
). By \cite[Proposition 1.17]{simpson-appendix} $X$ is pure-dimensional. Now decompose $Z$ into irreducible components $Z = \bigunion_{k=1}^n Z_k$ (cf. \cite[Definition 1.14]{simpson-appendix}
). By \cite[Lemma 1.13]{simpson-appendix} 
their images under $h$ form a covering of $X$ by closed adic subspaces, so by the irreducibility of $X$ there must be some $k \in \{ 1, \dots, n \}$ such that $\restrict h{Z_k}\colon Z_k \to X$ is surjective. It is enough to show that integral models descend along $\restrict h{Z_k}$ (then the same is obviously true for $h$), so we can replace $Z$ by $Z_k$ to assume that $Z$ is irreducible and hence pure-dimensional (again by \cite[Proposition 1.17]{simpson-appendix}
). Since $h$ is finite surjective, the dimensions of $X$ and $Z$ must coincide (see \cite[Lemma 1.10]{simpson-appendix}
), hence we can apply \cref{rslt:descent-of-integral-models-along-finite-map} to finish the proof. 
\end{proof}

\section{From Modules to Local Systems} \label{sec:construction}

Fix a proper adic space $X$ of finite type over $\Cpx_p$. We now construct a functor from a certain class of vector bundles on $X$ to $\Cpx_p$-local systems on $X$. In contrast to previous approaches (\cite{deninger-werner-parallel-transport-varieties}, \cite{matti-rigid-DW-functor}) we allow general proper covers in order to trivialize vector bundles modulo $p^n$. We also introduce a different construction of the associated local systems. We start with some observations relying on the results from \cref{sec:globsec}.

\begin{lemma}
Let $Y$ be a proper adic space of finite type over $\Cpx_p$ and fix $n \ge 1$.
\begin{lemenum}
	\item \label{rslt:sub-locsys-of-free-O-mod-p-module} Let $\mathcal E_n$ be a free $\check\ri^+_Y/p^n$-module. Then there is a unique free $\ri/p^n$-submodule $\mathcal L_n \injto \mathcal E_n$ of maximal rank.
	\item \label{rslt:morphism-on-free-vb-preserves-loc-sys} Let $\mathcal E_n$ and $\mathcal E'_n$ be free $\check\ri^+_Y/p^n$-modules and let $\mathcal L_n \injto \mathcal E_n$ and $\mathcal L'_n \injto \mathcal E'_n$ denote the free $\ri/p^n$-submodules of maximal rank. Then for any morphism $\psi\colon \mathcal E_n \to \mathcal E'_n$, the induced map $\mathcal L_n \to \mathcal E'_n$ factors uniquely over a morphism $\psi'\colon \mathcal L_n \to \mathcal L'_n$. Furthermore, if $\psi$ is injective, surjective or an isomorphism, the same is true for $\psi'$.
\end{lemenum}
\end{lemma}
\begin{proof}
We may assume that $Y$ is connected. In order to prove the first claim, note that existence of $\mathcal L_n$ is clear, since $\ri/p^n$ is a subsheaf of $\check\ri^+_Y/p^n$, so that every  $\check\ri^+_Y/p^n$-basis of $\mathcal E_n$ defines a free $\ri/p^n$-module inside $\mathcal E_n$ of maximal rank. On the other hand, every free $\ri/p^n$-module $\mathcal L_n$  inside $\mathcal E_n$  satisfies $\Gamma(Y, L_n ) \subset \Gamma(Y, \mathcal E_n) $. By \cref{rslt:global-sections-on-proper-space} we find $ \Gamma(Y, \mathcal E_n) \simeq \Gamma (Y,  \check\ri^+_X/p^n)^r = \Gamma(Y, \ri / p^n)^r = ( \ri / p^n)^r$, since $Y$ is connected. If $\mathcal L_n$ has maximal rank, we find $\Gamma(Y, \mathcal L_n) = \Gamma(Y, \mathcal E_n)$. Hence $\mathcal L_n$ is the free $\ri/p^n$-submodule of $\mathcal E_n$ generated by its global sections and therefore unique.

In order to prove the second claim, we have to construct the dashed arrow in the following diagram:
\begin{center}\begin{tikzcd}
	\mathcal L_n \arrow[r,hookrightarrow] \arrow[d,dashed] & \mathcal E_n \arrow[d,"\psi"]\\
	\mathcal L'_n \arrow[r,hookrightarrow] & \mathcal E'_n
\end{tikzcd}\end{center}
Note that  to give a morphism $\mathcal L_n \to \mathcal L'_n$ is the same as to give a morphism $\Gamma(Y, \mathcal L_n) \to \Gamma(Y, \mathcal L'_n)$. By \cref{rslt:global-sections-on-proper-space}, after applying $\Gamma(Y, -)$, the horizontal maps become isomorphisms. Hence the existence of the dashed arrow is clear.
\end{proof}

We now construct a functor from modules to local systems by working modulo all $p^n$. The appropriate class of $\check\ri^+_X/p^n$-modules is as follows.

\begin{definition}
Fix $n \ge 1$. A sheaf $\mathcal E_n$ of $\check\ri^+_X/p^n$-modules is called \emph{properly trivializable} if there is a surjective map $\alpha\colon U \to X$ with $U$ a proper adic space of finite type over $\Cpx_p$ such that $\alpha^* \mathcal E_n$ is free. The category of properly trivializable $\check\ri^+_X/p^n$-modules is denoted $\catmodproptriv{\check\ri^+_X/p^n}$.
\end{definition}

\begin{lemma} \label{rslt:descent-datum-exists-for-DW-n}
Let $\mathcal E_n$ be a properly trivializable $\check\ri^+_X/p^n$-module and let $\alpha\colon U \to X$ be as in the definition. Let $\mathcal L_{n,U} \injto \alpha^*\mathcal E_n$ be the unique free $\ri$-submodule of $\alpha^*\mathcal E_n$ of maximal rank (given by \cref{rslt:sub-locsys-of-free-O-mod-p-module}). Then the descent datum of $\alpha^*\mathcal E_n$ along $\alpha$ restricts to a descent datum of $\mathcal L_{n,U}$ along $\alpha$.
\end{lemma}
\begin{proof}
Note that we work on the v-site of $X$, so that $\alpha^*$ is just the restriction of sheaves. In particular, as on any site, we can descend arbitrary sheaves along covers.

We only need to check that the isomorphism $\pr1^*(\alpha^*\mathcal E_n) \isom \pr2^*(\alpha^*\mathcal E_n)$ restricts to a uniquely determined isomorphism $\pr1^{-1} \mathcal L_{n,U} \isoto \pr2^{-1}\mathcal L_{n,U}$, where $\pr1, \pr2\colon U \cprod_X U \to U$ are the two projections. This is \cref{rslt:morphism-on-free-vb-preserves-loc-sys} with $Y = U \cprod_X U$.
\end{proof}

\begin{definition}
Fix $n \ge 1$ and let $\mathcal E_n$ be a properly trivializable $\check\ri^+_X/p^n$-module with $\alpha\colon U \to X$ as in the definition. We define
\begin{align*}
	\Delta^+_n(\mathcal E_n) \in \catlocsys{\ri/p^n}X
\end{align*}
to be the $\ri/p^n$-local system on $X$ obtained by descending the sheaf $\mathcal L_{n,U}$ defined in \cref{rslt:descent-datum-exists-for-DW-n} along the v-cover $\alpha$.
\end{definition}

\begin{remark}
A priori the image of $\Delta^+_n$ consists of v-local systems. But by \cref{rslt:qproet-locsys-equiv-v-locsys-complete-huber-rings} and \cref{rslt:et-locsys-equiv-qproet-locsys} these are the same as étale local systems and in particular correspond to representations of the algebraic fundamental group. 
\end{remark}

\begin{lemma} \label{rslt:DW-n-is-functor}
The assignment $\mathcal E_n \mapsto \Delta^+_n(\mathcal E_n)$ provides a well-defined functor, i.e. it does not depend on the chosen cover $\alpha\colon U \to X$. Moreover, there is a natural isomorphism
\begin{align*}
	\Delta^+_n(\mathcal E_n) \tensor_\ri \check\ri^+_X = \mathcal E_n.
\end{align*}
In particular, every properly trivializable $\check\ri_X^+/p^n$-module already becomes trivial on a finite étale cover of $X$.
\end{lemma}
\begin{proof}
To check well-definedness, let $\alpha\colon U \to X$ and $\alpha'\colon U' \to X$ be two different choices of cover. By considering the fiber product $U \cprod_X U'$ we can w.l.o.g. restrict to the case of a double cover $U' \to U \to X$. Then we have to see that pulling back $\mathcal L_{n,U}$ along $U' \to U$ gives $\mathcal L_{n,U'}$. But this is clear from the uniqueness of $\mathcal L_{n,U}$ and $\mathcal L_{n,U'}$.

To prove the claimed isomorphism, first note that there is a natural map from left to right. Thus the isomorphism can be checked locally, for example on $U$ for cover $U \to X$ as in the definition of ``properly trivializable''. But on $U$ everything is free by construction, thus the isomorphism is obvious.

Finally, the last part of the claim follows directly from the isomorphism we just proved. Indeed, the isomorphism shows that every properly trivializable $\check\ri^+_X/p^n$-module is of the form $\mathcal E_n = \mathcal L_n \tensor_{\ri} \check\ri^+_X$ for some $\ri/p^n$-local system $\mathcal L_n$ on $X$. Thus the claim follows from \cref{rslt:fully-faithful-functor-loc-sys}.
\end{proof}

\begin{corollary} \label{rslt:DW-n-is-equivalence}
For every $n \ge 1$, the functor $\Delta^+_n$ is an equivalence of categories
\begin{align*}
	\Delta^+_n\colon \catmodproptriv{\check\ri^+_X/p^n} \longisoto \catlocsys{\ri/p^n}X.
\end{align*}
\end{corollary}
\begin{proof} Fully faithfulness follows from \cref{rslt:fully-faithful-functor-loc-sys}. A quasi-inverse of $\Delta^+_n$ is given by $- \tensor_{\ri/p^n} \check\ri^+_X/p^n$. Indeed, by \cref{rslt:DW-n-is-functor} we have $\Delta^+_n(-) \tensor_{\ri/p^n} \check\ri^+_X/p^n = \id{}$. It is obvious from the construction that $\Delta^+_n(- \tensor_{\ri/p^n} \check\ri^+_X/p^n) = \id{}$.
\end{proof}

\begin{remark}
As we have seen in \cref{rslt:DW-n-is-functor}, properly trivializable $\check\ri^+_X/p^n$-modules become in fact trivial on a  finite étale (instead of only proper) $\alpha\colon U \to X$. In the original algebraic approach \cite{deninger-werner-parallel-transport-varieties} and  in W\"urthen's analytic refinement \cite{matti-rigid-DW-functor} finite étale covers were used to define similar bundle categories. However, by choosing the less restrictive definition (which is enabled by new insights from the v-site) we gain much more freedom in applications. We also give an alternative definition of the $\ri/p^n$-local system given by such a bundle.
\end{remark}

Having constructed the functor modulo $p^n$, we now want to pass to the limit. For this we have to check the following:

\begin{lemma} \label{rslt:inverse-system-for-DW-n}
Let $\mathcal E$ be an $\check\ri^+_X$-module such that $\mathcal E/p^n$ is a properly trivializable $\check\ri^+_X/p^n$-module for all $n \ge 1$. Then the inverse system $(\mathcal E/p^n)_n$ induces a natural inverse system $(\Delta^+_n(\mathcal E/p^n))_n$.
\end{lemma}
\begin{proof}
Fix $n \ge 2$ and let $\mathcal E_n := \mathcal E/p^n$. We have to construct a natural map $\Delta^+_n(\mathcal E_n) \to \Delta_{n-1}(\mathcal E_{n-1})$. Let $\alpha\colon U \to X$ be a trivializing cover for $\mathcal E_n$. It is then also trivializing for $\mathcal E_{n-1}$. It is enough to construct a natural map $\mathcal L_{n,U} \to \mathcal L_{n-1,U}$. For this we take the unique such map induced by the projection $\alpha^*\mathcal E_n \surjto \alpha^*\mathcal E_{n-1}$ as in  \cref{rslt:morphism-on-free-vb-preserves-loc-sys}.
\end{proof}

\begin{definition}
\begin{defenum}
	\item \label{def:Modplus} We define $\Modplus$ to be the full subcategory of all $\check\ri^+_X$-modules $\mathcal E$ on $X_\vsite$ such that $\mathcal E / p^n$ is properly trivializable vor all $n$.
		\item Let $\mathcal E \in \Modplus$. We define
	\begin{align*}
		\Delta^+(\mathcal E) := \varprojlim_n \Delta^+_n(\mathcal E/p^n),
	\end{align*}
	where the inverse system on the right is the one defined in \cref{rslt:inverse-system-for-DW-n}.
\end{defenum}
\end{definition}

\begin{remark}
If the $\check\ri^+_X$-module $\mathcal E$ on $X_\vsite$ is étale locally free then the condition in \cref{def:Modplus} only needs to be checked for $n = 1$, see \cref{rslt:matti-trivializable-mod-pn-from-p} below.
\end{remark}

\begin{theorem} \label{rslt:DW-a-is-functor}
The assignment $\mathcal E \to \Delta^+(\mathcal E)$ is a fully faithful functor from the category $\Modplus$ to the category $\catlocsys\ri X$ of $\ri$-local systems on $X$.
Moreover, for every $\mathcal E$  in $\Modplus$ there is a natural isomorphism
\begin{align*}
	\Delta^+(\mathcal E) \tensor_\ri \check\ri^+_X = \mathcal E.
\end{align*}
\end{theorem}
\begin{proof} Note that  $\Delta(\mathcal E)$ is indeed an $\ri$-local system by \cref{rslt:projlim-of-locsys}. If $\varphi: \mathcal E \rightarrow \mathcal E'$ is a morphism in $\Modplus$, then we get a morphism $\Delta(\varphi) $ as $\varprojlim_n \Delta^+_n(\varphi  / p^n)$. Fully faithfulness follows again  from \cref{rslt:fully-faithful-functor-loc-sys}.
For the second part, first note that
\begin{align*}
	\Delta^+(\mathcal E) \tensor_{\ri} \check\ri^+_X = \left(\varprojlim_n \Delta^+_n(\mathcal E/p^n)\right) \tensor_\ri \check\ri^+_X = \varprojlim_n \left(\Delta^+_n(\mathcal E/p^n) \tensor_\ri \check\ri^+_X\right).
\end{align*}
For the second step note that we have a natural map from left to right, which is locally an isomorphism because locally everything is free. By \cref{rslt:DW-n-is-functor} we have a natural isomorphism $\Delta^+_n(\mathcal E/p^n) \tensor_\ri \check\ri^+_X = \mathcal E/p^n$. Thus $\mathcal E = \varprojlim_n \mathcal E/p^n$ implies the result.
\end{proof}

\begin{corollary} \label{rslt:DW-+-is-equivalence}
The functor $\Delta^+$ is an equivalence of categories
\begin{align*}
	\Delta^+\colon \Modplus \longisoto \catlocsys\ri X.
\end{align*}
\end{corollary}
\begin{proof}
By \cref{rslt:DW-a-is-functor}, the functor $\Delta^+$ is fully faithful. If $\mathcal L$ is a local system in $\catlocsys\ri X  $, we find that  $\Delta^+( \mathcal L\tensor_{\ri} \check\ri^+_X) $ is isomorphic to $ \mathcal L$, which proves our claim. \end{proof}

Let us now define a similar functor for $\check\ri_X$-modules.

\begin{definition}\label{def:functor}
\begin{defenum}
	\item We say that a $\check\ri_X$-module $E$ has \emph{properly trivializable reduction modulo all $p^n$} if there exists a subsheaf $\mathcal E \subset E$ such that $E = \mathcal E[p^{-1}]$ and $\mathcal E $ in $\Modplus$. In this case $\mathcal E$ is called a \emph{model} of $E$. The category of $\check\ri_X$-modules with properly trivializable reduction is denoted $\Mod$.

	\item \label{def:DW} Let $E \in \Mod$. We define
	\begin{align*}
		\Delta(E) := \Delta^+(\mathcal E)[p^{-1}],
	\end{align*}
	where $\mathcal E$ is any model of $E$.
\end{defenum}
\end{definition}

\begin{proposition} \label{rslt:DW-is-functor-on-very-nice}
The assignment $E \to \Delta(E)$ provides a well-defined fully faithful functor from the category $\Mod$ to the category of $\Cpx_p$-local systems on $X$.
\end{proposition}
\begin{proof}
Let $E$ be an $\check\ri_X$-module with properly trivializable reduction modulo all $p^n$ and let $\mathcal E, \mathcal E' \subset E$ be two models. We have $\Hom_{\check\ri_{X}}(E, E) = \Hom_{\check\ri_{X^+}}(\mathcal E, \mathcal E')[p^{-1}]$, so that $p^m: E \to E$ for suitable $m$ is induced by a morphism $f: \mathcal E \rightarrow \mathcal E'$, which is compatible with the inclusions of both models in $E$ and in particular injective. By construction, $\Delta^+(f)$ is the limit of all $\Delta^+ (f / p^n): \Delta^+(\mathcal E /p^n) \rightarrow \Delta^+( \mathcal E' / p^n)$, which is the restriction of $f / p^n: \mathcal{E} / p^n \rightarrow \mathcal E'/ p^n$. In the limit we find that $\Delta ^+( \mathcal E) \subset \mathcal E$, and that $\Delta^+ (f) $ is the restriction of $f$. After inverting $p$ we conclude that the identity on $E$ corresponds to  an injection $\Delta^+(\mathcal E)[p^{-1}] \subset \Delta^+(\mathcal E')[p^{-1}]$. Switching roles of $\mathcal E$ and $\mathcal E'$ we similarly deduce an injection in the other direction, hence together we get $\Delta^+(\mathcal E)[p^{-1}] = \Delta^+(\mathcal E')[p^{-1}]$. Hence $\Delta(E)$ is indeed well-defined. Fully faithfulness follows from \cref{rslt:DW-a-is-functor}.
\end{proof}

By construction, the image of $\Delta$ is contained in the category $\catlocsysint{\Cpx_p}X$ defined in \cref{def:rep-integral-model}. Note that the results in part (i) and (ii) of the following theorem are also proven in \cite[Lemma 3.11 and Corollary 3.12]{matti-rigid-DW-functor} in slightly different terms.

\begin{theorem}\label{rslt:DW-for-OX}
\begin{thmenum}
	\item \label{rslt:DW-is-equivalence} The functor $\Delta$ is an equivalence
	\begin{align*}
		\Delta \colon \Mod \longisoto \catlocsysint{\Cpx_p}X,
	\end{align*}
	which is compatible with direct sums, tensor products, internal homs and exterior powers.
	
	\item \label{rslt:DW-partial-inverse} For every $E \in \Mod$ there is a natural isomorphism
	\begin{align*}
		\Delta(E) \tensor_{\Cpx_p} \check\ri_X = E.
	\end{align*}

	\item \label{rslt:DW-commutes-with-pullback} Let $f\colon Y \to X$ be a morphism of finite-type proper adic spaces over $\Cpx_p$ and let $E \in \Mod$. Then $f^*E \in \ModY$ and
	\begin{align*}
		\Delta(f^*E) = f^*\Delta(E).
	\end{align*}

	\item \label{rslt:descent-of-prop-triv-reduction} Let $f\colon Y \to X$ be a surjective morphism of finite-type proper adic spaces over $\Cpx_p$ and assume that $X$ is normal. Let $E$ be an $\check\ri_X$-module. If $f^*E \in \catmodproptriv{\check\ri_Y}$ then $E \in \catmodproptriv{\check\ri_X}$. In other words, the property of having properly trivializable reduction descends along $f$.
\end{thmenum}
\end{theorem}
\begin{proof}
The fact that $\Delta$ is an equivalence of categories and (ii) follow easily from \cref{rslt:DW-+-is-equivalence} and \cref{rslt:DW-a-is-functor}. It is straightforward to see that $\Mod$ is closed under direct sums, tensor products and exterior powers, and that $\Delta$ is compatible with these operations. 
Part (ii) also shows that all objects in $\Mod$ are locally free $\check\ri_X$-modules. Hence $\Delta$ is also compatible with internal homs. 

To prove (iii), we view $Y$ as an element of the v-site of $X$ to see that $f^*$ is just the restriction functor and in particular commutes with limits. Then $f^*E \in \ModY$ follows easily and the second part of the claim reduces to the statement that $f^* \Delta^+_n(\mathcal E_n) = \Delta^+_n(f^* \mathcal E_n)$ for a properly trivializable $\check\ri^+_X/p^n$-module $\mathcal E_n$. But this isomorphism follows directly from the uniqueness of the construction of $\Delta^+_n$.

It remains to prove (iv), so let $E$ be given. Let $L_Y$ be the $\Cpx_p$-local system on $Y$ which corresponds to $f^* E$ under the functor $\Delta$ by (i). The glueing datum of $f^*E$ along the v-cover $Y \to X$ is given by an isomorphism $\varphi\colon \pr1^* f^*E \isoto \pr2^* f^*E$, where $\pr1, \pr2\colon Y \cprod_X Y \to Y$ are the two projections. Applying the functor $\Delta$ to $\varphi$ and using (iii) produces an isomorphism
\begin{align*}
	\Delta(\varphi)\colon &\pr1^* L_Y = \pr1^* \Delta(f^*E) = \Delta(\pr1^*f^*E) \\& \qquad\isoto \Delta(\pr2^*f^*E) = \pr2^* \Delta(f^*E) = \pr2^* L_Y.
\end{align*}
The functoriality of $\Delta$ (and commutation with pullback) implies that the cocycle condition for $\Delta(\varphi)$ follows from the cocycle condition for $\varphi$. Hence we can glue $L_Y$ along $f$ to get a $\Cpx_p$-local system $L \subset E$ on $X$. By definition of $\Delta$, $L_Y$ admits an integral model, so by \cref{rslt:descent-of-integral-model-along-morphism-of-proper-spaces-with-loc-irr-target} the same is true for $L$. On the other hand it is obviously true that the natural map $L \tensor_{\Cpx_p} \check\ri_X \to E$ is an isomorphism, because this can be checked v-locally. By (i) we conclude $E \in \Mod$ and $\Delta(E) = L$.
\end{proof}

\begin{corollary}\label{rslt:Mod-is-abelian}
The category $\Mod$ is abelian.
\end{corollary}
\begin{proof}
We may assume that $X$ is connected. Then by \cref{rslt:ILoc-is-abelian} the category $\catlocsysint{\Cpx_p}X$ is abelian, hence by \cref{rslt:DW-is-equivalence} the same is true for $\Mod$.
\end{proof}

\section{Vector Bundles on Adic Spaces}

In this section we assume that $X$ is a proper, connected, seminormal (see e.g. \cite[section 1.4]{relative-p-adic-hodge-2}) adic space of finite type over $\mathbb{C}_p$.
Then $X$ is induced by a proper rigid analytic variety over $\mathbb{C}_p$.

We study vector bundles on $X$, i.e. locally free $\ri_X$-modules for the analytic topology. We use the natural morphism of ringed sites $\tau: (X_\vsite, \check\ri_{X}) \rightarrow (X_{an}, \ri_{X})$ induced by composing  $\lambda\colon (X_\vsite, \check\ri_X) \to (X_\qproet, \hat\ri_X)
 $ with $ \nu \colon (X_\qproet, \hat \ri_X) \to (X_\et, \ri_{X})$ and with the natural projection from the étale to the analytic site on $X$. As in \cite[Lemma 7.3]{rigid-p-adic-hodge} pullback with respect to this projection is fully faithful.  By \cref{rslt:locfree-sheaf-fully-faithful-qproet-to-v} and \cref{rslt:locfree-sheaf-fully-faithful-et-to-qproet} the functor $\tau^\ast$ is therefore fully faithful on vector bundles on $X$. In a similar way, we define $\tau^+: (X_\vsite, \check\ri_{X}^+) \rightarrow (X_{an}, \ri_{X}^+)$.

\begin{definition}
\begin{defenum}
	\item We define $\catbdl{X}$ as the full subcategory of analytic vector bundles (i.e. locally free $\ri_ {X}$-sheaves) $E$ on $X$ such that $\tau^{\ast} E$ is an object of $\Mod$.
	\item For every $E$ in $\catbdl{X}$ we define a local system in $\catlocsysint{\Cpx_p}X$ by setting $\Delta (E) = \Delta(\tau^\ast E)$.
\end{defenum}
\end{definition}

\begin{proposition}\label{rslt:parallel-transport-for-adic}
The functor $\Delta: \catbdl{X} \longto \catlocsysint{\Cpx_p}X$ is fully faithful, exact and compatible with direct sums, tensor products, internal homs and exterior powers.  Moreover, for every vector bundle $E$ in $\catbdl{X} $ and every morphism of finite-type proper adic spaces $f: Y \longto X$ over $\mathbb{C}_p$ the pullback bundle $f^\ast E$ lies in $\catbdl{Y} $,  and we have $\Delta(f^\ast E) = f^\ast \Delta(E)$.
\end{proposition}
\begin{proof}
This follows from \cref{rslt:DW-for-OX}.
\end{proof}

Hence the category $\catbdl{X}$ of analytic vector bundles giving rise to local system has the important property that the required reduction behaviour can be  checked on a proper cover.

\begin{theorem}\label{rslt:proper-pullback}
Let $f\colon Y \to X$ be a surjective morphism of finite-type proper, connected seminormal adic spaces over $\Cpx_p$ and assume in addition that $X$ is normal. Let $E$ be a vector bundle on $X$. If $f^*E \in \catbdl{Y}$ then $E \in \catbdl{X}$.
\end{theorem}
\begin{proof}
This follows from \cref{rslt:descent-of-prop-triv-reduction}
\end{proof}

Note that we do not demand the existence of a suitable $\ri_X^+$-lattice for bundles in $\catbdl{X}$. If we add this condition, we only have to check the trivializability condition modulo $p$, as the following result based on \cite{matti-rigid-DW-functor} shows.

\begin{definition}
We define $\mathcal{B}_{\ri_X^+}(X)$ as the full subcategory of analytic vector bundles $E$ on $X$, such that there exists a locally free $\ri^+_X$-module $\mathcal E$ satisfying $E = \mathcal E[1/p]$ with $\tau^{+\ast} \mathcal E$ in $\Modplus$.
\end{definition}

Obviously $\mathcal{B}_{\ri_X^+}(X)$ is a full subcategory of $\catbdl{X}$. The following proposition allows for an easier description of this category.

\begin{proposition} \label{rslt:matti-trivializable-mod-pn-from-p} 
Let $\mathcal E$ be a locally free  $\ri^+_X$-module of finite rank on the adic space $X$ (which satisfies our general assumptions). Then the following are equivalent:
\begin{enumerate}[(a)]
	\item $\tau^{+\ast} \mathcal E$ is an object of $\Modplus$, i.e. $\tau^{+\ast} \mathcal E/p^n$ is a properly trivializable $\check\ri^+_X/p^n$-module for all $n \ge 1$.
	\item $\tau^{+\ast}\mathcal E/p$ is a properly trivializable $\check\ri^+_X/p$-module.
	\item $\mathcal E /p$ is trivializable on a finite étale cover of $X$.
\end{enumerate}
\end{proposition}
\begin{proof}
Assume that $\mathcal E$ satisfies (b) and let $\alpha\colon U \to X$ be a proper cover such that  $\alpha^{+*}\tau^{+\ast} \mathcal E/p$ is trivial. By \cref{rslt:DW-n-is-functor}, we may choose $U$ to be finite étale over $X$, hence (b) is equivalent to (c).
Moreover, by the proof of \cite[Theorem 3.19]{matti-rigid-DW-functor} we see that for all $n \ge 1$, the $\check\ri^+_Y/p^n$-module $\alpha^{+*}\mathcal E/p^n$ is properly trivializable. But then the same is true for $\tau^{+ \ast}\mathcal E/p^n$ on $X$.
\end{proof}

Note \cite[Proposition 2.10]{matti-rigid-DW-functor} that $\ri$-local systems on $X_\qproet$ correspond to continuous representations of $\pi_1^{\et}(X,{x})$ on  free $\ri$-modules of finite rank, where the association is given by taking the fiber over ${x}$. In \cite[Section 4.3]{matti-rigid-DW-functor} it is shown how to define parallel transport on suitable vector bundles in a similar way.  Let us adapt this construction to our situation.

We denote by $\Pi_1(X)$ the algebraic fundamental groupoid of $X$, i.e. the category with objects $X(\mathrm{Spa}(\mathbb{C}_p, \ri))$, such that the set of morphisms from $x$ to $y$ is given by the set of isomorphisms $\mathrm{Iso}(F_x, F_y)$, where $F_x$ is the algebraic fiber functor, mapping a finite \'etale cover $\alpha: Y \rightarrow X$ to the set $\alpha^{-1}(x)$. We endow $\mathrm{Iso}(F_x, F_y)$ with the profinite topology, thus making $\Pi_1(X)$ into a topological groupoid. We refer to elements in $\mathrm{Iso}(F_x, F_y)$ as \'etale paths from $x$ to $y$.

Let $\mathcal L$  be an element of $\catlocsys\ri{X}$, i.e. a local system for the quasi-pro-\'etale topoloy. For each $x \in X(\mathrm{Spa}(\mathbb{C}_p, \ri))$ the fiber $x^{-1} \mathcal L$  of $\mathcal L$ in $x$  is a constant local system on $\mathrm{Spa}(\mathbb{C}_p, \ri)$. Put $\mathcal L_x = \Gamma(\mathrm{Spa}(\mathbb{C}_p, \ri),  x^{-1} \mathcal L)$. 
By \cref{rslt:profinet-cover} we find a connected quasi-pro-\'etale cover $f: Y = \varprojlim_n Y_n$ of $X^\diamond$ trivializing $\mathcal L$, such that each $Y_n$ is the diamond of a connected finite  \'etale cover of $X$. 
Let $\gamma$ be an \'etale path from $x$ to $y$. Choose  a lift $x'$  of $x$  to a morphism (of diamonds) $\mathrm{Spd}(\mathbb{C}_p, \ri) \rightarrow Y$ by lifting $x$ successively to all $Y_n$. Then the \'etale path $\gamma$ induces a lift $y' = \gamma(x'): \mathrm{Spd}(\mathbb{C}_p, \ri) \rightarrow Y$ of $y$. There is a natural isomorphism $(\mathcal L|_{Y})_{x'} \isoto \mathcal L_x$. Since $\mathcal L|_{Y}$ is trivial as a local system on $Y$, we find that the natural morphism $H^0(Y, \mathcal L|_{Y}) \rightarrow (\mathcal L)_{x'}$ is an isomorphism of $\ri$-modules. We define an isomorphism of $\ri$-modules associated to $\gamma$ as
\[\rho_{\mathcal L}(\gamma): \mathcal L_x \isofrom (\mathcal L|_{Y})_{x'} \isofrom H^0(Y, \mathcal L|_{Y}) \isoto  (\mathcal L|_{Y})_{y'} \isoto \mathcal L_y.\]

\begin{proposition} \label{rslt:parallel-transport-localsys}
\begin{enumerate}[(a)]
	\item The isomorphism $\rho_{\mathcal L}(\gamma): \mathcal L_x \isoto \mathcal L_y$ is independent of all choices made in the construction.
	\item The association $x \mapsto \mathcal L_x$ together with $\gamma \mapsto \rho_{\mathcal L}(\gamma)$ induces for every $\mathcal L $ in $\catlocsys\ri{X}$ a continuous functor $\rho_{\mathcal L}$ from the fundamental groupoid $\Pi_1(X)$ to the category of finite free $\ri$-modules.
	The association $\mathcal L \mapsto \rho_{\mathcal L}$ is functorial and compatible with direct sums, tensor products, internal homs and exterior powers in an evident way. In particular, every local system $\mathcal L$ in $\catlocsys\ri{X}$ gives rise to a continuous representation
	 $\rho_{\mathcal L} $ of the \'etale fundamental group $\pi_1^{\et}(X,{x})$  on a finite free $\ri$-module.
\end{enumerate}
\end{proposition}
\begin{proof}
This is straightforward, using similar arguments as in \cite[Section 3.1]{matti-rigid-DW-functor}.\end{proof}

\begin{corollary} \label{rslt:comparison-to-Matti}
Every vector bundle $E$ in $\catbdl{X}$ gives rise to a continuous functor $\rho_E$ from the fundamental groupoid $\Pi_1(X)$ to the category of finite free $\mathbb{C}_p$-vector spaces. If $E$ is contained in $\mathcal{B}_{\ri_X^+}(X)$, our construction coincides with the one given by W\"urthen in \cite[Section 4.3]{matti-rigid-DW-functor}.
\end{corollary}
\begin{proof} Since $\Delta: \Mod \to \catlocsysint{\Cpx_p}X$ is a category equivalence by  \cref{rslt:DW-is-equivalence}, we can define $\rho_E = \rho_{\mathcal L} \otimes \mathbb{C}_p$ for $\Delta(E)  = \mathcal L [p^{-1}]$.   Compatibility with the construction in \cite[Section 4.3]{matti-rigid-DW-functor}  for bundles in $\mathcal{B}_{\ri_X^+}(X)$ is then evident. 
\end{proof}

\begin{theorem}\label{rslt:category-abelian}
The category $\catbdl{X}$ is abelian.
\end{theorem}
\begin{proof}
We may assume that $X$ is connected.  Let $\varphi: E' \rightarrow E$ be a morphism in $  \catbdl{X}$. We have to check that $\mathrm{ker}(\varphi)$ and $\mathrm{coker}(\varphi)$ are also objects in $  \catbdl{X}$.

Let us first look at the coherent $\ri_X$-module  $E'' = \mathrm{coker}(\varphi)$. Since the category $\Mod$ is abelian by \cref{rslt:Mod-is-abelian}, and $\tau^\ast$ is exact on locally free sheaves, it suffices to show that $E''$ is locally free on $X$.
For every classical rank-1-point $x \in X(\mathrm{Spa}(\mathbb{C}_p, \ri))$
we consider the fiber $E''_x = \Gamma(\mathrm{Spa}(\mathbb{C}_p, \ri), x^\ast E'')$, which is a finite-dimensional $\mathbb{C}_p$-vector space,
since $E_x \rightarrow E''_x $ is surjective. It suffices to show that the dimension of  $E''_x$ is constant for all $x \in X(\mathrm{Spa}(\mathbb{C}_p, \ri))$. If this is the case, we can argue as in \cite[Section 5, Lemma 1]{mumford-av} to prove that $E''$ is indeed a locally free $\ri_X$-module:
For every $x \in X(\mathrm{Spa}(\mathbb{C}_p, \ri))$ we lift a $\mathbb{C}_p$-basis of $E''_x$ to sections $\sigma_1, \ldots, \sigma_r \in E''(U)$ on a suitable open neighbourhood $U$ of $x$. These sections induce a morphism of $\ri_X|_U$-modules $\alpha: \ri_X^r|_U \rightarrow E''|_U$, which is surjective in the fiber over $x$. Hence by Nakayama's lemma, $\alpha$ is surjective in a neighbourhood of $x$, which we  may assume to be $U$. For every classical rank-1-point $z$ in $U(\mathrm{Spa}(\mathbb{C}_p, \ri))$, the fiber map $\alpha(z)$ is therefore an epimorphism of  $\mathbb{C}_p$-vector spaces of the same dimension (by our assumption) and hence an isomorphism.   If $E_0$ denotes the kernel of $\alpha$, we find that the fiber of $E_0$ on every classical rank-1-point $z$ in $U(\mathrm{Spa}(\mathbb{C}_p, \ri))$ vanishes.  Therefore the stalk of $E_0$ in $z$ is contained in $\mathfrak{m}_z \ri_{X,z}^r$. Since locally  $X$ is given by $\mathrm{Spa}(A, A^+)$ for a reduced Jacobson ring $A$ (cf. \cite[Proposition 3.3]{joao-master-thesis}), we deduce that  $E_0 = 0$. Hence $E''$ is indeed locally free of finite rank.

Let us now prove that the dimension of $E''_x$ is constant. By \cref{rslt:profinet-cover} and the category equivalence \cref{rslt:DW-is-equivalence}, we find a connected quasi-pro-\'etale cover $Y \rightarrow X$ trivializing $\tau^\ast E$,  $\tau^\ast E'$ and the cokernel of $\tau^\ast \varphi$, which we denote by $F$. Note that $F$ is also an object of $\Mod$ by \cref{rslt:Mod-is-abelian}.
Put $S = \mathrm{Spa}(\mathbb{C}_p, \ri)$, and let  $x \in X(S)$ be a classical rank-1-point in $X$. 
Note that we can lift $x$ to a point $y$ in $Y(S)$. Since $\tau^\ast E$, $\tau^\ast E'$ and $F$ have constant rank on $Y$,  we find that 
 $E''_x$, which is the  cokernel of the fiber map $E'_x \rightarrow E_x$, can be identified with $\Gamma(S, x^{\ast} F) = \Gamma(S, y^\ast (F|_Y)$, and thus has constant $\mathbb{C}_p$-rank on $X(S)$.

Now let us look at the kernel of $\varphi: E' \rightarrow E$. Since the cokernel of $\varphi$ is contained in $\catbdl{X}$, we may assume that $\varphi$ is surjective. Since $E$ is a locally free $\ri_X$-module, dualizing  $0 \rightarrow \mathrm{ker}(\varphi) \rightarrow E' \rightarrow E \rightarrow 0$ leads to an exact sequence $ 0 \rightarrow E^*  \rightarrow (E')^\ast \rightarrow \mathrm{ker}(\varphi)^\ast \rightarrow 0$, which implies that $\mathrm{ker}(\varphi)^\ast$ lies in $\catbdl{X}$, and therefore also the double dual $\mathrm{ker}(\varphi)^{\ast \ast} $.
Now we have a commutative diagram with exact lines
\begin{center}\begin{tikzcd}
0 \arrow[r]	 & \mathrm{ker}(\varphi) \arrow[r] \arrow[d]  & E' \arrow[r] \arrow[d,"\rotatebox{90}{\(\sim\)}"] & E \arrow[r] \arrow[d,"\rotatebox{90}{\(\sim\)}"] & 0\\
0 \arrow[r]	  & \mathrm{ker}(\varphi)^{\ast \ast} \arrow[r]  & (E')^{\ast \ast}  \arrow[r] & E^{\ast \ast}  \arrow[r] &0
\end{tikzcd}\end{center}
which implies that $\mathrm{ker}(\varphi)$  is isomorphic to its double dual and therefore in $\catbdl{X}$.
\end{proof}

\section{Application: The Algebraic Case}
\label{section:algebraic-case}

In this section we want to relate the results of the present paper to previous results by Deninger and the second author. In \cite{deninger-werner-parallel-transport-varieties}, a theory of $p$-adic parallel transport for vector bundles with numerically flat reduction on a smooth, proper variety $V$ over $\overline{\mathbb{Q}}_p$ is developed. This generalizes previous results on curves (\cite{deninger-werner-parallel-transport-1}, \cite{deninger-werner-parallel-transport-2}, \cite{deninger-werner-Tannaka}). Recall that a vector bundle $E$ is called numerically flat, if $E$ and its dual are nef.

For every normal, proper variety $V$ over $\overline{\mathbb{Q}}_p$ let  $\mathcal{B}^s_{V_{\mathbb{C}_p}}$ denote the (full) category of all vector bundles $E$ on $V_{\mathbb{C}_p}$ such that

\begin{itemize}
	\item there exists a flat, proper scheme $\mathcal{V}$ of finite presentation over ${\overline{\mathbb{Z}}_p}$ with generic fiber $V$, and

	\item there exists a vector bundle $\mathcal{E}$ on $\mathcal V \otimes \mathfrak{o}_p$ with generic fiber isomorphic to $E$ such that
	the special fiber $\mathcal{E}_k = \mathcal{E} \otimes k$ is numerically flat on $\mathcal V_k$, where $k$ is the residue field of $\mathbb{C}_p$.
\end{itemize}

The main theorem of \cite{deninger-werner-parallel-transport-varieties} defines (if $V$ is smooth) for every vector bundle $E$ in $\mathcal{B}^s_{V_{\mathbb{C}_p}}$ a continuous functor
\begin{equation} \label{eqn:DW} \tag{DW}
	\rho_E: \Pi_1(V_{\mathbb{C}_p}) \rightarrow \mathrm{Vec_{\mathbb{C}_p}}
\end{equation}
from the étale fundamental groupoid of $V_{\mathbb{C}_p}$ to the category of finite-dimensional ${\mathbb{C}_p}$-vector spaces which is the fiber map $x \mapsto E_x$ on objects, and which is continuous on morphisms.
The association $E \mapsto \rho_E$ is functorial and respects  direct sums, tensor products,  internal homs and exterior powers. Moreover it is functorial with respect to morphisms of varieties, see \cite[Proposition 9.4 and Theorem 9.5]{deninger-werner-parallel-transport-varieties}.

In order to show that these results fit into the framework of the present paper,  denote by  $X$ the adic space associated to $V_{\mathbb{C}_p}$.
Every algebraic vector bundle $E$ on $V$ gives rise to a vector bundle $E^{ad}$ on the analytic site $X_{an} $.

We will now generalize  the main result of  \cite{deninger-werner-parallel-transport-varieties} to vector bundles which have numerically flat reduction only after pullback to a proper cover. (Note that the relaxation on the smoothmess hypothesis was already established in \cite{matti-rigid-DW-functor}.) This provides a vast generalization of the previous \cite[Theorem 10]{deninger-werner-parallel-transport-2} which deals with finite covers of curves.

\begin{theorem}
\label{rslt:algebraic-case}
Let $E$ be a vector bundle on $V_{\mathbb{C}_p}$ such that there exists a normal proper variety $W$ over $\overline{\mathbb{Q}}_p$ and a  surjective $\overline{\mathbb{Q}}_p$-morphism $f: W \rightarrow V$ such that $f^\ast E$ lies in the category $\mathcal{B}^s_{W_{\mathbb{C}_p}}$, i.e. $f^\ast E$  has numerically flat reduction on $W$.

Then there exists a continuous functor  $\rho_E: \Pi_1(V_{\mathbb{C}_p}) \rightarrow \mathrm{Vec_{\mathbb{C}_p}}$ such that the identification of fibers $(f^\ast E)_x \simeq E_{f(x)}$ induces a natural isomorphism $\rho_{f^\ast E} \simeq \rho_E \circ f_\ast$. The association $E \mapsto \rho_E$ is functorial and respects direct sums, tensor products, internal homs and exterior powers. Moreover it is functorial with respect to morphisms of varieties and coincides with the functor (\ref{eqn:DW}) if $E$ lies in $\mathcal{B}^s_{V_{\mathbb{C}_p}}$ and $V$ is smooth.
\end{theorem}
\begin{proof}
 If $E$ is contained in the category $\mathcal{B}^s_{V_{\mathbb{C}_p}}$ the associated adic vector bundle $E^{ad}$ is an object of $\mathcal{B}_{\ri_X^+}(X)$. By \cref{rslt:comparison-to-Matti} combined with \cite[Theorem 4.28]{matti-rigid-DW-functor}, the parallel transport functor $\rho_E$ from (\ref{eqn:DW}) is induced by the functor $\rho_{E^{ad}}$ defined in  \cref{rslt:comparison-to-Matti} via precomposition with the natural analytification functor $\Pi_1(V_{\mathbb{C}_p}) \rightarrow \Pi_1(X)$.
 Let $Y$ be the adic space associated to $W_{\mathbb{C}_p}$, and write by abuse of notation $f = f^{ad}: Y \rightarrow X$ for the induced adic morphism. If $f^\ast E$ is an object of $\mathcal{B}^s_{W_{\mathbb{C}_p}}$, the associated adic vector bundle $f^\ast E^{ad}$ on $X$ is an object of $\catbdl{Y}$. Hence $E^{ad}$ lies in $\catbdl{X}$ by \cref{rslt:proper-pullback}, and gives thus rise to a continuous functor of parallel transport $\rho_{E^{ad}}: \Pi_1(X) \rightarrow \mathrm{Vec_{\mathbb{C}_p}}$. We define $\rho_E$ again by precomposition with the natural analytification functor $\Pi_1(V_{\mathbb{C}_p}) \rightarrow \Pi_1(X)$. The rest of our claim follows from \cref{rslt:parallel-transport-localsys} and  \cref{rslt:parallel-transport-for-adic}.
\end{proof}

Note that it is not clear that the category of bundles $E$ on $V_{\mathbb{C}_p}$ such that a proper pullback as in the previous theorem lies in the category $\mathcal{B}^s_{W_{\mathbb{C}_p}}$ lies in the algebraic category $\mathcal{B}^s_{V_{\mathbb{C}_p}}$, since we do not know if $E$ admits an algebraic model with numerically flat reduction.


\bibliography{bibliography}
\addcontentsline{toc}{section}{References}

\end{document}